%% file: Space_Varying_Escande_Weiss_2015.tex
\documentclass[11pt]{article}

\pdfoutput=1
\input{preamble}

\begin{document}
\input{Body/1-frontmatter.tex}

\input{Body/5-intro.tex}
\input{Body/6-notation.tex}
\input{Body/7-model.tex}
\input{Body/20-sparse.tex}
\input{Body/25-algorithms.tex}

\input{Body/30-numerical.tex}
\input{Body/40-conclusion.tex}
\input{Body/50-acknowledgements.tex}

\bibliographystyle{abbrv}
\bibliography{biblio}

\appendix
\input{Body/109-proof_decay.tex}
\input{Body/116-proof_thresh_d_dim.tex}
 
\end{document}

%% file: preamble.tex
\usepackage[utf8]{inputenc}

\usepackage[english]{babel}

\usepackage{geometry}

\usepackage{graphicx}
\graphicspath{{./images/}{./}}

\usepackage{caption}
\usepackage{subcaption}

\usepackage{url}
\usepackage{placeins}

\usepackage{longtable} 
\usepackage{booktabs} 

\usepackage{multirow}
\usepackage{rotating}

\usepackage{amsthm}
\usepackage{amsmath}\allowdisplaybreaks
\usepackage{amssymb}
\usepackage{mathrsfs}
\usepackage{dsfont}
\usepackage{amsbsy}
\usepackage{enumitem}

\usepackage[ruled,vlined]{algorithm2e}

\usepackage[usenames,dvipsnames,svgnames,table]{xcolor}

\usepackage{xspace}

\usepackage{pgfplots}
\pgfplotsset{compat=newest}
\usetikzlibrary{plotmarks}
\usepackage{grffile}

\newtheorem{theorem}{Theorem}
\newtheorem{lemma}[theorem]{Lemma}

\theoremstyle{definition}
\newtheorem{definition}{Definition}
\newtheorem{remark}{Remark}

\newcommand{\R}{ \mathbb{R} }
\newcommand{\Z}{ \mathbb{Z} }
\newcommand{\N}{ \mathbb{N} }

\newcommand{\LL}{ \mathbb{L} }

\def \bH{\mathbf{H}}
\def \bS{\mathbf{S}}
\def \bTheta{\mathbf{\Theta}}
\def \bDelta{\mathbf{\Delta}}
\def \bSigma{\mathbf{\Sigma}}

\def \bPsi{\mathbf{\Psi}}
\def \bA{\mathbf{A}}

\def \bN{\pmb{\mathcal{N}}}
\def \bU{\mathbf{U}}
\def \bu{\mathbf{u}}
\def \bv{\mathbf{v}}
\def \bW{\mathbf{W}}

\def \bpsi{\boldsymbol\psi}

\def \bbeta{\boldsymbol\eta}

\def \bz{\mathbf{z}}
\def \bO{\mathbf{O}}

\def \mN{\mathcal{N}}

\def \balpha{\alpha}
\def \be{e}
\def \bzero{0}
\def \bm{m}

\newcommand{\norm}[1]{\left\Vert #1\right\Vert}
\newcommand{\abs}[1]{\left\vert #1\right\vert}
\newcommand{\opnorm}[1]{{\left\vert\kern-0.25ex\left\vert\kern-0.25ex\left\vert #1 \right\vert\kern-0.25ex\right\vert\kern-0.25ex\right\vert}}
\newcommand{\dotproduct}[2]{\ensuremath{\left\langle #1, #2 \right\rangle}\xspace}
\newcommand{\p}{\partial}

\newcommand{\argmin}{\mathop{\mathrm{arg\,min}}}
\newcommand{\argmax}{\mathop{\mathrm{arg\,max}}}
\newcommand{\supp}{\mathop{\mathrm{supp}}}

\newcommand{\set}[1]{\left\{ #1 \right\}}
\newcommand{\wtilde}[1]{\widetilde{#1}}
\DeclareMathOperator*{\esssup}{ess\,sup}

\newcommand{\dist}[2]{\mathop{\mathrm{dist}} \left( #1, #2 \right)}

\font\dsrom=dsrom10 scaled 1200
\def \indic{\textrm{\dsrom{1}}}
\newcommand{\fa}{ \textrm{for all }}

\providecommand{\SetAlgoLined}{\SetLine}
\providecommand{\DontPrintSemicolon}{\dontprintsemicolon}

\usepackage{hyperref}
\usepackage[squaren,Gray]{SIunits}

%% file: Body/1-frontmatter.tex
\title{Sparse Wavelet Representations of Spatially Varying Blurring Operators.}

\author{Paul Escande \footnotemark[2] \and Pierre Weiss \footnotemark[3] }

\maketitle

\renewcommand{\thefootnote}{\fnsymbol{footnote}}
\footnotetext[2]{D\'epartement Math\'ematiques, Informatique, Automatique (DMIA), Institut Sup\'erieur de l'A\'eronautique et de l'Espace (ISAE), Toulouse, France,  \url{paul.escande@gmail.com} }
\footnotetext[3]{Institut des Technologies Avanc\'ees en Sciences du Vivant, ITAV-USR3505 and Institut de Math\'ematiques de Toulouse, IMT-UMR5219, CNRS and universit\'e de Toulouse, Toulouse, France, \url{pierre.armand.weiss@gmail.com}}

\begin{abstract}
Restoring images degraded by spatially varying blur is a problem encountered in many disciplines such as astrophysics, computer vision or biomedical imaging. 
One of the main challenges to perform this task is to design efficient numerical algorithms to approximate integral operators.


We introduce a new method based on a sparse approximation of the blurring operator in the wavelet domain. This method requires $\mathcal{O}\left(N \epsilon^{-d/M}\right)$ operations to provide $\epsilon$-approximations, where $N$ is the number of pixels of a $d$-dimensional image and $M\geq 1$ is a scalar describing the regularity of the blur kernel. In addition, we propose original methods to define sparsity patterns when only the operators regularity is known.

Numerical experiments reveal that our algorithm provides a significant improvement compared to standard methods based on windowed convolutions.
\end{abstract}



%% file: Body/5-intro.tex
\section{Introduction}

The problem of image restoration in the presence of spatially varying blur appears in many domains.
Examples of applications in computer vision, biomedical imaging and astronomy are shown in Figures \ref{fig:dwarves_motion} and \ref{fig:billes} respectively. 
In this paper, we propose new solutions to address one of the main difficulties associated to this problem: the computational evaluation of matrix-vector products.

A spatially variant blurring operator can be modelled as a linear operator and therefore be represented by a matrix $\bH$ of size $N\times N$, where $N$ represents the number of pixels of a $d$-dimensional image. 
Sizes of typical images range from $N=10^6$ for small 2D images, to $N=10^{10}$ for large 2D or 3D images. 
Storing matrices and computing matrix-vector products using the standard representation is impossible for such sizes: it amounts to tera or exabytes of data/operations. 
In cases where the Point Spread Functions (PSF) supports are sufficiently small in average over the image domain, the operator can be coded as a sparse matrix and be applied using traditional approaches. 
However, in many practical applications this method turns out to be too intensive and cannot be applied with decent computing times. 
This may be due to i) large PSFs supports or ii) the need for super-resolution applications where the PSFs sizes increase with the resolution.
Spatially varying blurring matrices therefore require the development of computational tools to compress them and evaluate them in an efficient way.
 

\subsection*{Existing approaches}

To the best of our knowledge, the first attempts to address this issue appeared at the beginning of the seventies (see e.g. \cite{sawchuk_space-variant_1972}). Since then, many techniques were proposed. We describe them briefly below

\paragraph{Composition of diffeomorphisms and convolutions}
One of the first method proposed to reduce the computational complexity, is based on first applying a diffeomorphism to the image domain \cite{sawchuk_space-variant_1972, sawchuk_space-variant_1974,mcnown1994approximate,tabernero1999duality,estatico2013shift} followed by a convolution using FFTs and an inverse diffeomorphism. The diffeomorphism is chosen in order to transform the spatially varying blur into an invariant one. This approach suffers from two important drawbacks:
\begin{itemize}
\item first it was shown that not all spatially varying kernel can be approximated by this approach \cite{mcnown1994approximate}, 
\item second, this method requires good interpolation methods and the use of Euclidean grids with small grid size in order to correctly estimate integrals.
\end{itemize}

\paragraph{Separable approximations}
Another common idea is to approximate the kernel of the operator by a separable one that operates in only one dimension. The computational complexity of a product is thus reduced to $d$ applications of one-dimensional operators. It drastically improves the performance of algorithms. For instance, in 3D fluorescence microscopy, the authors of \cite{preza2004depth,maalouf2011fluorescence,hadj2012space, zhang2007gaussian} proposed to approximate PSFs by anisotropic Gaussians and assumed that the Gaussian variances only vary along one direction (e.g., the direction of light propagation). The separability assumption implies that both the PSF and its variations are separable. Unfortunately, most physically realistic PSFs are not separable and do not vary in a separable manner (see e.g., Figure \ref{fig:nonSeparablePSF}). This method is therefore usually too crude.

\paragraph{Wavelet or Gabor multipliers}
Some works \cite{CMY-Foveation, escande2012spatially,feichtinger2003first,hrycak2011practical} proposed to approximate blurring operators $\bH$ using operators diagonal in wavelet bases, wavelet packet or Gabor frames.  
This idea consists of defining an approximation $\wtilde \bH$ of kind $\wtilde \bH = \bPsi \bSigma \bPsi^{*}$, where $\bPsi^{*}$ and $\bPsi$ are wavelet or Gabor transforms and $\bSigma$ is a diagonal matrix.
These diagonal approximations mimic the fact that shift-invariant operators are diagonal in the Fourier domain. 
These approaches lead to fast $\mathcal{O}(N)$ or $\mathcal{O}(N\log(N))$ algorithms to compute matrix-vector products. 
In \cite{escande2012spatially}, we proposed to deblur images using diagonal approximations of the blurring operators in redundant wavelet packet bases. This approximation was shown to be fast and efficient in deblurring images when the exact operator was scarcely known or in high noise levels. It is however too coarse for applications with low noise levels. This approach seems however promising. Gabor multipliers are considered the state-of-the-art for 1D signals in ODFM systems for instance (slowly varying smoothing operators).

\paragraph{Weighted convolutions}
Probably the most commonly used approaches consist of approximating the integral kernel by spatially weighted sum of convolutions. Among these approaches two different ideas have been explored. The first one will be called \emph{windowed convolutions} in this paper and appeared in \cite{nagy1997fast,nagy1998restoring, hansen2006deblurring, hirsch2010efficient,denis2014fast}. The second one was proposed in \cite{rigault2005anisoplanatic} and consists of expanding the PSFs in a common basis of small dimensionality.

Windowed convolutions consists of locally stationary approximations of the kernel. We advise the reading of \cite{denis2014fast} for an up-to-date description of this approach and its numerous refinements. The main idea is to decompose the image domain into subregions and perform a convolution on each subregion. 
The results are then gathered together to obtain the blurred image. 
In its simplest form, this approach consists in partitioning the domain $\Omega$ in squares of equal sizes. 
More advanced strategies consist in decomposing the domain with overlapping subregions. 
The blurred image can then be obtained by using windowing functions that interpolate the kernel between subregions (see, e.g., \cite{nagy1997fast,hirsch2010efficient,denis2014fast}). Various methods have been proposed to interpolate the PSF. In \cite{hirsch2010efficient}, a linear interpolation is performed, and in \cite{denis2014fast} higher order interpolation of the PSF are handled.

\paragraph{Sparse wavelet approximations}
The approach studied in this paper was proposed recently and independently in \cite{wei2014fast, wei2009fast,escande2013image}. The main idea is to represent the operator in the wavelet domain by using a change of basis. This change of basis, followed by a thresholding operation allows sparsifying the operator and use sparse matrix-vector products. The main objective of this work is to provide solid theoretical foundations to these approaches.

\subsection{Contributions of the paper}

Our first contribution is the design of a new approach based on sparse approximation of $\bH$ in the wavelet domain.
Using techniques initially developed for pseudo-differential operators \cite{BCR,meyer1992wavelets}, we show that approximations $\widetilde \bH$ satisfying $\|\bH - \widetilde \bH\|_{2\to 2}\leq \epsilon$, can be obtained with this new technique, in no more than $\displaystyle \mathcal{O}\left(N \epsilon^{-d/M} \right)$ operations. In this complexity bound, $M\geq 1$ is an integer that describes the smoothness of the blur kernel. 

Controlling the spectral norm is usually of little relevance in image processing. 
Our second contribution is the design of algorithms that iteratively construct sparse matrix patterns adapted to the structure of images. 
These algorithms rely on the fact that both natural images and operators can be compressed simultaneously in the same wavelet basis.

As a third contribution, we propose an algorithm to design a generic sparsity structure when only the operators regularity is known. This paves the way to the use of wavelet based approaches in blind deblurring problems where operators need to be inferred from the data.

We finish the paper by numerical experiments. 
We show that the proposed algorithms allow significant speed ups compared to some windowed convolutions based methods.

Let us emphasize that the present paper is a continuation of our recent contribution \cite{escande2013image}. 
The main evolution is that i) we provide all the theoretical foundations of the approach with precise hypotheses, ii) we propose a method to automatically generate adequate sparsity patterns and iii) we conduct a thorough numerical analysis of the method.

\subsection{Outline of the paper}

The outline of this paper is as follows. 
We introduce the notation used throughout the paper in Section \ref{sec:notation}. 
We propose an original mathematical description of blurring operators appearing in image processing in Section \ref{sec:model}.
We introduce the proposed method and analyze its theoretical efficiency Section \ref{sec:sparserep}.
We then propose various algorithms to design good sparsity patterns in Section \ref{sec:algorithms}.
Finally, we perform numerical tests to analyze the proposed method and compare it to the standard windowed convolutions based methods in Section \ref{sec:numerics}. 

\begin{figure}[htp]
\centering
\begin{subfigure}[b]{0.4\textwidth}
	\includegraphics[width=\textwidth]{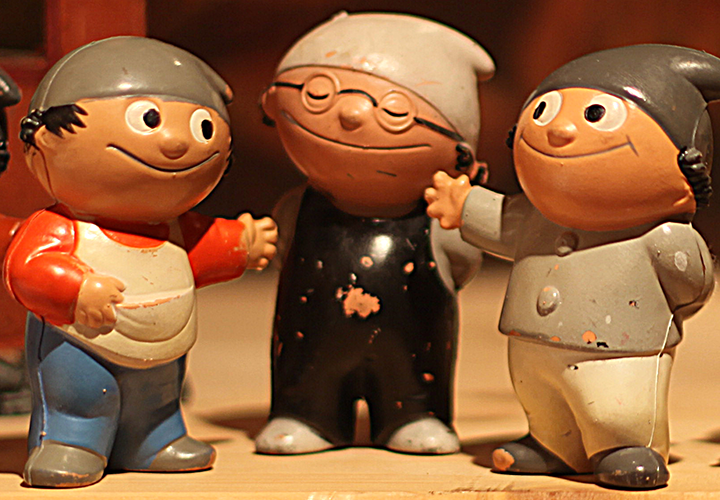}
	\caption{Sharp image}
\end{subfigure}
\quad
\begin{subfigure}[b]{0.48\textwidth}
	\includegraphics[width=\textwidth]{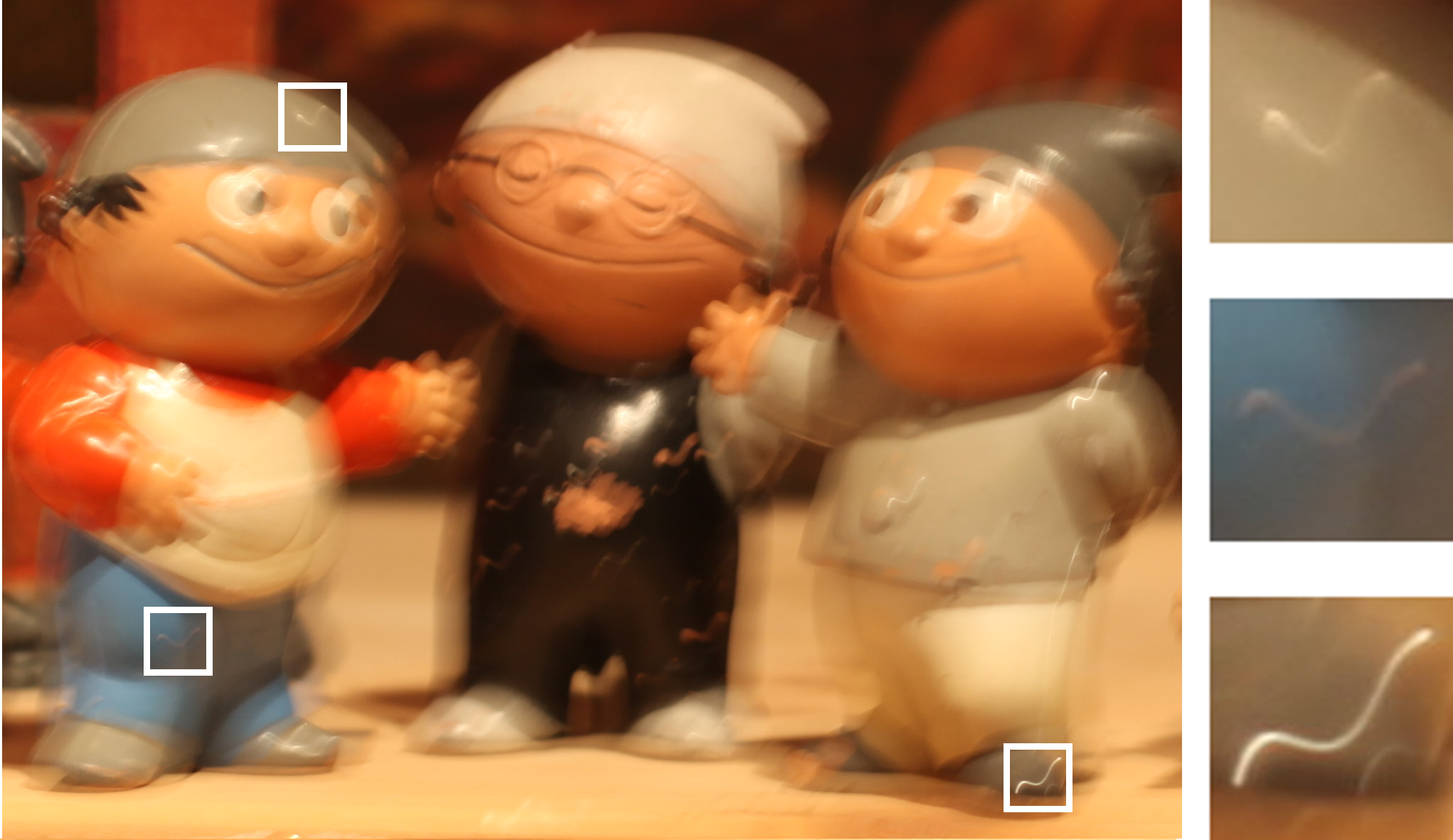}
	\caption{Blurred image and the associated PSF}
\end{subfigure}
	\caption{An example in computer vision. Image degraded by spatially varying blur due to a camera shake. Images are from \cite{hirsch2011fast} and used here by courtesy of Michael Hirsch.} \label{fig:dwarves_motion}
\end{figure}

%

\begin{figure}[htp]
\centering
\includegraphics[width=0.8\textwidth]{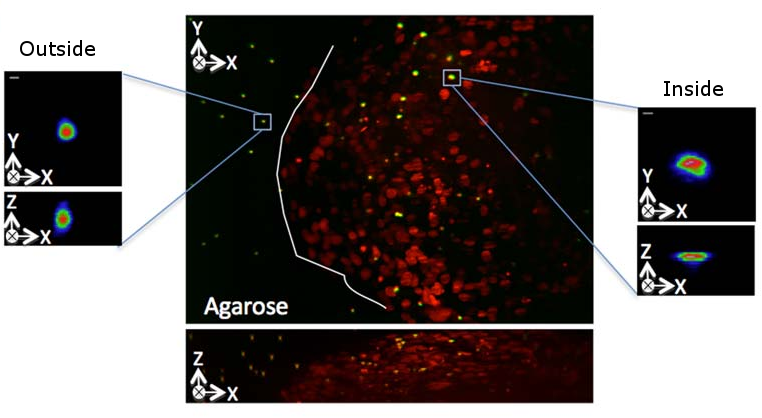}
\caption{An example in biology. Image of a multicellular tumor spheroid imaged in 3D using Selective Plane Illumination Microscope (SPIM). Fluorescence beads (in green) are inserted in the tumor model and allow the observation of the PSF at different locations. Nuclei are stained in red. On the left-hand-side, 3D PSFs outside the sample are observed. On the right-hand-side, 3D PSFs inside the sample are observed. This image is from \cite{10.1371/journal.pone.0035795} and used here by courtesy of Corinne Lorenzo.} \label{fig:billes}
\end{figure}
 
 \begin{figure}[htp]
\centering
\includegraphics[scale=0.8]{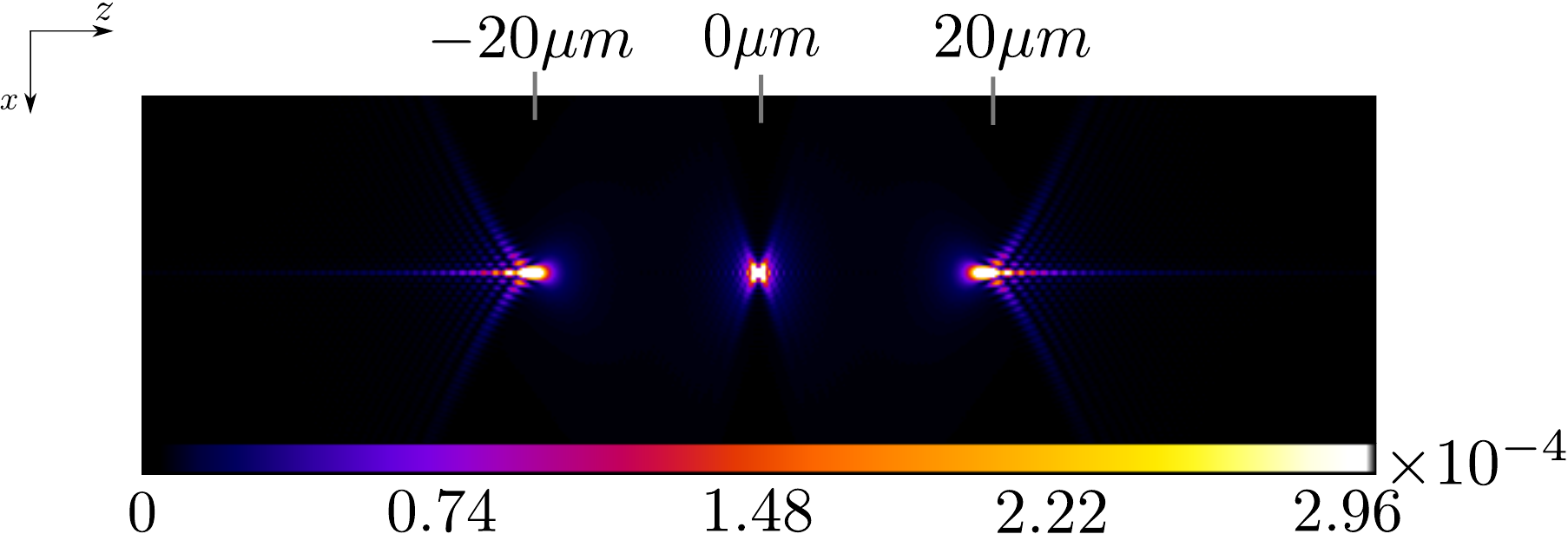}
\caption{Three PSFs displayed in a XZ plan at different $z$ depths: $-20 \micro\meter$
, $0 \micro\meter$ and $20 \micro\meter$. PSFs are generated using Gibson and Lanni 3D optical model from the PSF Generator \cite{Kirshner2011Models}. The parameters used are $n_i = 1.5$, $n_s=1.33$, $t_i = 150 \micro\meter$, $\textrm{NA} = 1.4$ and a wavelength of $610 \nano\meter$.} \label{fig:nonSeparablePSF}
\end{figure}

%% file: Body/6-notation.tex
\section{Notation} \label{sec:notation}

In this paper, we consider $d$ dimensional images defined on a domain $\Omega=[0,1]^d$. 
The space $\LL^2(\Omega)$ will denote the space of squared integrable functions defined on $\Omega$. 

Let $\balpha=(\alpha_1,\hdots, \alpha_d)$ denote a multi-index. The sum of its components is denoted $|\balpha|=\sum_{i=1}^d \alpha_i$. 
The Sobolev spaces $W^{M,p}$ are defined as the set of functions $f \in \LL^p$ with partial derivatives up to order $M$ in $\LL^p$ where $p \in [1, + \infty]$ and $M \in \N$. 
These spaces, equipped with the following norm are Banach spaces
\begin{equation}
  \norm{f}_{W^{M,p}} = \norm{f}_{\LL^p} + \abs{f}_{W^{M,p}}, \quad \text{ where,} \quad \abs{f}_{W^{M,p}} = \sum_{\abs{\balpha} = M } \norm{\p^{\balpha} f}_{\LL^p}.
\end{equation}
In this notation, $\p^{\balpha} f = \frac{\partial^{\alpha_1}}{\partial x_1^{\alpha_1}}\hdots \frac{\partial^{\alpha_d}}{\partial x_d^{\alpha_d}} f$.

Let $X$ and $Y$ denote two metric spaces endowed with their respective norms $\|\cdot\|_X$ and $\|\cdot\|_Y$. In all the paper $H:X\to Y$ will denote a linear  operator and $H^*$ its adjoint operator. The subordinate operator norm is defined by 
\begin{align*}
\|H\|_{X\to Y}= \sup_{x\in X, \|x\|_X=1} \|Hx\|_Y.
\end{align*}
The notation $\norm{H}_{p \rightarrow q}$ corresponds to the case where $X$ and $Y$ are endowed with the standard $\LL^p$ and $\LL^q$ norms. 
In all the paper, operators acting in a continuous domain are written in plain text format $H$. 
Finite dimensional matrices are written in bold fonts $\bH$. 
Approximation operators will be denoted $\wtilde H$ in the continuous domain or $\wtilde \bH$ in the discrete domain.

In this paper we consider a compactly supported wavelet basis of $\LL^2(\Omega)$. We first introduce wavelet basis of $\LL^2([0,1])$. We let $\phi$ and $\psi$ denote the scaling and mother wavelets.
We assume that the mother-wavelet $\psi$ has $M$ vanishing moments, i.e. 
\begin{equation*}
	\fa 0 \leq m < M, \quad \int_{[0,1]} t^m \psi(t) dt = 0.
\end{equation*}
We assume that $\supp(\psi)=[-c(M)/2,c(M)/2]$. Note that $c(M)\geq 2M-1$, with equality for Daubechies wavelets, see, e.g., \cite[Theorem 7.9, p. 294]{Mallat-Book}.

We define translated and dilated versions of the wavelets for $j \geq 0$ as follows
\begin{equation*}
  \phi_{j,l} = 2^{j/2} \phi\left( 2^{j} \cdot - l \right),
\end{equation*}
\begin{equation}\label{eq:defwavelets}
  \psi_{j,l} = 2^{j/2} \psi\left( 2^{j} \cdot - l \right),
\end{equation}
with $l \in \mathcal{T}_j$ and $\mathcal{T}_j = \{0,\ldots,2^j-1\}$.

In dimension $d$, we use separable wavelet bases, see, e.g., \cite[Theorem 7.26, p. 348]{Mallat-Book}. 
Let $\bm=(m_1,\hdots,m_d)$.
Define $\rho_{j,l}^0=\phi_{j,l}$ and $\rho_{j,l}^1=\psi_{j,l}$. 
Let $\be=(e_1,\hdots,e_d)\in \{0,1\}^d$. 
For ease of reading, we will use the shorthand notation $\lambda = (j,m,e)$. 
We also denote 
\begin{equation*}
\Lambda_0 = \set{ (j,m,e) \; | \; j \in \Z, \; m \in \mathcal{T}_j, \; e \in \set{0,1}^d} 
\end{equation*}
 and  
\begin{equation*}
\Lambda = \set{ (j,m,e) \; | \; j \in \Z, \; m \in \mathcal{T}_j, \; e \in \set{0,1}^d \setminus \{0\} }. 
\end{equation*}

Wavelet $ \psi_\lambda $ is defined by $ \psi_{\lambda}(x_1, \ldots, x_d) = \psi_{j,\bm}^\be(x_1,\hdots,x_d)=\rho_{j,m_1}^{e_1}(x_1)\hdots \rho_{j,m_d}^{e_d}(x_d)$.
Elements of the separable wavelet basis consist of tensor products of scaling and mother wavelets at the same scale. 
Note that if $\be\neq \bzero$ wavelet $\psi_{j,\bm}^\be$ has $M$ vanishing moments in $\R^d$.
We let $\displaystyle I_{j,m}=\cup_{e} \supp{\psi^e_{j,m}}$ and $I_{\lambda} = \supp{ \psi_{\lambda}}$.

We assume that every function $f\in\LL^2(\Omega)$ can be written as
\begin{equation*}
	\begin{split}
 u & = \dotproduct{u}{\psi^0_{0,0}} \psi^0_{0,0} + \sum_{e\in \{0,1\}^d \setminus \{0\}} \sum_{j=0}^{+\infty}\sum_{m \in \mathcal{T}_j} \dotproduct{u}{\psi^e_{j,m}} \psi^e_{j,m} \\
	& = \dotproduct{u}{\psi^0_{0,0}} \psi^0_{0,0} + \sum_{ \lambda \in \Lambda} \dotproduct{u}{\psi_{\lambda}} \psi_\lambda \\
	& = \sum_{ \lambda \in \Lambda_0} \dotproduct{u}{\psi_{\lambda}} \psi_\lambda
 	\end{split}
\end{equation*}
This is a slight abuse since wavelets defined in \eqref{eq:defwavelets} do not define a Hilbert basis of $\LL^2([0,1]^d)$. 
There are various ways to define wavelet bases on the interval \cite{cohen1993wavelets} and wavelets having a support intersecting the boundary should be given a different definition. We stick to these definitions to keep the proofs simple.

We let $\Psi^* : \LL^2(\Omega)  \rightarrow  l^2(\Z)$ denote the wavelet decomposition operator and $\Psi : l^2(\Z) \rightarrow \LL^2(\Omega)$ its associated reconstruction operator. 
The discrete wavelet transform is denoted $\bPsi:\R^N\to \R^N$. 
We refer to \cite{Mallat-Book, daubechies_ten_1992,cohen1993wavelets} for more details on the construction of wavelet bases.


%% file: Body/7-model.tex
\section{Blurring operators and their mathematical properties}
\label{sec:model}

\subsection{A mathematical description of blurring operators}

In this paper, we consider $d$-dimensional real-valued images defined on a domain $\Omega =[0,1]^d$, where $d$ denotes the space dimension. 
We consider a blurring operator $H:\LL^2(\Omega)\to \LL^2(\Omega)$ defined for any $u \in \LL^2(\Omega)$ by the following integral operator:
\begin{equation} \label{eq:integral_operator}
	\forall x \in \Omega, \quad Hu(x) = \int_{y \in \Omega} K(x,y) u(y) dy.
\end{equation}
The function $K: \Omega \times \Omega \rightarrow \R$ is a kernel that defines the Point Spread Function (PSF) $K(\cdot,y)$ at each location $y \in \Omega$. 
The image $Hu$ is the blurred version of $u$. By the Schwartz kernel theorem, a linear operator of kind \eqref{eq:integral_operator} can represent any linear operator if $K$ is a generalized function. 
We thus need to determine properties of $K$ specific to blurring operators that will allow to design efficient numerical algorithms to approximate the integral \eqref{eq:integral_operator}. 

We propose a definition of the class of blurring operators below. 

\begin{definition}[Blurring operators] \label{def:blurring_operators}
	Let $M\in \N$ and $f:[0,1]\to \R_+$ denote a non-increasing bounded function.
	An integral operator is called a blurring operator in the class $\mathcal{A}(M,f)$ if it satisfies the following properties: 
	\begin{enumerate}
		\item Its kernel $K\in W^{M,\infty}(\Omega\times \Omega)$; 
		\item The partial derivatives of $K$ satisfy:
		\begin{enumerate}
		  \item \label{def:blurring_operators:PSFsmoothness} 
		  \begin{equation} \label{eq:blurring_operators:dx_decay}
			  \forall \abs{\alpha} \leq M, \ \forall (x,y) \in \Omega\times \Omega, \quad \abs{\p_x^\alpha K(x,y)} \leq f\left( \norm{x-y}_\infty \right).
		  \end{equation}
		  \item \label{def:blurring_operators:PSFVariationSmoothness}
		  \begin{equation} \label{eq:blurring_operators:dy_decay}
			  \forall \abs{\alpha} \leq M, \ \forall (x,y) \in \Omega\times \Omega, \quad \abs{\p_y^\alpha K(x,y)} \leq f \left( \norm{x-y}_\infty \right).
		  \end{equation}
	  \end{enumerate}
	\end{enumerate}
\end{definition}

Let us justify this model from a physical point of view. Most imaging systems satisfy the following properties:
\begin{description}
 \item[Spatial decay.] \hfill \\ 
 The PSFs usually have a bounded support (e.g. motion blurs, convolution with the CCD sensors support) or at least a fast spatial decay (Airy pattern, Gaussian blurs,...). This property can be modelled as property \ref{def:blurring_operators:PSFsmoothness}. For instance, the 2D Airy disk describing the PSF due to diffraction of light in a circular aperture satisfies \ref{def:blurring_operators:PSFsmoothness} with $f(r)=\frac{1}{(1+r)^4}$ (see e.g. \cite{born1999principles}). 

 \item[PSF smoothness.] \hfill \\ 
 In most imaging applications, the PSF at $y \in \Omega$, $K(\cdot, y)$ is smooth. Indeed it is the result of a convolution with the acquisition device impulse response which is smooth (e.g. Airy disk). This assumption motivates inequality \eqref{eq:blurring_operators:dx_decay}.

 \item[PSFs variations are smooth]  \hfill \\
We assume that the PSF does not vary abruptly on the image domain. 
This property can be modelled by inequality \eqref{eq:blurring_operators:dy_decay}. 
It does not hold true in all applications. 
For instance, when sharp discontinuities occur in the depth maps, the PSFs can only be considered as piecewise regular. 
This assumption simplifies the analysis of numerical procedures to approximate $H$. 
Moreover, it seems reasonable in many settings. For instance, in fluorescence microscopy, the PSF width (or Strehl ratio) mostly depends on the optical thickness, i.e. the quantity of matter laser light has to go through, and this quantity is intrinsically continuous. Even in cases where the PSFs variations are not smooth, the discontinuities locations are usually known only approximately and it seems important to smooth the transitions in order to avoid reconstruction artifacts \cite{bar2007restoration}. 
\end{description}

\begin{remark}
A standard assumption in image processing is that the constant functions are preserved by the operator $H$. This hypothesis ensures that brightness is preserved on the image domain.
In this paper we do not make this assumption and thus encompass image formation models comprising blur and attenuation. Handling attenuation is crucial in domains such as fluroescence microscopy. 
\end{remark}

\begin{remark}
The above properties are important to derive mathematical theories, but only represent an approximation of real systems. 
The methods proposed in this paper may be applied even if the above properties are not satisfied and are likely to perform well.
It is notably possible to relax the boundedness assumption.
\end{remark}

%% file: Body/20-sparse.tex
\section{Wavelet representation of the blurring operator}
\label{sec:sparserep}

In this section, we show that blurring operators can be well approximated by sparse representations in the wavelet domain. 
Since $H$ is a linear operator in a Hilbert space, it can be written as $H = \Psi \Theta \Psi^*$, where $\Theta: l^2(\Z) \rightarrow l^2(\Z)$ is the (infinite dimensional) matrix representation of the blur operator in the wavelet domain. Matrix $\Theta$ is characterized by the coefficients:
\begin{equation}
  \label{eq:definition_theta}
  \theta_{\lambda, \mu} =  \dotproduct{H \psi_\lambda}{\psi_{\mu}} , \qquad \forall \lambda, \mu \in \Lambda.
\end{equation}

In their seminal papers \cite{meyer1992wavelets,coifman1997wavelets,BCR}, Y. Meyer, R. Coifman, G. Beylkin and V. Rokhlin prove that the coefficients of $\Theta$ decrease fastly away from its diagonal for a large class of pseudo-differential operators. They also show that this property allows to design fast numerical algorithms to approximate $H$, by thresholding $\Theta$ to obtain a sparse matrix. 
In this section, we detail this approach precisely and adapt it to the class of blurring operators. 

This section is organized as follows: first, we discuss the interest of approximating $H$ in a wavelet basis rather than using the standard discretization. 
Second, we provide various theoretical results concerning the number of coefficients necessary to obtain an $\epsilon$-approximation of $H$.

\subsection{Discretization of the operator by projection}\label{sec:discretization}

The proposed method relies on a Galerkin discretization of $H$. 
The main idea is to use a projection on a finite dimensional linear subspace $V_q=\mathrm{Span}(\varphi_1,\hdots, \varphi_q)$ of $\LL^2(\Omega)$ where $(\varphi_1,\varphi_2,\hdots)$ is an orthonormal basis of $\LL^2(\Omega)$. 
We define a projected operator $H_q$ by $H_q u=P_{V_q} H P_{V_q} u$. where $P_{V_q}$ is the projector on $V_q$.
We can associate a $q \times q$ matrix $\bTheta$ to this operator defined by $\bTheta = \left( \dotproduct{H\varphi_i}{\varphi_j}\right)_{1\leq i,j \leq q}$.

It is very common in image processing to assume that natural images belong to functional spaces containing functions with some degree of regularity. 
For instance, images are often assumed to be of bounded total variation \cite{rudin1992nonlinear}.
This hypothesis implies that 
\begin{equation}
\label{eq:function_approximation_error}
\|u-P_{V_q}u\|_2=\mathcal{O}(q^{-\alpha})
\end{equation} 
for a certain $\alpha>0$. 
For instance, in 1D, if $(\varphi_1,\varphi_2, \hdots)$ is a wavelet or a Fourier basis and $u\in H^1(\Omega)$ then $\alpha=2$. For $u\in BV(\Omega)$ (the space of bounded variation functions), $\alpha=1$ in 1D and $\alpha=1/2$ in 2D \cite{Mallat-Book,petrushev1999nonlinear}. 

Moreover, if we assume that $H$ is a regularizing operator, meaning that $\|Hu - P_{V_q} Hu\|_2 = \mathcal{O}(q^{-\beta})$ with $\beta\geq\alpha$ for all $u$ satisfying \eqref{eq:function_approximation_error}, then we have:
\begin{align*}
 &\|Hu- H_q u\|_2 \\
&=\|Hu- P_{V_q} H (u+ P_{V_q} u - u) \|_2 \\
&\leq \|Hu- P_{V_q} H u \|_2 + \|P_{V_q}H\|_{2\to 2}\|P_{V_q} u - u \|_2 \\
&= \mathcal{O}(q^{-\alpha}).
\end{align*}

This simple analysis shows that under mild assumptions, the Galerkin approximation of the operator converges and that the convergence rate can be controlled. 
The situation is not as easy for standard discretization using finite elements for instance (see, e.g., \cite{wang2011error,bartels2012total} where a value $\alpha=1/6$ is obtained in 2D for BV functions, while the simple analysis above leads to $\alpha=1/2$).

\subsection{Discretization by projection on a wavelet basis}

In order to get a representation of the operator in a finite dimensional setting, we truncate the wavelet representation at scale $J$. 
This way, we obtain an operator $\widetilde H$ acting on a space of dimension $N$, where $N=1+\sum_{j=0}^{J-1}(2^d-1) 2^{dj} $ denotes the numbers of wavelets kept to represent images.

After discretization, it can be written in the following convenient form:
\begin{equation}\label{eq:matrixH}
 \bH= \bPsi \bTheta \bPsi^*
\end{equation}
where $\bPsi:\R^{N} \to \R^{N}$ is the discrete separable wavelet transform.
Matrix $\bTheta$ is an $N \times N$ matrix which corresponds to a truncated version (also called finite section) of the matrix $\Theta$ defined in \eqref{eq:definition_theta}.

\subsection{Theoretical guarantees with sparse approximations}

Sparse approximations of integral operators have been studied theoretically in \cite{BCR, meyer1992wavelets}. 
They then have been successfully used in the numerical analysis of PDEs \cite{dahmen_wavelet_1993,cohen2002adaptive,cohen2003numerical}.
Surprisingly, they have been scarcely applied to image processing. 
The two exceptions we are aware of are the paper \cite{CMY-Foveation}, where the authors show that wavelet multipliers can be useful to approximate foveation operators. 
More recently, \cite{wei2014fast} proposed an approach that is very much related to that of our paper. 

Let us provide a typical result that motivates the proposed approach. 
	
		

\begin{lemma}[Decay of $\theta_{\lambda,\mu}$] \label{lem:decay}
Assume that $H$ is a blurring operator (see Definition \ref{def:blurring_operators}) in the class $\mathcal{A}(M,f)$.
Assume that the mother wavelet is compactly supported with $M$ vanishing moments.
  
Then, the coefficients of $\Theta$ satisfy the following inequality for all $\lambda = (j,m,e) \in \Lambda$ and $\mu = (k,n,e') \in \Lambda$:
\begin{equation} \label{eq:decay}
  \abs{\theta_{\lambda, \mu} } \leq C_M 2^{-\left(M + \frac{d}{2} \right)\abs{j-k}} 2^{-\min(j,k)\left(M + d\right)} f_{\lambda,\mu}
\end{equation}
where $f_{\lambda, \mu} = f\left( \dist{I_{\lambda}}{I_{\mu}} \right)$, $C_M$ is a constant that does not depend on $\lambda$ and $\mu$
and 
\begin{align}
\dist{I_{\lambda}}{I_{\mu}} &= \inf_{x\in I_{\lambda},\, y\in I_{\mu}} \|x-y\|_\infty \nonumber \\
& = \max \left(0, \norm{2^{-j}m - 2^{-k}n}_\infty - (2^{-j} + 2^{-k})\frac{c(M)}{2}\right). \label{eq:defdist}
\end{align}
\end{lemma}
	
\begin{proof}
  See Appendix \ref{appendix:proof_decay}.
\end{proof}


Lemma \ref{lem:decay} is the key to obtain all subsequent complexity estimates.

\begin{theorem} \label{thm:proof_thresh}
Let $\bTheta_{\eta}$ be the matrix obtained by zeroing all coefficients in $\mathbf{\Theta}$ such that
\begin{equation*}
	2^{-min(j,k)(M+d)} f_{\lambda, \mu}  \leq \eta,
\end{equation*}	
with $\lambda = (j,m,e) \in \Lambda$ and $\mu = (k,n,e') \in \Lambda$.

Let $\wtilde \bH_{\eta} = \mathbf{\Psi} \bTheta_{\eta} \mathbf{\Psi}^*$ denote 
the resulting operator.
Suppose that $f$ is compactly supported in $[0,\kappa]$ and that $\eta\leq \log_2(N)^{-(M+d)/d}$. Then:
\begin{enumerate}
 \item[i)] 	The number of non zero coefficients in $\bTheta_\eta$ is bounded above by
	\begin{equation}\label{eq:numberofcoeffs}
	 C'_M N \kappa^d \; \eta^{-\frac{d}{M+d}} 
	\end{equation}
	where $C'_M>0$ is independent of $N$.
 \item[ii)] The approximation $\wtilde \bH_{\eta}$ satisfies $\norm{\bH - \wtilde \bH_{\eta}}_{2 \rightarrow 2} \lesssim \eta^{ \frac{M}{M+d} }$.
 \item[iii)] The number of coefficients needed to satisfy $\norm{\bH - \wtilde \bH_{\eta}}_{2 \rightarrow 2} \leq \epsilon$ is bounded above by
	\begin{equation}\label{eq:wavelet_quality}
		C''_M N \kappa^d \; \epsilon^{-\frac{d}{M}}
	\end{equation}
	where $C''_M>0$ is independent of $N$.
\end{enumerate}
\end{theorem}
	\begin{proof}
		See Appendix \ref{appendix:proof_thresh_d_dim}.
	\end{proof}

Let us summarize the main conclusions drawn from this section:
\begin{itemize}
 \item A discretization in the wavelet domain provides better theoretical guarantees than the standard quadrature rules (see Section \ref{sec:discretization}).
 \item  The method is capable of handling \textit{automatically} the degree of smoothness of the integral kernel $K$ since there is a dependency in $\epsilon^{-\frac{d}{M}}$ where $M$ is the smoothness of the integral operator.
 \item We will see in the next section that the method is quite versatile since different sparsity patterns can be chosen depending on the knowledge of the blur kernel and on the regularity of the signals that are to be processed. 
 \item The method can also handle more general singular operators as was shown in the seminal papers \cite{meyer1992wavelets,coifman1997wavelets,BCR}.
\end{itemize}

\begin{remark}
	Similar bounds as \eqref{eq:decay} can be derived with less stringent assumptions. First, the domain can be unbounded, given that kernels have a sufficiently fast decay at infinity. Second, the kernel can blow up on its diagonal, which is the key to study Calderon-Zygmund operators (see \cite{meyer1992wavelets,coifman1997wavelets,BCR} for more details). We sticked to this simpler setting to simplify the proofs.
\end{remark}

%% file: Body/25-algorithms.tex
\section{Identification of sparsity patterns}
\label{sec:algorithms}

A key step to control the approximation quality is the selection of the coefficients in the matrix $\bTheta$ that should be kept. 
For instance, a simple thresholding of $\bTheta$ leads to sub-optimal and somewhat disappointing results. 
In this section we propose algorithms to select the most relevant coefficients for images belonging to functional spaces such as that of bounded variation functions.
We study the case where $\bTheta$ is known completely and the case where only an upper-bound such as \eqref{eq:decay} is available.

\subsection{Problem formalization}

Let $\bH$ be the $N^d \times N^d$ matrix defined in equation \eqref{eq:matrixH}. 
We wish to approximate $\bH$ by a matrix $\wtilde \bH_K$ of kind $\mathbf{\Psi} \bS_K \mathbf{\Psi^*}$ where $\bS_K$ is a matrix with at most $K$ non-zero coefficients.
Let $\mathbb{S}_K$ denote the space of $N\times N$ matrices with at most $K$ non-zero coefficients. 
The problem we address in this paragraph reads
\begin{align*}
  &\min_{\bS_K \in \mathbb{S}_K } \norm{ \bH - \wtilde \bH_K  }_{X\to 2}  \\
 &= \min_{\bS_K \in \mathbb{S}_K} \max_{\norm{\bu}_X \leq 1} \norm{\bH \bu - \mathbf{\Psi} \bS_K \mathbf{\Psi^*} \bu}_2.
\end{align*}
The solution of this problem provides the best $K$-sparse matrix $\bS_K$, in the sense that no other choice provides a better SNR uniformly on the unit-ball $\{\bu \in \R^{N}, \norm{\bu}_X \leq 1\}$.

\subsubsection{Theoretical choice of the space $X$}

The norm $\norm{ \cdot }_X$ should be chosen depending on the type of images that have to be blurred.
For instance, it is well-known that natural images are highly compressible in the wavelet domain  \cite{Mallat-Book,simoncelli1999modeling}.
This observation is the basis of JPEG2000 compression standard.
Therefore, a natural choice could be to set $\norm{\bu}_X = \norm{ \mathbf{\Psi}^* \bu}_1$. 
This choice will ensure a good reconstruction of images that have a wavelet decomposition with a low $\ell^1$-norm.

Another very common assumption in image processing is that images have a bounded total variation.
The space of functions with bounded total variation \cite{aubert2006mathematical} contains images discontinuous along edges with finite length. 
It is one of the most successful tools for image processing tasks such as denoising, segmentation, reconstruction, ...
Functions in $BV(\Omega)$ can be characterized by their wavelet coefficients \cite{petrushev1999nonlinear,Mallat-Book}.
For instance, if $u \in BV(\Omega)$, then
\begin{equation}\label{eq:charaBV}
\sum_{\lambda \in \Lambda_0} 2^{j(1 - \frac{d}{2})} \abs{\dotproduct{u}{\psi_{\lambda}}} < +\infty
\end{equation}
for all wavelet bases. This results is due to embeddings of $BV$ space in Besov spaces which are characterized by their wavelet coefficients (see \cite{cohen2003numerical} for more details on Besov spaces).
This result motivated us to consider norms defined by
\begin{equation*}
 \norm{\bu}_X = \norm{\bSigma \bPsi^* \bu}_1
\end{equation*}
where $\bSigma=\mathrm{diag}(\sigma_1,\hdots,\sigma_N)$ is a diagonal matrix. Depending on the regularity level of the images considered, different diagonal coefficients can be used. 
For instance, for BV signals in 1D, one could set $\sigma_i = 2^{j(i)/2}$ where $j(i)$ is the scale of the $i$-th wavelet, owing to \eqref{eq:charaBV}.

\subsubsection{Practical choice of the space $X$}\label{subsub:practicalX}

More generally, it is possible to adapt the weights $\sigma_i$ depending on the images to recover.
Most images exhibit a similar decay of wavelet coefficients across subbands. 
This decay is a characteristic of the functions regularity (see e.g. \cite{hernandez1996first}).
To illustrate this fact, we conducted a simple experiment in Figure \ref{fig:decay_wavelet}. 
We evaluate the maximal value of the amplitude of wavelet coefficients of three images with different contents across scales. 
The wavelet transform is decomposed at level $4$ and we normalize the images so that their maximum wavelet coefficient is $1$. 
As can be seen even though the maximal values differ from one image to the next, their overall behavior is the same: amplitudes decay nearly dyadically from one scale to the next. The same phenomenon can be observed with the mean value. 

\begin{figure}[htp]
\centering
\begin{subfigure}[b]{0.3\textwidth} 
	\frame{\includegraphics[width=\textwidth]{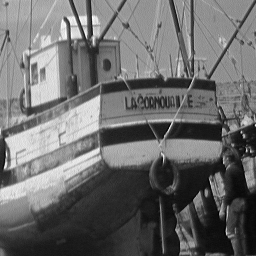}}
	\caption{Boat \\ ($1 - 0.02 - 0.02 - 0.009 - 0.004$)} 
\end{subfigure}
\quad
\begin{subfigure}[b]{0.3\textwidth} 
	\frame{\includegraphics[width=\textwidth]{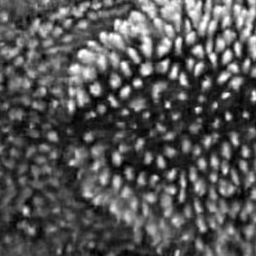}}
	\caption{Drosophila \\ ($1 - 0.04 - 0.02 - 0.007 - 0.004$)} 
\end{subfigure}
\quad
\begin{subfigure}[b]{0.3\textwidth} 
	\frame{\includegraphics[width=\textwidth]{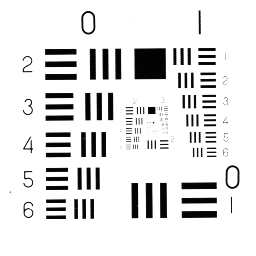}}
	\caption{Pattern \\ ($1 - 0.02 - 0.02 - 0.009 - 0.004$)} 
\end{subfigure}
\caption{Three pictures and the mean amplitude of their wavelet coefficients at each scale of the wavelet transform.} \label{fig:decay_wavelet}
\end{figure}

This experiment suggests setting $\sigma_i=2^{j(i)}$ in order to normalize the wavelet coefficients amplitude in each subband. 
Once again, the same idea was explored in \cite{wei2014fast}. 

\subsubsection{An optimization problem}

We can now take advantage of the fact that images and operators are sparse in the same wavelet basis.
Let $\bz = \bPsi^* \bu$ and $\bDelta = \bTheta - \bS_K$. 
Since we consider orthogonal wavelet transforms, we have $\norm{\bPsi \bu}_2 = \norm{\bu}_2$, for any $\bu \in \R^N$ therefore:
\begin{align*}		
\norm{ \bH - \wtilde{\bH}_K }_{X \rightarrow 2} &= \max_{ \norm{\bu}_X \leq 1 } \norm{ \bPsi ( \bTheta - \bS_K ) \bPsi^* \bu }_2 \\
&= \max_{ \norm{\bSigma \bz}_1 \leq 1 } \norm{ ( \bTheta - \bS_K ) \bz}_2 \\
& = \max_{ \norm{\bz}_1 \leq 1 } \norm{ \bDelta \bSigma^{-1} \bz}_2.
\end{align*}
Since the operator norm $\norm{\bA}_{1 \to 2} = \displaystyle \max_{1 \leq i \leq N} \norm{\bA^{(i)}}_2$, where $\bA^{(i)}$ denote the $i$-th column of the $N \times N$ matrix $\bA$ and by remarking that $(\bDelta \bSigma^{-1})^{(i)} = \bDelta^{(i)} \sigma_i^{-1}$, we finally get the following simple expression for the operator norm:
\begin{equation} \label{eq:normX-2}
  \norm{ \bH - \wtilde{\bH} }_{X \rightarrow 2} = \max_{1\leq i \leq N} \frac{1}{\sigma_i} \norm{ \bDelta^{(i)} }_2.
\end{equation}
Our goal is thus to find the solution of:
\begin{equation}\label{eq:finalthreshprob}
\min_{\bS_K \in \mathbb{S}_K} \max_{1\leq i \leq N} \frac{1}{\sigma_i} \norm{ \bDelta^{(i)} }_2.
\end{equation}


\subsection{Link with the approach in \cite{wei2014fast}}\label{subsec:wei}

In this paragraph, we show that the method proposed in \cite{wei2009fast, wei2014fast}, can be interpreted with the formalism given above. 
In those papers, $\bTheta$ is approximated by $\wtilde \bTheta$ using the following rule:
\begin{equation}\label{eq:thresh_rule_bouman}
	\wtilde \bTheta_{i,j} = \left\{ \begin{array}{l l}
		\bTheta_{i,j} & \text{if } \frac{\bTheta_{i,j}}{w_j} \text{ is in the $K$ largest values of $\bTheta \bW^{-1}$} \\
		0			  & \text{otherwise}.
	\end{array} \right.
\end{equation}
The weights $w_i$ are set as constant by subbands and learned as described in paragraph \ref{subsub:practicalX}.

The thresholding rule \eqref{eq:thresh_rule_bouman} can be interpreted as the solution of the following problem:
\[
	\min_{\wtilde \bTheta \in \mathbb{S}_K} \norm{\bTheta - \wtilde \bTheta}_{\bW \to \infty},
\]
where here $\norm{ x }_{\bW} = \norm{\bW x}_1$ with $\bW = \textrm{diag}(w_i)$ a diagonal matrix. 
Indeed, the above problem is equivalent to:
\[
	\min_{\wtilde \bTheta \in \mathbb{S}_K} \max_{1 \leq i,j \leq N} \abs{ \frac{1}{w_j} \left(\bTheta - \wtilde \bTheta\right)_{i,j} }.
\]
In other words, the method proposed in \cite{wei2009fast, wei2014fast} constructs a $K$ best-term approximation of $\bTheta$ in the metric $\|\cdot\|_{\bW \to \infty}$.

Overall, the problem is very similar to \eqref{eq:finalthreshprob}, except that the image quality is evaluated through an infinite norm in the wavelet domain, while we propose using a Euclidean norm in the spatial domain. We believe that this choice is more relevant for image processing since the SNR is the most common measure of image quality. In practice, we will see in the numerical experiments that both methods lead to very similar practical results.

Finally, let us mention that the authors in \cite{wei2014fast} have an additional concern of storing the matrix representation with the least memory. They therefore \textit{quantize} the coefficients in $\bTheta$. Since the main goal in this paper is the design of fast algorithms for matrix-vector products, we do not consider this extra refinement.

\subsection{An algorithm when $\bTheta$ is known}

Finding the minimizer of problem \eqref{eq:finalthreshprob} can be achieved using a simple greedy algorithm: the matrix $\bS_{k+1}$ is obtained by adding the largest coefficient of the column $\bDelta_i$ with largest Euclidean norm to $\bS_k$. This procedure can be implemented efficiently by using quick sort algorithms. The complete procedure is described in Algorithm \ref{algo:thresh}. 
The overall complexity of this algorithm is $\mathcal{O}(N^{2} \log(N))$. 
The most computationally intensive step is the sorting procedure in the initialisation. 
The loop on $k$ can be accelerated by first sorting the set $(\gamma_j)_{1\leq j\leq N}$, but the algorithm's complexity remains essentially unchanged. 

\begin{algorithm}[htp]
\SetKwInput{KwInit}{Initialization}
\DontPrintSemicolon
\SetAlgoLined
  \KwIn{  \\
	$\bTheta$: $N\times N$ matrix; \\
	$\bSigma$: Diagonal matrix; \\
	$K$: the number of elements in the thresholded matrix;
  }
  \KwOut{\\ $\bS_K$: Matrix minimizing \eqref{eq:finalthreshprob}}
  \KwInit{\\
      Set $\bS_K= \mathbf{0} \in \R^{N\times N}$; \\
      Sort the coefficients of each column $\bTheta^{(j)}$ of $\bTheta$ in decreasing order; \\
      Obtain $\bA^{(j)}$ the sorted columns $\bTheta^{(j)}$ and index sets $I_j$; \\
      The sorted columns $\bA^{(j)}$ and index set $I_j$ satisfy $\bA^{(j)}(i) = \bTheta^{(j)}(I_j(i))$; \\
      Compute the norms $\gamma_j=\frac{\|\bTheta^{(j)}\|_2^2}{\sigma_j^2}$;\\
      Define $\bO=(1,\hdots,1) \in \R^{N}$; \\
      $\bO(j)$ is the index of the largest coefficient in $\bA^{(j)}$ not yet added to $\bS_K$;\\
   }
  \Begin{
     \For{$k=1$ \KwTo $K$}{
	Find $l = \displaystyle \argmax_{j=1 \ldots N} \gamma_j$ ; \\
	\textit{(Find the column $l$ with largest Euclidean norm)} \\
	Set $\bS_K(I_l(\bO(l)),l)  = \bTheta ( I_l(\bO(l)),l )$ ; \\
	\textit{(Add the coefficient in the $l$-th column at index $I_l( \bO(l) )$} \\
	Update $\displaystyle \gamma_{l} = \gamma_{l} - \left(\frac{\bA^{(l)}( \bO(l) )}{\sigma_l}\right)^2$ ; \\
	\textit{ (Update norms vector) } \\
	Set $\bO(l)=\bO(l)+1$ ; \\
	\textit{(The next value to add in $l$-th column will be at index $\bO(l)+1$)}  \\
    }
  }
\caption{An algorithm to find the minimizer of \eqref{eq:finalthreshprob}.}\label{algo:thresh}
\end{algorithm}

\subsection{An algorithm when $\bTheta$ is unknown}

In the previous paragraph, we assumed that the full matrix $\bTheta$ was known.
There are at least two reasons that make this assumption irrelevant. 
First, computing $\bTheta$ is very computationally intensive and it is not even possible to store this matrix in RAM for medium sized images (e.g. $512\times 512$). 
Second, in blind deblurring problems, the operator $\bH$ needs to be inferred from the data and adding priors on the sparsity pattern of $\bS_K$ might be an efficient choice to improve the problem identifiability.

When $\bTheta$ is unknown, we may take advantage of equation \eqref{eq:decay} to define sparsity patterns. 
A naive approach would consist in applying Algorithm \eqref{algo:thresh} directly on the upper-bound \eqref{eq:decay}. 
However, this matrix cannot be stored and this approach is applicable only for small images. 
In order to reduce the computational burden, one may take advantage of the special structure of the upper-bound: 
equation \eqref{eq:decay} indicates that the coefficients $\theta_{\lambda,\mu}$ can be discarded for sufficiently large $|j-k|$ and sufficiently large distance between the wavelet supports.
Equation \eqref{eq:decay} thus means that for a given wavelet $\psi_{\lambda}$, only its spatial neighbours in neighbouring scales have significant correlation coefficients $ \dotproduct{H \psi_{\lambda}}{\psi_{\mu}}$. 
We may thus construct sparsity patterns using the notion of multiscale neighbourhoods defined below.

\begin{definition}[Multiscale shift]
The multiscale shift $s_{\lambda, \mu} \in \mathbb{Z}^d$ between two wavelets $\psi_{\lambda}$ and $\psi_{\mu}$ is defined by
\begin{equation}\label{eq:defshift}
s_{\lambda, \mu}=\left\lfloor \frac{n}{2^{\max(k-j,0)}} \right \rfloor - \left\lfloor \frac{m}{2^{\max(j-k,0)}} \right \rfloor.
\end{equation}
\end{definition}
We recall that $\lambda =(j,m,e) \in \Lambda$ and $\mu =(k,n,e') \in \Lambda$. Note that for $k=j$, the multi-scale shift is just $s_{\lambda, \mu}=n-m$ and corresponds to the standard shift between wavelets, measured as a multiple of the characteristic size $2^{-j}$.
The divisions  by $2^{\max(k-j,0)}$ and $2^{\max(j-k,0)}$ allow to rescale the shifts at the coarsest level.
This definition is illustrated in Figure \ref{fig:multiscale_neighborhood1}.
\begin{figure}[htp]
\centering
\includegraphics[width=0.5\textwidth]{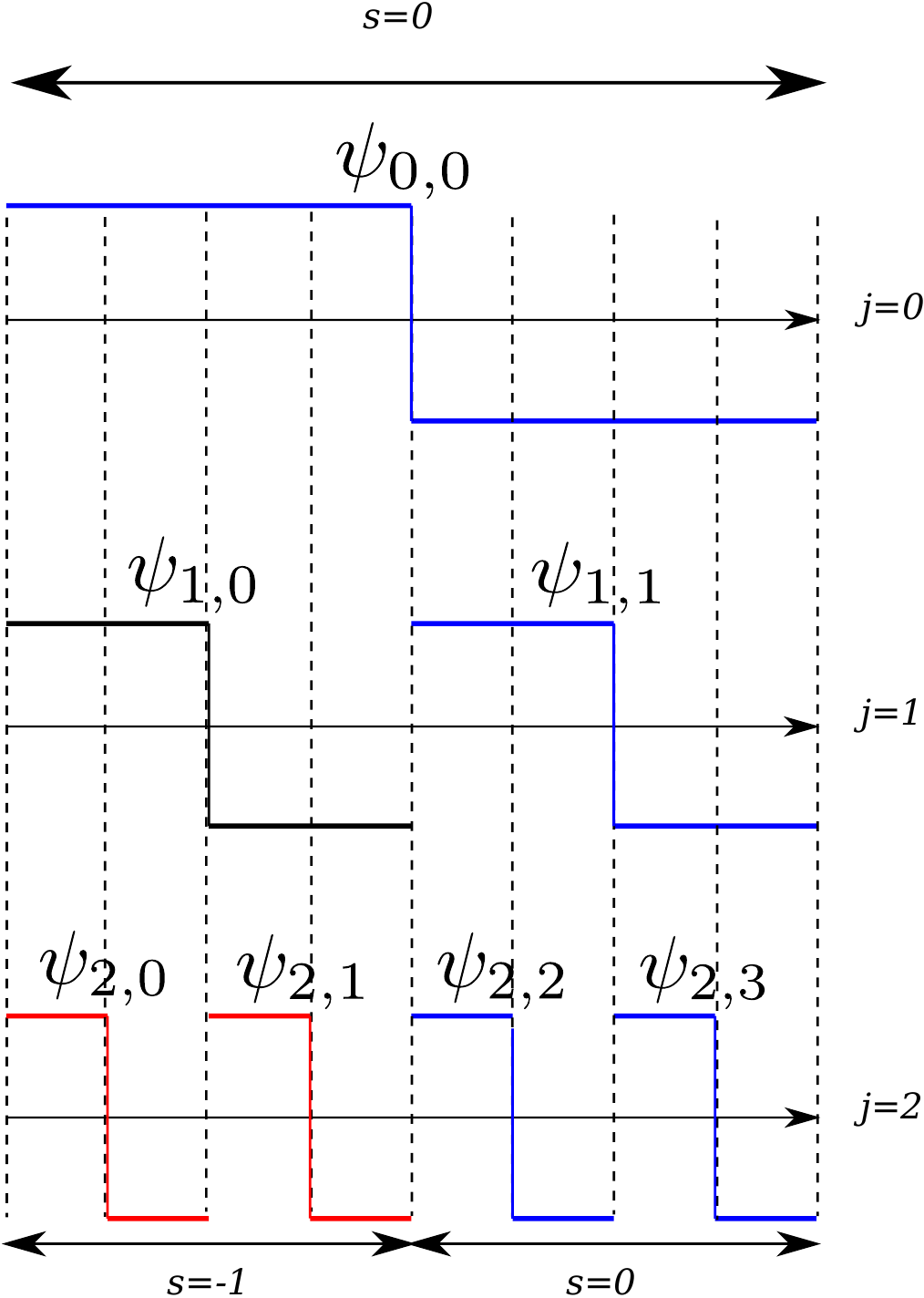}
\caption{Illustration of a multiscale shift on a 1D signal of size 8 with the Haar basis. The shifts are computed with respect to wavelet $\psi_{1,1}$. 
Wavelets $\psi_{0,0}$, $\psi_{2,2}$ and $\psi_{2,3}$ have a multiscale shift $s=0$ with $\psi_{1,1}$ since their support intersects that of $\psi_{1,1}$. Wavelets $\psi_{1,0}$, $\psi_{2,0}$ and $\psi_{2,1}$ have a multiscale shift $s=-1$ with $\psi_{1,1}$ since their support intersects that of $\psi_{1,0}$.} \label{fig:multiscale_neighborhood1}
\end{figure}

\begin{definition}[Multiscale neighborhood]
Let 
\begin{equation*}
\bN=\set{(j,(k,s)), (j,k)\in \{0,\hdots, \log_2(N)-1\}^2, s\in \{0,\hdots, 2^{\min(j,k)}-1\}^d} 
\end{equation*}
denote the set of all neighborhood relationships, i.e. the set of all possible couples of type (scale, (scale,shift)). 
A multiscale neigborhood $\mN$ is an element of the powerset $\mathcal{P}(\bN)$.
\end{definition}

\begin{definition}[Multiscale neighbors]
Given a multiscale neigborhood $\mN$, two wavelets $\psi_\lambda$ and $\psi_{\mu}$ will be said to be $\mN$-neighbors if 
$(j,(k,s_{\lambda,\mu}))\in \mN$ where $s_{\lambda,\mu}$ is defined in equation \eqref{eq:defshift}.
\end{definition}

\begin{figure}[htp]
\centering
\includegraphics[width=0.99\textwidth]{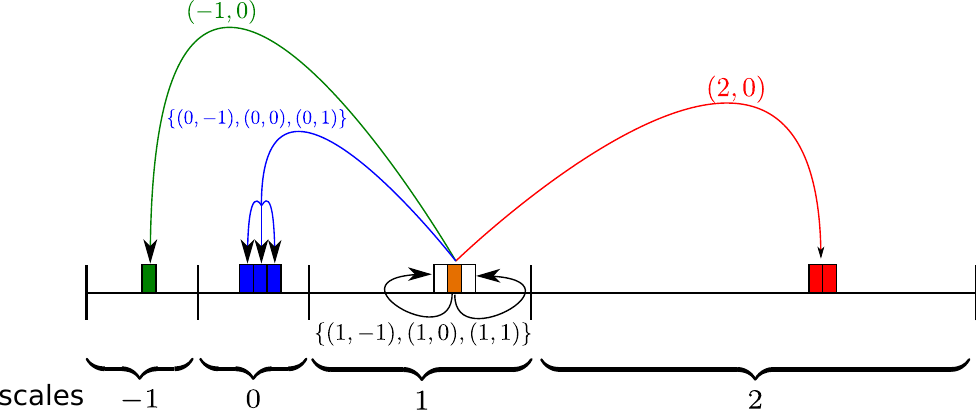}
\caption{Illustration of a multiscale neighborhood on a 1D signal. In this example, the neighborhood at scale 1 is $\mN(1) = \set{(-1,0),(0,-1),(0,0),(0,1),(1,-1),(1,0),(1,1),(2,0)}$. Notice that the two red wavelets at scale 2 are neighbors of the orange wavelet at scale 1 and that this relationship is described through only one shift.} \label{fig:multiscale_neighborhood2}
\end{figure}

The problem of finding a sparsity pattern is now reduced to finding a good multiscale neighborhood. 
In what follows, we let $\bN(j)=\{(k,s), (j,(k,s))\in \bN\}$ denote the set of all possible neighborhood relationships at scale $j$. This is illustrated in Figure \ref{fig:multiscale_neighborhood2}.
Let $\mN\in \mathcal{P}(\bN)$ denote a multiscale neighborhood. 
We define the matrix $\bS_{\mN}$ as follows:
\begin{equation*}
 \bS_{\mN}(\lambda,\mu) = \left\{ 
\begin{array}{ll} 
\theta_{\lambda, \mu} & \textrm{ if } \psi_{\lambda} \textrm{ is an } \mN \textrm{-neighbor of } \psi_{\mu}  \\ 
0 & \textrm{otherwise}.
\end{array} \right.
\end{equation*}
Equation \eqref{eq:decay} indicates that
\begin{equation*}
 |\theta_{\lambda,\mu}| \leq u(j,k,s)
\end{equation*}
with
\begin{equation} \label{eq:upper_bound_u}
u(j,k,s)=  C_M 2^{-\left(M + \frac{d}{2} \right)\abs{j-k}-\left(M + d\right)\min(j,k)} f_{j,k,s}
\end{equation}
and $f_{j,k,s}=f\left( \max \left(0, 2^{-\min(j,k)}\norm{s}_\infty - (2^{-j} + 2^{-k})\frac{c(M)}{2}\right)\right)$. 
Let $\bU$ be the matrix defined by $\bU(\lambda,\mu)= u(j,k,s_{\lambda,\mu})$. Finding a good sparsity pattern can now be achieved by solving the following problem:
\begin{equation}\label{eq:optim_multi_neigh}
\min_{ \substack{ \mN\in \mathcal{P}(\bN) \\ |\mN|=K}} \max_{1\leq i \leq N} \frac{1}{\sigma_i} \norm{ (\bU - \bS_\mN)^{(i)} }_2
\end{equation}
where $(\bU - \bS_\mN)^{(i)}$ denotes the $i$-th column of $(\bU - \bS_\mN)$. 

In what follows, we assume that $\sigma_i$ only depends on the scale $j(i)$ of the $i$-th wavelet.
Similarly to the previous section, finding the optimal sparsity pattern can be performed using a greedy algorithm.
A multiscale neighborhood is constructed by iteratively adding the couple (scale, (scale,shift)) that minimizes a residual. 
This technique is described in Algorithm \ref{algo:thresh2}.

\begin{algorithm}[htp]
\SetKwInput{KwInit}{Initialization}
\DontPrintSemicolon
\SetAlgoLined
  \KwIn{  \\
        	$u$: Upper-bound defined in \eqref{eq:upper_bound_u}; \\
        	$\bSigma$: Diagonal matrix; \\
        	$K$: the number of elements of the neighborhood;
  }
  \KwOut{\\ $\mN$: multiscale neighborhood minimizing \eqref{eq:optim_multi_neigh}}
  \KwInit{\\
      Set $\mN=\emptyset$; \\
      Compute the norms $\gamma_k=\frac{\|\bU^{(k)}\|_2^2}{\sigma_k^2}$ using the upper-bound $u$;\\
   }
  \Begin{
     \For{$k=1$ \KwTo $K$}{
	Find $j^* = \displaystyle \argmax_{j=1 \ldots N} \gamma_j$ ; \\ \textit{(The column with largest norm)} \\
	Find $(k^*,s^*) = \displaystyle \argmax_{(k,s)\in \bN(j^*)\ } u^2(j^*,k,s)2^{\max(j^*-k,0)}$ ; \\ \textit{(The best scale and shift for this column is $(k^*,s^*)$)}\\
	\textit{(The number of elements in the neighborhood relationship $(j^*,(k,s))$ is $2^{\max(j^*-k,0)}$)} \\
	Update $\displaystyle \mN=\mN\cup\{(j^*,(k^*,s^*))\}$ ; \\
	Set $\gamma_k=\gamma_k-u^2(j^*,k^*,s^*)\cdot 2^{\max(j^*-k,0)}$ \\
    }
  }
\caption{An algorithm to find the minimizer of \eqref{eq:optim_multi_neigh}.} \label{algo:thresh2}
\end{algorithm}

Note that the norms $\gamma_k$ only depend on the scale $j(k)$, so that the initialisation step only requires $\mathcal{O}(N\log_2(N))$ operations.
Similarly to Algorithm \ref{algo:thresh}, this algorithm can be accelerated by first sorting the elements of $u(j,k,s)$ in decreasing order. 
The overall complexity for this algorithm is $\mathcal{O}(N\log(N)^2)$ operations.

%% file: Body/30-numerical.tex
\section{Numerical experiments}
\label{sec:numerics}

In this section we perform various numerical experiments in order to illustrate the theory proposed in the previous sections and to compare the practical efficiency of wavelet based methods against windowed convolutions (WC) based approaches. We first describe the operators and images used in our experiments. Second, we provide numerical experiments for the direct problems. Finally, we provide numerical comparisons for the inverse problem.

\subsection{Preliminaries}

\subsubsection{Test images}
We consider a set of 16 images of different natures: standard image processing images (the boat, the house, Lena, Mandrill (see Figure \ref{im:mandrill}), peppers, cameraman), two satellite images, three medical images, three buildings images, and two test pattern images (see Figure \ref{im:letters}).
Due to memory limitations, we only consider images of size $N=256\times 256$. Note that a full matrix of size $N \times N$ stored in double precision weighs around 32 gigabytes.

\subsubsection{Test operators}
Three different blur kernels of different complexities are considered, see Figure \ref{fig:test_kernels}. 
The PSFs in Figure \ref{fig:kernel1} and \ref{fig:kernelRotation} modeled for all $x \in [0,1]^2$ by 2D Gaussians. 
Therefore the associated kernel is defined for all $(x,y) \in [0,1]^2 \times [0,1]^2$ by 
\[K(x,y) = \frac{1}{2\pi \abs{C(y)}} \exp{\left[\frac{1}{2}(y-x)^T C^{-1}(y) (y - x)\right]}.\] 
The covariance matrices $C$ are defined as:
\begin{itemize}
  \item In Figure \ref{fig:kernel1}: $C(y) = \begin{pmatrix} f(y_1) & 0 \\ 0 & f(y_1) \end{pmatrix}$ with $f(t) = 2t$, for $t \in [0,1]$. The PSFs are truncated out of a $11 \times 11$ support.
  \item In Figure \ref{fig:kernelRotation}: $C(y) = R(y)^T D(y) R(y)$ where $R(y)$ is a rotation matrix of angle $\theta=\arctan{\left(\frac{y_1 - 0.5}{y_2 - 0.5}\right)}$ and $D(y) = \begin{pmatrix}
	g(y) & 0 \\ 0 & h(y) 
\end{pmatrix}$ with $g(y) = 10 \norm{y - (0.5,0.5)^T}_2$ and  $h(y) = 2 \norm{y - (0.5,0.5)^T}_2$. 
The PSFs are truncated out of a $21 \times 21$ support.
\end{itemize}
The PSFs in Figure \ref{fig:kernel_rotation_skew} were proposed in \cite{simpkins2014parameterized} as an approximation of real spatially optical blurs.

\begin{figure}[htp]
\centering
\begin{subfigure}[b]{0.4\textwidth}
	\includegraphics[width=\textwidth]{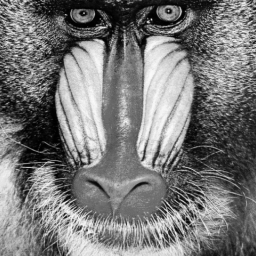}
	\caption{Mandrill} \label{im:mandrill}
\end{subfigure}
\quad
\begin{subfigure}[b]{0.4\textwidth} 
	\frame{\includegraphics[width=\textwidth]{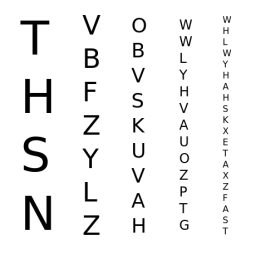}}
	\caption{Letters} \label{im:letters}
\end{subfigure}
	\caption{The two images of size $256 \times 256$ used in these numerical experiments} \label{fig:test_images}
\end{figure}

\begin{figure}[htp]
\centering
\begin{subfigure}[b]{0.30\textwidth} 
	\frame{\includegraphics[width=\textwidth]{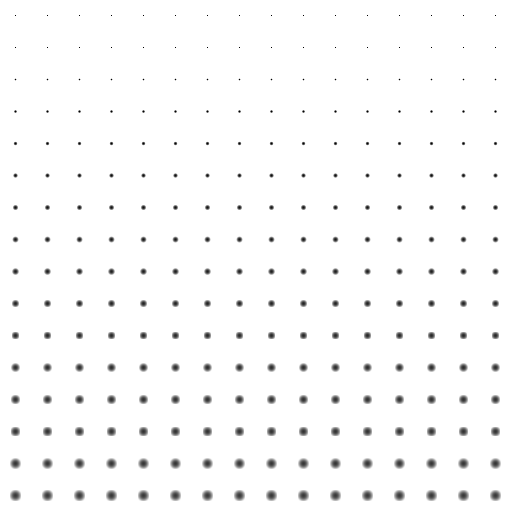}}
	\caption{} \label{fig:kernel1}
\end{subfigure}
\quad
\begin{subfigure}[b]{0.30\textwidth}
	\frame{\includegraphics[width=\textwidth]{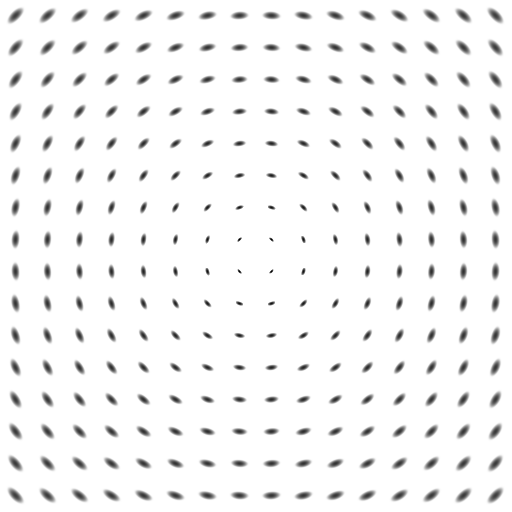}}
	\caption{} \label{fig:kernelRotation}
\end{subfigure}
\quad
\begin{subfigure}[b]{0.30\textwidth}
	\frame{\includegraphics[width=\textwidth]{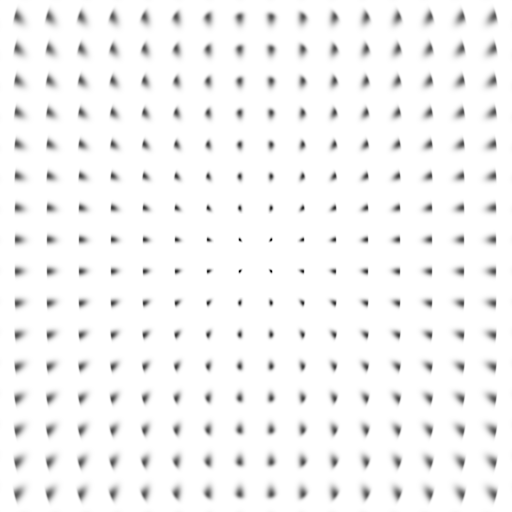}}
	\caption{} \label{fig:kernel_rotation_skew}
\end{subfigure}
	\caption{PSFs maps used in the paper. The PSFs in Figure \ref{fig:kernel1} are Gaussians with equal variances increasing in the vertical direction. The PSFs in Figure \ref{fig:kernelRotation} are anisotropic Gaussians with covariance matrices that depend on the polar coordinates. The PSFs in Figure \ref{fig:kernel_rotation_skew} are based on paper \cite{simpkins2014parameterized}.} \label{fig:test_kernels}
\end{figure}

\FloatBarrier
\subsubsection{Computation of the full $\bTheta$ matrix}
Before applying our approximation methods, matrix $\bTheta$ needs to be computed explicitly. 
The coefficients $\dotproduct{H \psi_{\lambda}}{\psi_{\mu}}$ are approximated by their discrete counterparts.
If $\bpsi_{\lambda}$ and $\bpsi_{\mu}$ denote discrete wavelets, we simply compute the wavelet transform of $\bH \bpsi_{\lambda}$ and store it into the $\lambda$-th column of $\bTheta$.
This computation scheme is summarized in Algorithm \ref{algo:btheta}. 
This algorithm corresponds to the use of rectangle methods to evaluate the dot-products:
\begin{equation}
	\int_{\Omega} \int_{\Omega} K(x,y) \psi_{\lambda}(y) \psi_{\mu}(x) dy dx \simeq \frac{1}{N^{2d}} \sum_{x \in X} \sum_{y \in X} K(x,y) \psi_{\lambda}(y) \psi_{\mu}(x).
\end{equation}

\begin{algorithm}[htp]
\SetKwInput{KwInit}{Initialization}
\DontPrintSemicolon
\SetAlgoLined
  \KwOut{\\ $\bTheta$: the full matrix of $\bH$}
  \Begin{
     \ForAll{$\lambda$}{
	Compute the wavelet $\bpsi_{\lambda}$ using an inverse wavelet transform\\
	Compute the blurred wavelet $\bH \bpsi_{\lambda}$ \\
	Compute $\left(\dotproduct{\bH \bpsi_{\lambda}}{\bpsi_{\mu}}\right)_{\mu}$ using one forward wavelet transform \\
	Set $\left(\dotproduct{\bH \bpsi_{\lambda}}{\bpsi_{\mu}}\right)_{\mu}$ in the $\lambda$-th column of $\bTheta$.
    }
  }
\caption{An algorithm to compute $\bTheta$}\label{algo:btheta}
\end{algorithm}

\subsection{Application to direct problems}\label{sec:direct_prob}

In this section, we investigate the approximation properties of the proposed approaches in the aim of computing matrix-vector products.
In all numerical experiments, we use an orthogonal wavelet transform with 4 decomposition levels. We always use Daubechies wavelets.

\subsubsection{Influence of vanishing moments} \label{sec:numerics_moments}

First of all we demonstrate the influence of vanishing moments on the quality of approximations. For each number of vanishing moments $M \in \set{1,2,4,6,10}$, a sparse approximation $\wtilde{\bH}$ is constructed by thresholding $\bTheta$, keeping the $K = l \times N$ largest coefficients with $l \in \set{0\ldots 40}$. Then for each $\bu$ in the set of 16 images, we compare $\wtilde{\bH} \bu$ to $\bH \bu$ computing the pSNR. We then plot the average of pSNRs over the set of images with respect to the number of operations needed for a matrix-vector product.
The results of this experiment are displayed in Figure \ref{fig:xp_moments}. It appears that for the considered operators, using as many vanishing moments as possible was preferable. 
Using more than 10 vanishing moments however led to insignificant performance increase while making the numerical complexity higher. Therefore, in all the following numerical experiments we will use Daubechies wavelets with 10 vanishing moments. Note that paper \cite{wei2014fast} only explored the use of Haar wavelets. This experiment shows that very significant improvements can be obtained by leveraging regularity of the integral kernel using vanishing moments. The behavior was predicted by Theorem \ref{thm:proof_thresh}.

\begin{figure}[htp]
\centering
\input{figureTIKZ/Figure_8_Classique_1_Daubechies_20_fKernel_6.tex}
\caption{pSNR of the blurred image using the approximated operator $\wtilde{\bH} \bu$ with respect to the blurred image using the exact operator $\bH \bu$. pSNRs have been averaged over the set of test images. Daubechies wavelets have been used with different number vanishing moments $M \in \set{1,2,4,6,10}$. The case $M=1$ corresponds to Haar wavelets.} \label{fig:xp_moments}
\end{figure}
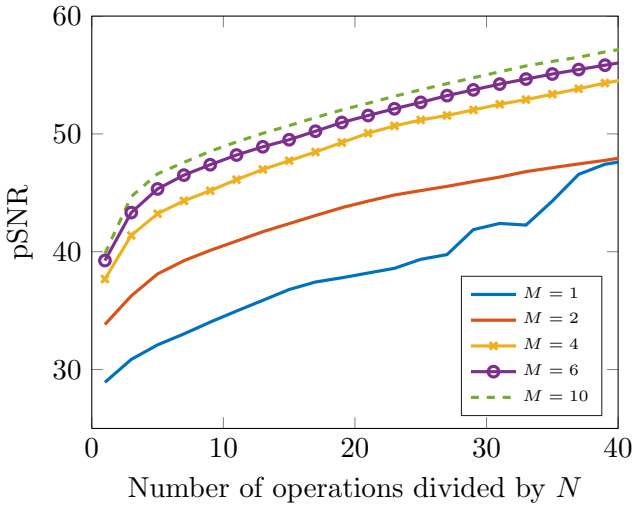

\subsubsection{Comparison of different methods} \label{sec:comparison_direct}

\paragraph{Wavelets VS windowed convolutions.} In this first numerical experiment, we evaluate $\norm{\bH - \wtilde \bH}_{2 \to 2}$ where $\wtilde \bH$ is obtained by windowed convolutions method or sparse approximations in the wavelet domain. 

The sparse approximation of the operator is constructed by thresholding the matrix $\bTheta$ in order to keep the $K$ largest coefficients. 
We have set $K = 2^{l} \times N$ with $l \in \set{ 0 \ldots 2\log_2 N }$. This way $K$ is a multiple of the number of pixels in the image.
The windowed convolutions method is constructed by partitioning the image into $2^l \times 2^l$ sub-images where $l \in \set{ 1 \ldots \log_2 N }$. We also studied the case where sub-images overlap and linearly interpolated the blur between sub-images as proposed in \cite{nagy1997fast,hirsch2010efficient}. 
The overlap has been fixed to $ 50\% $ of the sub-images sizes. 

For each sub-image size, and each overlap, the norm $\norm{\bH - \wtilde \bH}_{2 \to 2}$ is approximated using a power method \cite{golub2012matrix}. We stop the iterative process when the difference between the eigenvalues of two successive iterations is smaller than $10^{-8} \|\bH\|_{2\to 2}$. The number of operations associated to each type of approximation is computed using theoretical complexities. For sparse matrix-vector product the number of operations is proportional to the number of non-zero coefficients in the matrix. For windowed convolutions methods, the number of operations is proportional to the number of windows ($2^l\times 2^l$) multiplied by the cost of a discrete convolution over a window $\left(\frac{N}{2^l} + N\kappa\right)^2 \log_2\left(\frac{N}{2^l} + N\kappa\right)$. 

Figure \ref{fig:normOperator} shows the results of this experiment. The wavelet based method seems to perform much better than windowed convolutions methods for both operators. The gap is however significantly larger for the rotation blur in Figure \ref{fig:kernelRotation}. This experiment therefore suggests that the advantage of wavelet based approaches will depend on the type of blur considered.

\begin{figure}[htp]
\centering
\begin{subfigure}[b]{0.48\textwidth}
\input{figureTIKZ/fKernel_6_ops_loglog.tex}
\end{subfigure}
\begin{subfigure}[b]{0.48\textwidth}
\input{figureTIKZ/fKernel_rotation_ops_loglog.tex}
\end{subfigure}
\begin{subfigure}[b]{0.48\textwidth}
\input{figureTIKZ/fKernel_rotation_skew_ops_loglog.tex}
\end{subfigure}
	\caption{The operator norms $\norm{\bH - \wtilde \bH}_{2 \to 2}$ are displayed for the three proposed kernels. (Left: kernel Figure \ref{fig:kernel1}, middle: kernel in Figure \ref{fig:kernelRotation}, right: kernel in Figure \ref{fig:kernel_rotation_skew}). Norms are plotted with respect to the number of operations needed to compute $\wtilde \bH \bu$. The abscissas are in log scale.} \label{fig:normOperator}
\end{figure}
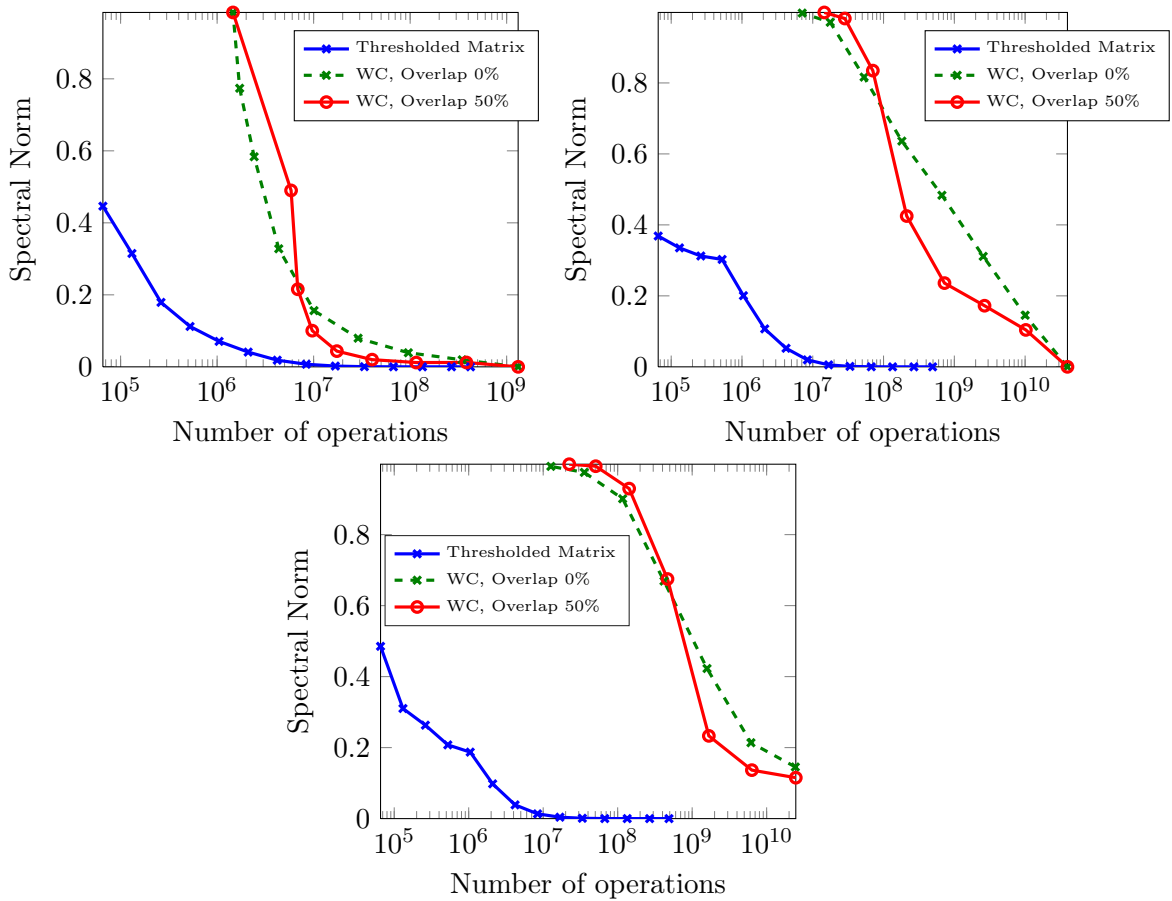

\paragraph{The influence of sparsity patterns.} In this numerical experiment, we obtain a $K$-sparse matrix $\bTheta_K$ using either a simple thresholding strategy, Algorithm \ref{algo:thresh} or Algorithm \ref{algo:thresh2}. We evaluate the error $\norm{\bH - \wtilde \bH}_{X \to 2}$ defined in \eqref{eq:normX-2} for each methods. We set $\sigma_i=2^{j(i)}$, where $j(i)$ corresponds to the scale of the $i$-th wavelet. As can be seen from Figure \ref{fig:normBesovOperator}, Algorithm \ref{algo:thresh} provides a much better error decay for each operator than the simple thresholding strategy. 
This fact will be verified for real images in next paragraph.
Algorithm \ref{algo:thresh2} has a much slower decay than both thresholding algorithm. 
Notice that this algorithm is essentially blind, in the sense that it does not require knowing the exact matrix $\bTheta$ to select the pattern. 
It would therefore work for a whole class of blur kernels, whereas the simple thresholding and Algorithm \ref{algo:thresh} work only for a specific matrix.

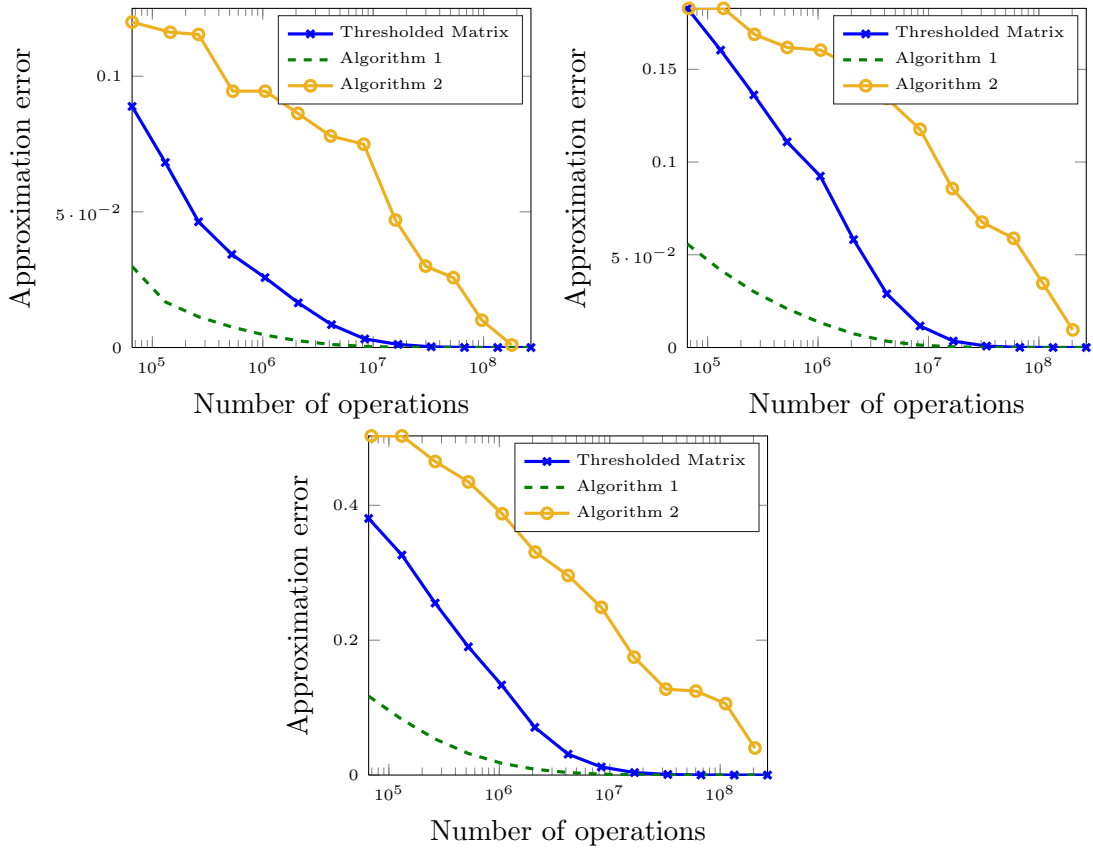
\begin{figure}[htp]
\centering
\begin{subfigure}[b]{0.48\textwidth}
\input{figureTIKZ/fKernel_6_ops_Besov_loglog.tex}
\end{subfigure}
\begin{subfigure}[b]{0.48\textwidth}
\input{figureTIKZ/fKernel_rotation_ops_Besov_loglog.tex}
\end{subfigure}

\begin{subfigure}[b]{0.48\textwidth}
\input{figureTIKZ/fKernel_rotation_skew_ops_Besov_loglog.tex}
\end{subfigure}
	\caption{The operator norms $\norm{\bH - \wtilde \bH}_{X \to 2}$ are displayed for kernels Figure \ref{fig:kernel1} (left) and Figure \ref{fig:kernelRotation} (right); and with respect to the number of operations needed to compute $\wtilde \bH u$. The abscissas are in log scale. Daubechies wavelets with 10 vanishing moments have been used.} \label{fig:normBesovOperator}
\end{figure}

Figure \ref{fig:structures} shows the sparsity patterns of matrices obtained with Algorithms \ref{algo:thresh} and \ref{algo:thresh2} for $K=30 N$ and $K=128 N$ coefficients. The sparsity patterns look quite similar.
However, Algorithm \ref{algo:thresh} selects subbands that are not selected by Algorithm \ref{algo:thresh2}, which might explain the significant performance differences. 
Similarly, Algorithm \ref{algo:thresh2} select subbands that would probably be crucial for some blur kernels, but which are not significant for this particular blur kernel.

\begin{figure}[htp]
\centering
\begin{subfigure}[b]{0.45\textwidth}\includegraphics[width=\textwidth]{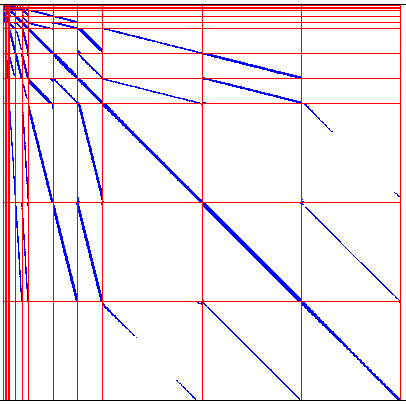} 
\caption{Algorithm \ref{algo:thresh} -- $K = 30 N$}
\end{subfigure}
\quad
\begin{subfigure}[b]{0.45\textwidth}\includegraphics[width=\textwidth]{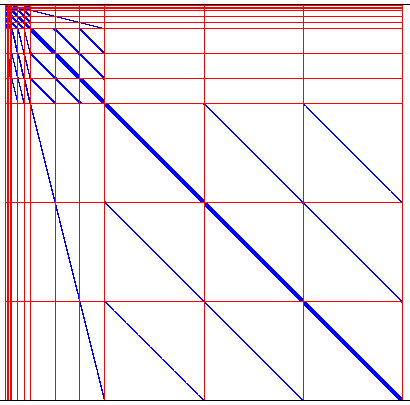}
\caption{Algorithm \ref{algo:thresh2} -- $K = 30 N$}
\end{subfigure}

\begin{subfigure}[b]{0.45\textwidth}\includegraphics[width=\textwidth]{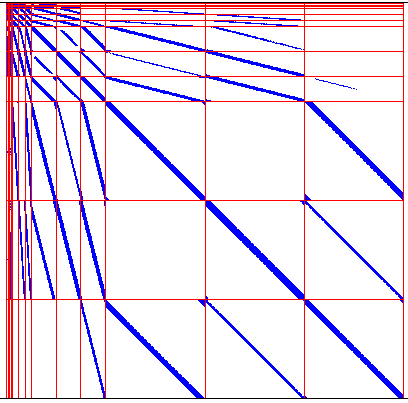}
\caption{Algorithm \ref{algo:thresh} -- $K = 128 N$}
\end{subfigure}
\quad
\begin{subfigure}[b]{0.45\textwidth}\includegraphics[width=\textwidth]{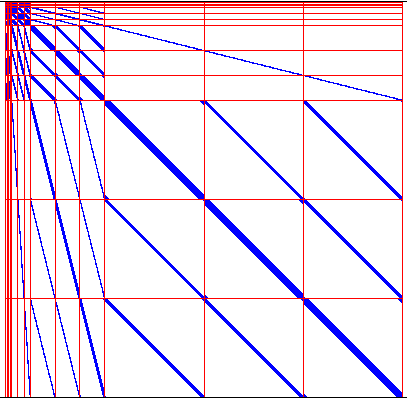}
\caption{Algorithm \ref{algo:thresh2} -- $K = 128 N$}
\end{subfigure}
\caption{The structure of the wavelet matrices of $\bTheta_K$ are displayed for Algorithms \ref{algo:thresh} and \ref{algo:thresh2} and for $K = 30 N$ and $K = 128 N$ coefficients. Algorithm \ref{algo:thresh} has been applied using the second $\bSigma = \textrm{diag}(2^{j(i)})_i$ matrix. This experiment corresponds to the blur in Figure \ref{fig:kernelRotation}} \label{fig:structures}
\end{figure}

\FloatBarrier
\subsubsection{Quality of matrix-vector products for real images}

In this section, we evaluate the performance of wavelet based methods for matrix-vector products with real images. 

\paragraph{Quality VS complexity.}
We compare $\wtilde \bH \bu$ to $\bH \bu$, where $\bu$ is the image in Figure \ref{im:letters} and where $\wtilde \bH$ is obtained either by windowed convolutions methods or by sparse wavelet approximations. 
We plot the pSNR between the exact blurred image $\bH \bu$ and the blurred image using the approximated operator $\wtilde \bH \bu$ in Figure \ref{fig:pSNR_Direct}. Different approximation methods are tested:
\begin{description}[font=\normalfont]
 \item[\textit{Thresholded matrix}:] This corresponds to a simple thresholding of the wavelet matrix $\bTheta$.
 \item[\textit{$\bSigma$ n$^\circ$1}:] This corresponds to applying Algorithm \ref{algo:thresh} with $\sigma_{i} = 1, \ \forall i$ where $j(i)$ corresponds to the scale of the $i$-th wavelet.
 \item[\textit{$\bSigma$ n$^\circ$2}:] This corresponds to applying Algorithm \ref{algo:thresh} with $\sigma_{i} = 2^{j(i)} \ \forall i$.
 \item[\textit{\cite{wei2014fast}}:] The method presented in \cite{wei2014fast} with  $K = l \times N$ coefficients in the matrix, with $l\in \set{1,\ldots,100}$.
 \item[\textit{WC, Overlap 50\%}:] This corresponds to the windowed convolution with 50\% overlap. We use this overlap since it produces better pSNRs.
 \item[\textit{Algo \ref{algo:thresh2}}:]  The algorithm finds multi-scale neighbourhoods until $K = l \times N$ coefficients populate the matrix, with $l\in \set{1,\ldots,100}$. In this experiment, we set $M=1$, $f(t)=\frac{1}{1+t}$ and $\sigma_{i} = 2^{j(i)}, \ \forall i$.
\end{description}

The pSNRs are averaged over the set of 16 images. The results of this experiment are displayed in Figure \ref{fig:pSNR_Direct} for the two kernels from Figures \ref{fig:kernelRotation} and \ref{fig:kernel1}. 
Let us summarize the conclusions from this experiment: 
\begin{itemize}
 \item A clear fact is that windowed convolution methods are significantly outperormed by wavelet based methods for all blur kernels. Moreover, the differences between wavelet and windowed convolution based methods get larger as the blurs regularity decreases.  
 \item A second result is that wavelet based methods with fixed sparsity patterns (Algo \ref{algo:thresh2}) are quite satisfactory for very sparse patterns (i.e. less than $20N$ operations) and kernels \ref{fig:kernel1} and \ref{fig:kernelRotation}. We believe that the most important regime for applications is in the range $[N,20N]$, so that this result is rather positive. However, Algo \ref{algo:thresh2} suffers from two important drawbacks: first, the increase in SNR after a certain value becomes very slow. Second, this algorithm provides very disappointing results for the last blur map \ref{fig:kernel_rotation_skew}. These results suggest that this method should be used with caution if one aims at obtaining very good approximations. In particular, the algorithm is dependent on the bound \eqref{eq:decay} which itself depends on user given parameters such as function $f$ in \eqref{def:blurring_operators:PSFsmoothness}. Modifying those parameters might result in better results, but is usually hard to tweak manually.
 \item The methods $\bSigma$ n$^\circ$ 1, $\bSigma$ n$^\circ$ 2, \textit{Thresholded matrix} all behave similarly. Method $\bSigma$ n$^\circ$ 1 is however significantly better, showing the importance of choosing the weights $\sigma_i$ in equation \eqref{eq:finalthreshprob} carefully.
 \item The methods $\bSigma$ n$^\circ$ 1, $\bSigma$ n$^\circ$ 2, \textit{Thresholded matrix} outperform the method proposed in \cite{wei2014fast} for very sparse patterns ($<20N$) and get outperformed for mid-range sparisfication $>40N$. The main difference between algorithm \cite{wei2014fast} and the methods proposed in this paper is the number of vanishing moments. In \cite{wei2014fast}, the authors propose using the Haar wavelet (i.e. 1 vanishing moment), while we use Daubechies wavelets with 10 vanishing moments. In practice, this results in better approximation properties in the very sparse regime, which might be the most important in applications. For mid-range sparsification, the Haar wavelet provides better results. Two reasons might explain this phenomenon: first, Haar wavelets have a small spatial support, therefore matrix $\bTheta$ contains less non-zero coefficients when expressed with Haar wavelets than Daubechies wavelets. Second, the constants $C_M'$ and $C_M''$ in Theorem \eqref{thm:proof_thresh} are increasing functions of the number of vanishing moments. 
\end{itemize}

\paragraph{Illustration of artefacts.}

Figure \ref{fig:comparisonPC_WB} provides a comparison of the windowed convolutions methods and the wavelet based approach in terms of approximation quality and computing times. 
The following conclusions can be drawn from this experiment:
\begin{itemize}
 \item The residual artefacts appearing in the windowed convolutions approach and wavelet based approach are different. They are localized at the interfaces between sub-images for the windowed convolutions approach while they span the whole image domain for the wavelet based approach. It is likely that using translation and/or rotation invariant wavelet would improve the result substantially.
 \item The approximation using the second $\bSigma$ matrix produces the best results and should be preferred over more simple approaches.
 \item In our implementation, the windowed convolutions approach (implemented in C) is  outperformed by the wavelet based method (implemented in Matlab with C-mex files). For instance, for a precision of 45dBs, the wavelet based approach is about 10 times faster. 
 \item The computing time of $1.21$ seconds for the windowed convolutions approach with a $2\times 2$ partition might look awkward since the computing times are significantly lower for finer partitions. This is because the efficiency of FFT methods depend greatly on the image size. The time needed to compute an FFT is usually lower for sizes that have a prime factorization comprising only small primes (e.g. less than 7). This phenomenon explains the fact that the practical complexity of windowed convolutions algorithms may increase in a chaotic manner with respect to $m$.
\end{itemize}

\begin{figure}[htp]
\centering
\begin{subfigure}[b]{0.42\textwidth}
 \input{figureTIKZ/Figure_8_Classique_4_fKernel_6.tex} 
\end{subfigure}
\qquad\qquad
\begin{subfigure}[b]{0.42\textwidth}
\input{figureTIKZ/Figure_8_Classique_4_fKernel_rotation.tex} 
\end{subfigure}

\begin{subfigure}[b]{0.45\textwidth}
\input{figureTIKZ/Figure_8_Classique_4_fKernel_rotation_skew.tex}
\end{subfigure}
\caption{pSNR of the blurred image using the approximated operators $\wtilde \bH \bu$ with respect to the blurred image using the exact operator $\bH \bu$. 
The results have been obtained with blur Figure \ref{fig:kernel1} for top-left graph, blur Figure \ref{fig:kernelRotation} for top-right graph and blur Figure \ref{fig:kernel_rotation_skew} for the bottom. pSNR are averaged over the set of 16 images.} \label{fig:pSNR_Direct}
\end{figure}
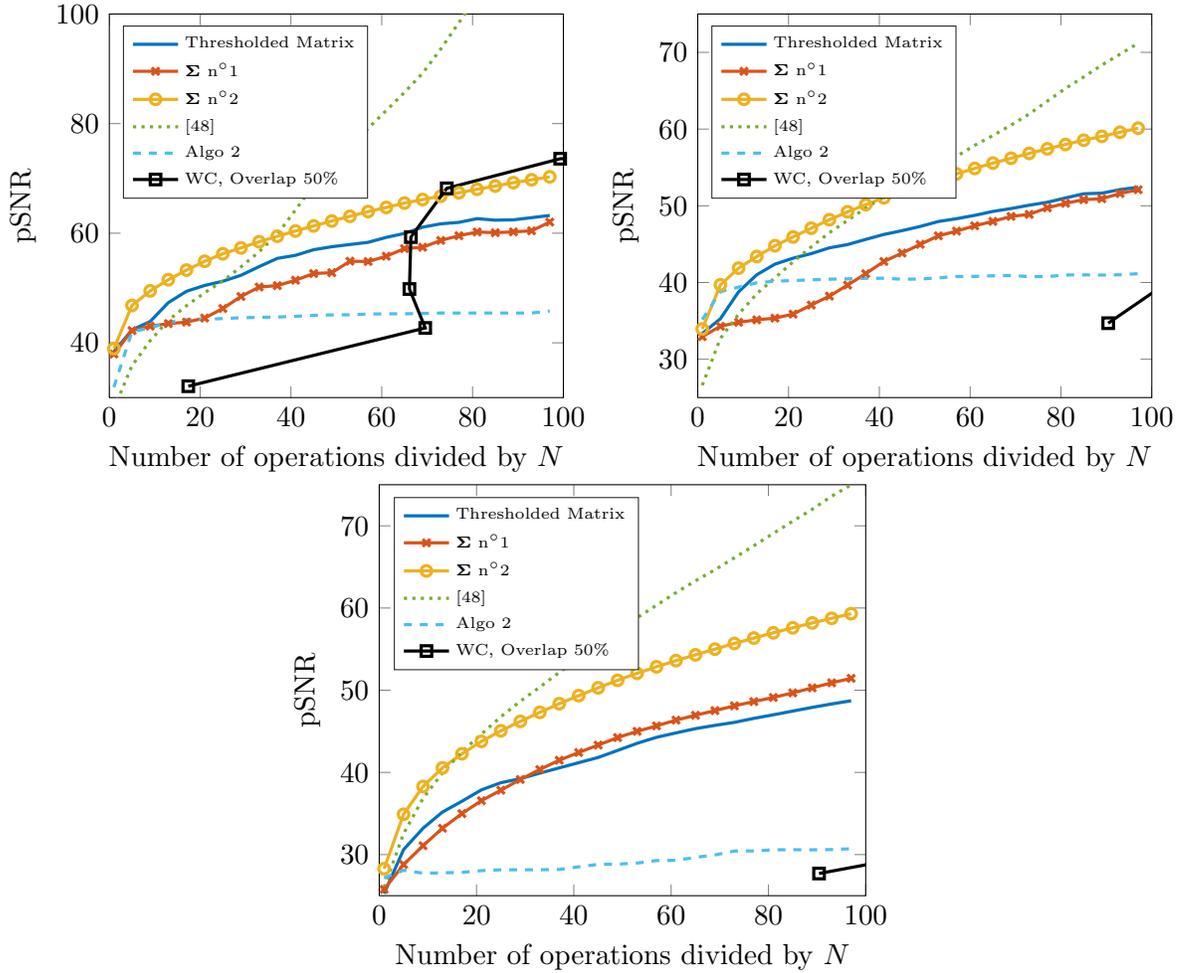

\begin{figure}[htbp]
	\centering
	\begin{tabular}{l|c c || c c |l}
	&  Piece. Conv. & Difference & Algorithm \ref{algo:thresh} & Difference & $l=$\\ \hline
	$2 \times 2$ &  31.90 dB & & 36.66 dB & & 5\\
	 1.21 sec & \includegraphics[width=0.17\textwidth]{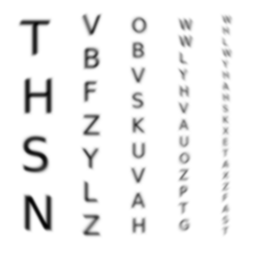}  & \includegraphics[width=0.17\textwidth]{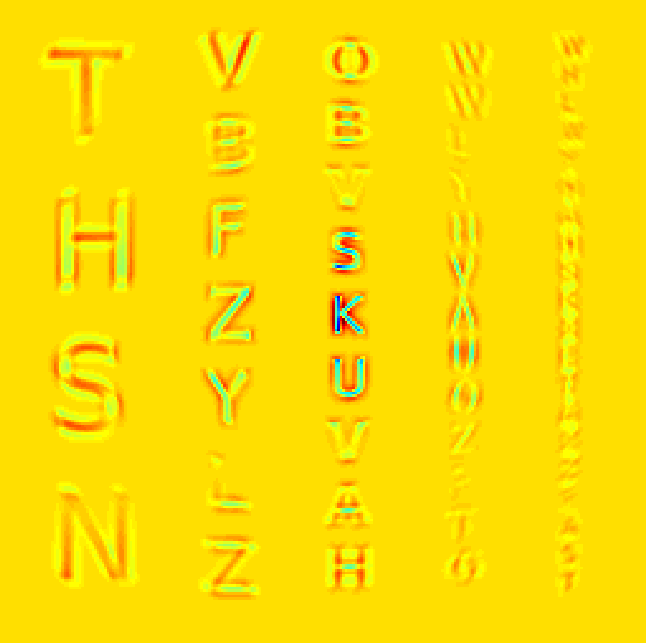} &  \includegraphics[width=0.17\textwidth]{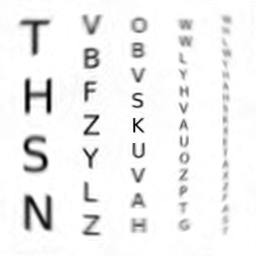} & \includegraphics[width=0.17\textwidth]{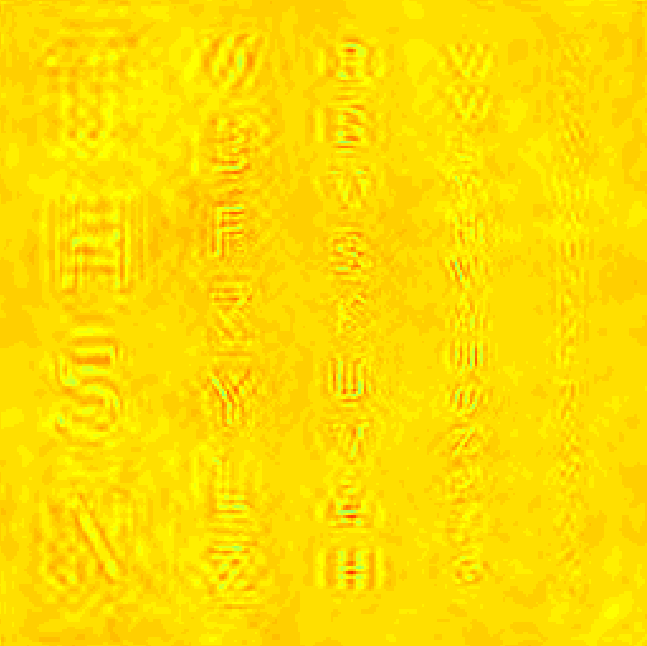} & 0.039 sec \\ \hline 
	$4 \times 4$ &  38.49 dB & & 45.87 dB &  & 30 \\
	 0.17 sec & \includegraphics[width=0.17\textwidth]{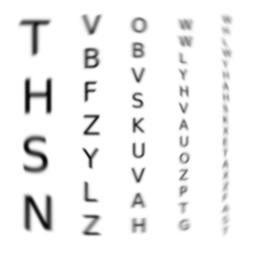} & \includegraphics[width=0.17\textwidth]{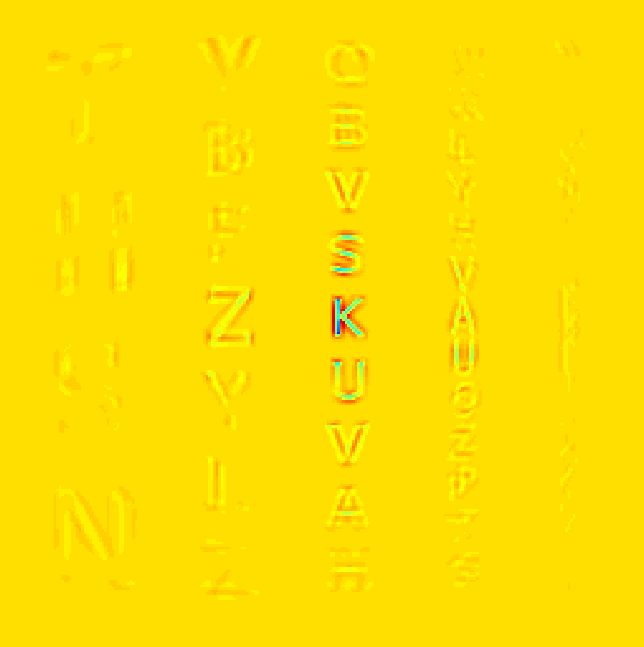} & \includegraphics[width=0.17\textwidth]{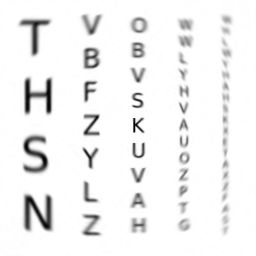} & \includegraphics[width=0.17\textwidth]{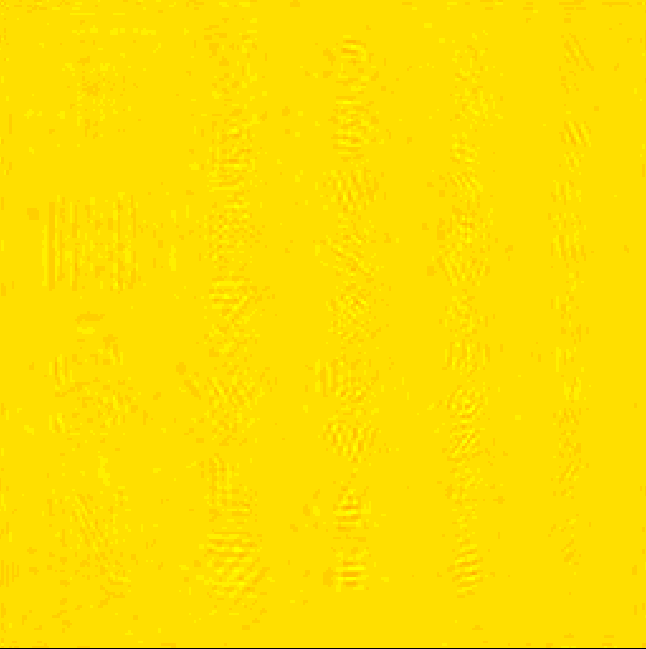} & 0.040 sec \\ \hline 
	$8 \times 8$ &  44.51 dB & & 50.26 dB & & 50\\
		0.36 sec & \includegraphics[width=0.17\textwidth]{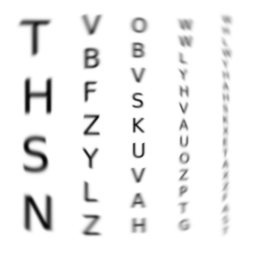} & \includegraphics[width=0.17\textwidth]{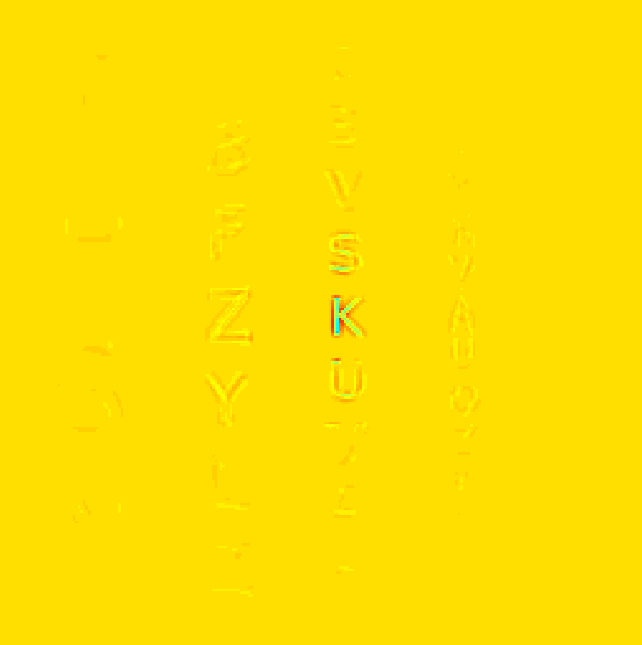} & \includegraphics[width=0.17\textwidth]{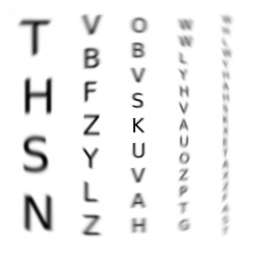} & \includegraphics[width=0.17\textwidth]{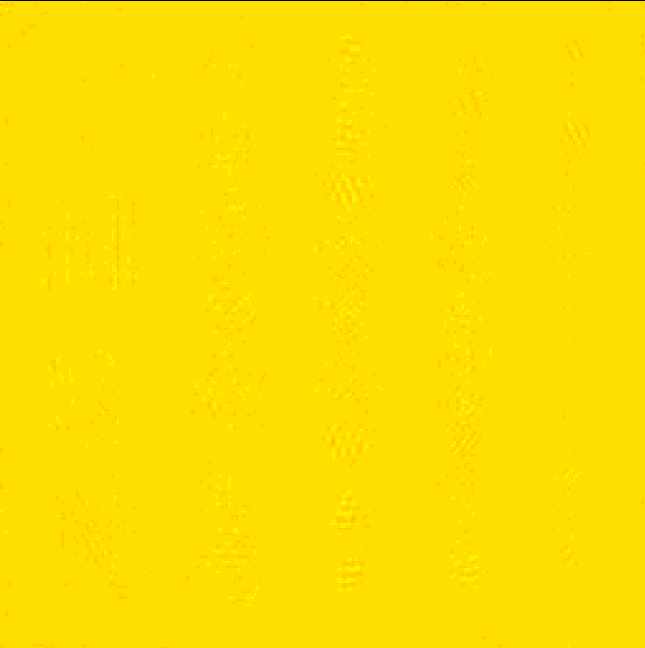} & 0.048 sec \\ \hline 
		$16 \times 16$ &  53.75 dB & & 57.79 dB & & 100\\
		0.39 sec & \includegraphics[width=0.17\textwidth]{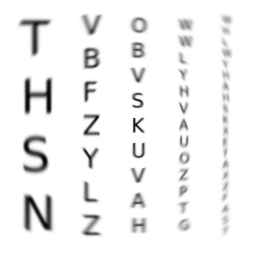} & \includegraphics[width=0.17\textwidth]{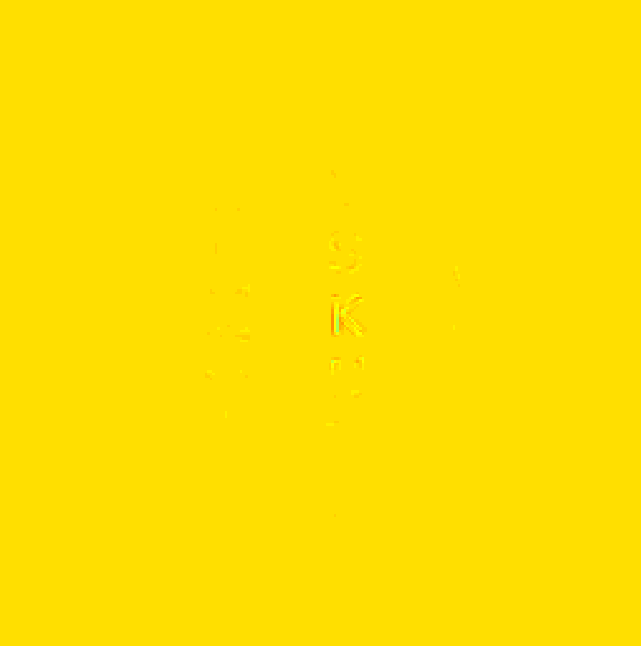} & \includegraphics[width=0.17\textwidth]{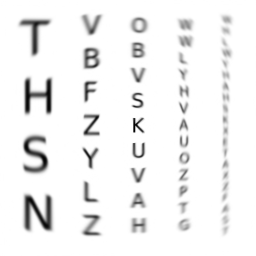} & \includegraphics[width=0.17\textwidth]{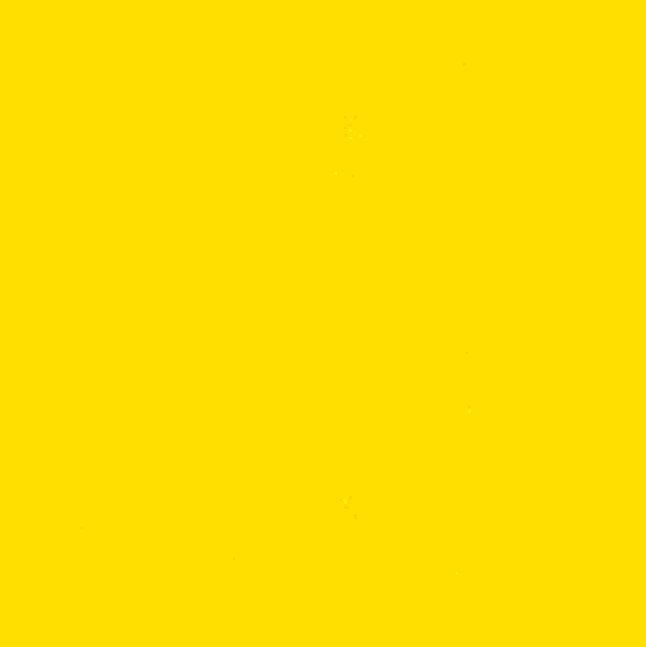} & 0.058 sec \\ \hline 
	\end{tabular}
	\caption{Blurred images and the differences $\bH \bu - \wtilde \bH \bu$ for the kernel Figure \ref{fig:kernelRotation}. Results on the left are obtained using windowed convolutions approximations with $2 \times 2$, $4 \times 4$, $8 \times 8$ and $16 \times 16$ partitionings all with 50\% overlap. Results on the right are obtained using Algorithm \ref{algo:thresh} with the second $\bSigma = \textrm{diag}(2^{j(i)})_i$ matrix keeping $K = l N$ coefficients. The pSNR and the time needed for the computation for the matrix-vector product are shown. }
	\label{fig:comparisonPC_WB}
\end{figure}

\FloatBarrier

\subsection{Application to inverse problems}\label{sec:Inverse_Problems}

In this experiment we 
compare the methods efficiency in deblurring problems. We assume the following classical image degradation model
\begin{equation}
	\bv = \bH \bu + \bbeta, \quad \bbeta \sim \mathcal{N}\left(0, \sigma^2 \textrm{Id}\right),
\end{equation}
where $\bv$ is the degraded image observed, $\bu$ is the image to restore, $\bH$ in the blurring operator and $\sigma^2$ is the noise variance. A standard TV-L2 optimization problem is solved to restore the image $\bu$:
\begin{equation}
	\text{Find } \bu^* \in \argmin_{\bu \in \R^{N}, \norm{ \wtilde \bH \bu -\bv }_2^2 \leq \alpha} TV(\bu),
\end{equation}
where $\wtilde \bH$ is an approximating operator and $TV$ is the isotropic total variation of $\bu$. The optimization problem is solved using the primal-dual algorithm proposed in \cite{chambolle2011first}. We do not detail the resolution method since it is now well documented in the literature.

An important remark is that the interest of the total variation term is not only to regularize the ill-posed inverse problem, but also to handle the errors in the operator approximation. In practice we found that setting $\alpha = (1 + \epsilon) \sigma^2 N$ where $\epsilon > 0$ is a small parameter provides good experimental results. 

In Figures \ref{fig:deconv_letters_fKernelRotation_0} and \ref{fig:deconv_letters_fKernelRotation_2e-2}, we present deblurring results using Figure \ref{im:letters} with kernel \ref{fig:kernelRotation}. 

In both the noisy and noiseless cases, the $4 \times 4$ windowed convolutions method performs worst reconstructions than wavelet approaches with $30N$. Moreover, they are between 4 and 6 times significantly slowlier. Surprisingly even the implementation in the space domain is faster. The reason for that is probably a difference in the quality of implementation: we use Matlab sparse matrix-vector products for space and wavelet methods. This routine is cautiously optimized while our c implementation of windowed convolutions can probably be improved. 
In addition, let us mention that two wavelet transforms need to be computed at each iteration with the wavelet based methods, while this is not necessary with the space implementation. It is likely that the acceleration factor would have been significantly higher if wavelet based regularizations had been used.

In the noiseless case, the simple thresholding approach provides significantlty better SNRs than the more advanced proposed in this paper and in \cite{wei2014fast}. Note however that it produces more significant visual artefacts. 
This result might come as a surprise at first sight. 
However, as was explained in section \ref{sec:algorithms}, our aim to design sparsity patterns was to minimize an operator norm $\|\bH - \wtilde \bH\|_{X\to 2}$. When dealing with an inverse problem, approximating the direct operator is not as relevant as approximating its inverse. 
This calls for new methods specific to inverse problems.

In the noisy case, all three thresholding strategies produce results of a similar quality. The Haar wavelet transform is however about twice faster since the Haar wavelet support is smaller. Moreover, the results obtained with the approximated matrices are nearly as good as the ones with the true operator. 
It suggests that it is not necessary to construct accurate approximations of the operators in practical problems. This observation is also supported by the experiment in Figure \ref{fig:deconv_pSNR_vs_ops}. In this experiment, we plot the pSNR of the deblurred image in presence of noise with respect to the number of elements in $\bTheta_K$. Interestingly, a matrix containing only $20 N$ coefficients leads to deblurred images close to the results obtained with the exact operator. In this experiment, a total of $K=5N$ coefficients in $\bTheta_K$ is enough to retrieve satisfactory results. This is a very encouraging result for blind deblurring problems.


\begin{figure}[htbp]
\centering
\begin{subfigure}[b]{0.40\textwidth} \centering
\includegraphics[width=0.8\textwidth]{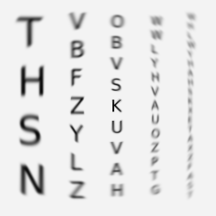} 
\caption{Degraded image \\ 21.85dB }
\end{subfigure}
\begin{subfigure}[b]{0.40\textwidth} \centering
\includegraphics[width=0.8\textwidth]{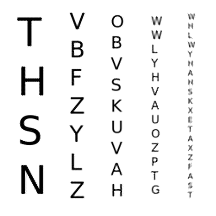} 
\caption{Exact operator \\ 34.53dB -- 64.87 sec}
\end{subfigure}

\begin{subfigure}[b]{0.40\textwidth} \centering
\includegraphics[width=0.8\textwidth]{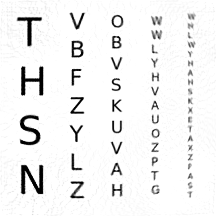} 
\caption{Simple thresh\\ 31.68dB -- 21.68 sec}
\end{subfigure}
\begin{subfigure}[b]{0.40\textwidth} \centering
\includegraphics[width=0.8\textwidth]{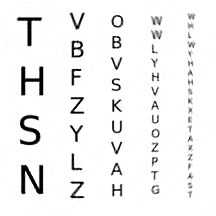} 
\caption{Algorithm \ref{algo:thresh} \\ 30.57dB -- 21.16 sec}
\end{subfigure}

\begin{subfigure}[b]{0.40\textwidth} \centering
\includegraphics[width=0.8\textwidth]{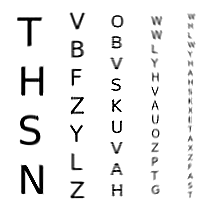} 
\caption{WC $4\times4$  \\ 28.37dB -- 85.60 sec}
\end{subfigure}
\begin{subfigure}[b]{0.40\textwidth} \centering
\includegraphics[width=0.8\textwidth]{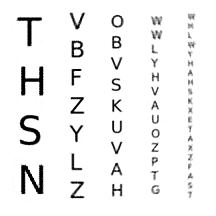} 
\caption{ \cite{wei2014fast} \\ 30.53dB -- 14.12 sec}
\end{subfigure}
\caption{\small Deblurring results for kernel Figure \ref{fig:kernelRotation} and without noise. Top-left: degraded image. Top-right: deblurred using the exact operator. Middle-left: deblurred by the wavelet based method and a simple thresholding. Middle-right: deblurred by the wavelet based method and Algorithm \ref{algo:thresh2} with the second $\bSigma = \textrm{diag}(2^{j(i)})_i$ matrix. Bottom: deblurred using a $4 \times 4$ windowed convolutions algorithm with 50\% overlap. For wavelet methods $K = 30 N$ coefficients are kept in matrices. pSNR are displayed for each restoration.}\label{fig:deconv_letters_fKernelRotation_0}
\end{figure}

\begin{figure}[htbp]
\centering
\begin{subfigure}[b]{0.40\textwidth} \centering
\includegraphics[width=0.8\textwidth]{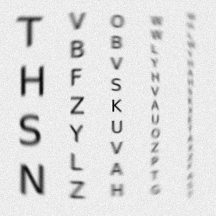} 
\caption{Degraded image \\ 21.62dB }
\end{subfigure}
\begin{subfigure}[b]{0.40\textwidth} \centering
\includegraphics[width=0.8\textwidth]{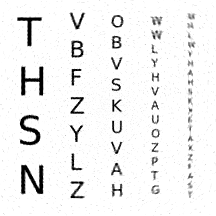} 
\caption{Exact operator \\ 29.09dB -- 64.87 sec}
\end{subfigure}

\begin{subfigure}[b]{0.40\textwidth} \centering
\includegraphics[width=0.8\textwidth]{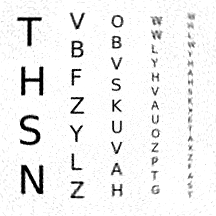} 
\caption{Simple thresh \\ 28.64dB -- 21.68 sec}
\end{subfigure}
\begin{subfigure}[b]{0.40\textwidth} \centering
\includegraphics[width=0.8\textwidth]{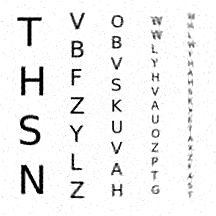} 
\caption{Algorithm \ref{algo:thresh} \\ 28.24dB -- 21.16 sec}
\end{subfigure}

\begin{subfigure}[b]{0.40\textwidth} \centering
\includegraphics[width=0.8\textwidth]{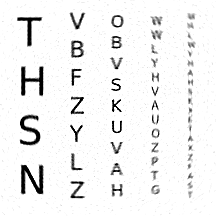} 
\caption{WC $4\times4$  \\ 27.62dB -- 85.60 sec}
\end{subfigure}
\begin{subfigure}[b]{0.40\textwidth} \centering
\includegraphics[width=0.8\textwidth]{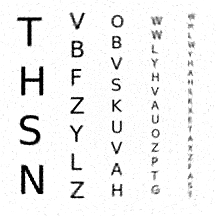} 
\caption{\cite{wei2014fast} \\ 28.37dB -- 14.12 sec}
\end{subfigure}
\caption{\small Deblurring results for kernel Figure \ref{fig:kernelRotation} and with $\sigma=0.02$ noise. Top-left: degraded image. Top-right: deblurred using the exact operator. Middle-left: deblurred by the wavelet based method and a simple thresholding. Middle-right: deblurred by the wavelet based method and Algorithm \ref{algo:thresh2} with the second $\bSigma = \textrm{diag}(2^{j(i)})_i$ matrix. Bottom: deblurred using a $4 \times 4$ windowed convolutions algorithm with 50\% overlap. For wavelet methods $K = 30 N$ coefficients are kept in matrices. pSNR are displayed for each restoration.} \label{fig:deconv_letters_fKernelRotation_2e-2}
\end{figure}


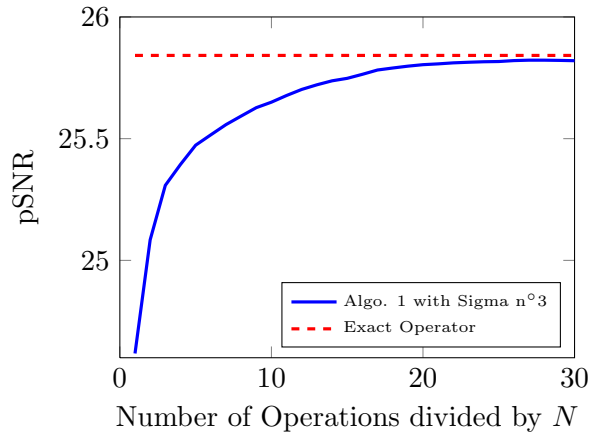
\begin{figure}[htbp]
\centering
\input{figureTIKZ/pSNR_vs_ops.tex}
\caption{pSNR of the deblurred image with respect to the number of coefficients in the matrix divided by $N$ for the image Figure \ref{im:mandrill} and the kernel Figure \ref{fig:kernel1}. The matrix is constructed using Algorithm \ref{algo:thresh} with the second $\bSigma = \textrm{diag}(2^{j(i)})_i$ matrix with $K = l N$ coefficients for $l$ from $1$ to $30$. Deblurred imaged using these matrices are compared with the one obtained with the exact operator.} \label{fig:deconv_pSNR_vs_ops}
\end{figure}

%% file: figureTIKZ/Figure_8_Classique_1_Daubechies_20_fKernel_6.tex
%
%
\definecolor{mycolor1}{rgb}{0.00000,0.44700,0.74100}%
\definecolor{mycolor2}{rgb}{0.85000,0.32500,0.09800}%
\definecolor{mycolor3}{rgb}{0.92900,0.69400,0.12500}%
\definecolor{mycolor4}{rgb}{0.49400,0.18400,0.55600}%
\definecolor{mycolor5}{rgb}{0.46600,0.67400,0.18800}%
\begin{tikzpicture}

\begin{axis}[%
width=0.463\textwidth,
height=0.363\textwidth,
at={(0\textwidth,0\textwidth)},
scale only axis,
xmin=0,
xmax=40,
xlabel={Number of operations divided by $N$},
ymin=25,
ymax=60,
ylabel={pSNR},
axis background/.style={fill=white},
legend style={at={(0.97,0.03)},anchor=south east,legend cell align=left,align=left,draw=white!15!black, font=\tiny}
]
\addplot [color=mycolor1,solid,very thick]
  table[row sep=crcr]{%
1	28.9135140948629\\
3	30.8597582706318\\
5	32.0923117177178\\
7	33.0231235190368\\
9	34.024884078119\\
11	34.9639878468593\\
13	35.874766422109\\
15	36.7913804292298\\
17	37.4298161465535\\
19	37.789024263118\\
21	38.1927312880622\\
23	38.5924413527526\\
25	39.3558832198913\\
27	39.7561073196713\\
29	41.8754731508077\\
31	42.4025050581327\\
33	42.2611765803974\\
35	44.2999524238007\\
37	46.5682656442241\\
39	47.4224771247898\\
41	47.7886851699113\\
43	51.7250716871411\\
45	56.0592292617529\\
47	58.4289228960898\\
49	59.9503788012955\\
51	60.6890335125736\\
53	61.5612983046321\\
55	62.0719449059631\\
57	63.0590692189266\\
59	64.6773633193385\\
61	66.3969369645492\\
63	66.4647971891537\\
65	70.0324997239055\\
67	70.3954993471659\\
69	72.9038772349705\\
71	73.7929269192593\\
73	75.1255884549943\\
75	76.0375457567257\\
77	78.536132944933\\
79	81.8803501214992\\
81	82.8348843245246\\
83	85.3691806808126\\
85	88.2298896592893\\
87	88.8550450245408\\
89	91.0812058090827\\
91	94.4514521546513\\
93	94.8712058939234\\
95	96.0836419958321\\
97	100.543293654355\\
99	104.142244276135\\
101	107.289177864472\\
103	110.193178922954\\
105	111.811381807114\\
107	117.584837930739\\
109	127.547474946146\\
111	135.599680421606\\
113	153.13429903733\\
115	160.417517479576\\
117	180.301989826871\\
119	208.98879285119\\
121	290.999673990018\\
123	291.000158306076\\
125	291.000158306076\\
127	291.000158306076\\
129	291.000158306076\\
131	291.000158306076\\
133	291.000158306076\\
135	291.000158306076\\
137	291.000158306076\\
139	291.000158306076\\
141	291.000158306076\\
143	291.000158306076\\
145	291.000158306076\\
147	291.000158306076\\
149	291.000158306076\\
151	291.000158306076\\
153	291.000158306076\\
155	291.000158306076\\
157	291.000158306076\\
159	291.000158306076\\
161	291.000158306076\\
163	291.000158306076\\
165	291.000158306076\\
167	291.000158306076\\
169	291.000158306076\\
171	291.000158306076\\
173	291.000158306076\\
175	291.000158306076\\
177	291.000158306076\\
179	291.000158306076\\
181	291.000158306076\\
183	291.000158306076\\
185	291.000158306076\\
187	291.000158306076\\
189	291.000158306076\\
191	291.000158306076\\
193	291.000158306076\\
195	291.000158306076\\
197	291.000158306076\\
199	291.000158306076\\
};
\addlegendentry{$M=1$};

\addplot [color=mycolor2,solid,very thick]
  table[row sep=crcr]{%
1	33.8248751911764\\
3	36.2530336914422\\
5	38.1147114016981\\
7	39.2387342605311\\
9	40.1148060726994\\
11	40.9037641010605\\
13	41.7000393173624\\
15	42.3855914912204\\
17	43.0730355489825\\
19	43.7395466014134\\
21	44.2952856557304\\
23	44.8151703200151\\
25	45.1880747108362\\
27	45.546961388589\\
29	45.9550463528132\\
31	46.3370299428111\\
33	46.7992291994859\\
35	47.1338149816056\\
37	47.4531634561686\\
39	47.7536440774701\\
41	48.1075193684618\\
43	48.4200599901535\\
45	48.7641013121149\\
47	49.0496093650583\\
49	49.3439889226318\\
51	49.6324347553922\\
53	49.9101663743555\\
55	50.1930136644798\\
57	50.4802894949347\\
59	50.7982652662295\\
61	51.1225835716788\\
63	51.4932617153001\\
65	51.8856384303045\\
67	52.2189466426478\\
69	52.682945866236\\
71	52.7139529432235\\
73	53.3514021950508\\
75	53.7189917609158\\
77	53.956508981294\\
79	54.0943229709627\\
81	54.5075602764572\\
83	54.8947500734628\\
85	55.1521169387858\\
87	55.2473556028245\\
89	55.4429974984694\\
91	55.7306167992663\\
93	55.9831954892182\\
95	56.275348715161\\
97	56.6458415859581\\
99	56.8069850183952\\
101	57.3497646073266\\
103	57.9135256500833\\
105	58.2006878035203\\
107	58.5913468741743\\
109	58.8090490670821\\
111	58.9833720661389\\
113	59.2807538654723\\
115	59.6653270725997\\
117	60.1098356406934\\
119	60.381497717218\\
121	60.4728376691834\\
123	60.5536987077446\\
125	60.8252726963269\\
127	61.3329396764844\\
129	61.4263749167471\\
131	61.649420647658\\
133	62.0323371320937\\
135	62.3900266990578\\
137	62.910280572653\\
139	63.1993226736155\\
141	63.5460307887795\\
143	64.1333781028703\\
145	64.5437195658743\\
147	64.9220620196497\\
149	65.1053306488396\\
151	65.6197853702708\\
153	65.9967423704104\\
155	66.3091846835778\\
157	66.6942937745949\\
159	66.9202342116981\\
161	67.2759454582862\\
163	67.5822307592173\\
165	68.0071140691285\\
167	68.4823085279217\\
169	68.7688591109772\\
171	69.1241899325209\\
173	69.441872536265\\
175	69.8236020859649\\
177	70.2251345886901\\
179	70.5451047523985\\
181	70.9468022199347\\
183	71.3550876029574\\
185	71.6590621761632\\
187	71.9521874770034\\
189	72.3331706939528\\
191	72.6939910076678\\
193	73.0336814791156\\
195	73.4843544038134\\
197	73.9112145359938\\
199	74.1669207923791\\
};
\addlegendentry{$M=2$};

\addplot [color=mycolor3,solid,very thick,mark=x]
  table[row sep=crcr]{%
1	37.6804120443062\\
3	41.3849175678233\\
5	43.2248886660788\\
7	44.3189205290585\\
9	45.1790159213114\\
11	46.1160427665303\\
13	46.984328419743\\
15	47.7253118087508\\
17	48.4625498116933\\
19	49.2646913470829\\
21	50.0613829492419\\
23	50.6788577098918\\
25	51.1837597681624\\
27	51.566240083016\\
29	52.0386834328744\\
31	52.517937037849\\
33	52.9094368958106\\
35	53.3754488848882\\
37	53.8323527344281\\
39	54.3195805734775\\
41	54.7076111615993\\
43	55.0972543314961\\
45	55.4708846545305\\
47	55.8151687329942\\
49	56.1753092056021\\
51	56.554650700841\\
53	56.9142744001112\\
55	57.2588779552016\\
57	57.5648957304717\\
59	57.8758758407385\\
61	58.1416782005331\\
63	58.4239737473009\\
65	58.7220572646862\\
67	58.9840982311468\\
69	59.2279905068318\\
71	59.4377036252282\\
73	59.6429105383343\\
75	59.8643724914234\\
77	60.0478889429531\\
79	60.2535707259904\\
81	60.4776824975043\\
83	60.6977092368015\\
85	60.9443278544843\\
87	61.1436814264578\\
89	61.3332815861385\\
91	61.5463646316067\\
93	61.7669413936556\\
95	61.9430716677181\\
97	62.137537729211\\
99	62.3505703796345\\
101	62.5430323650978\\
103	62.7213703551656\\
105	62.8857212670087\\
107	63.0602123112568\\
109	63.2482005968407\\
111	63.430727008806\\
113	63.6214062548811\\
115	63.8437955955382\\
117	64.0445778066522\\
119	64.2467698983887\\
121	64.4391271417182\\
123	64.6313299905711\\
125	64.8545052831902\\
127	65.0498746867442\\
129	65.2368448355609\\
131	65.4352051720134\\
133	65.6404143631267\\
135	65.8582372326015\\
137	66.0547702148911\\
139	66.2544462640753\\
141	66.4495046372169\\
143	66.6320376261298\\
145	66.8302971577568\\
147	67.0122092205597\\
149	67.184521965196\\
151	67.3798702311943\\
153	67.5866065811567\\
155	67.7726477072189\\
157	67.9421788559431\\
159	68.1077788926985\\
161	68.268373692505\\
163	68.4394008805101\\
165	68.8629669160653\\
167	69.024921800758\\
169	69.1984370704367\\
171	69.3606068180166\\
173	69.5185390460891\\
175	69.6900996963839\\
177	69.7544719494595\\
179	70.0246573129664\\
181	70.1969447482627\\
183	70.2583962908751\\
185	70.3226208054606\\
187	70.477450543244\\
189	70.6178223534472\\
191	70.7840704939923\\
193	70.9332630977629\\
195	71.0719382935662\\
197	71.2038653818574\\
199	71.3466499316058\\
};
\addlegendentry{$M=4$};

\addplot [color=mycolor4,solid,very thick,mark=o]
  table[row sep=crcr]{%
1	39.2500731469879\\
3	43.320990567436\\
5	45.3292338313409\\
7	46.5015328759968\\
9	47.3846242061004\\
11	48.207768109195\\
13	48.9072512054426\\
15	49.506059395651\\
17	50.2277853795737\\
19	50.9673217361089\\
21	51.5621279203454\\
23	52.1351220146263\\
25	52.6771506429642\\
27	53.2496561052746\\
29	53.7384018269167\\
31	54.2286510739872\\
33	54.655882324834\\
35	55.0822126659763\\
37	55.4632132981263\\
39	55.8324384544419\\
41	56.2048181679487\\
43	56.579633270763\\
45	56.9537634779524\\
47	57.3355342431415\\
49	57.7020424211105\\
51	58.0467205758429\\
53	58.3770920806853\\
55	58.7392659625871\\
57	59.0517083896523\\
59	59.3916192149964\\
61	59.6940213669336\\
63	60.0062847623419\\
65	60.2878725909239\\
67	60.5685812151277\\
69	60.8347865277213\\
71	61.1007171654108\\
73	61.3352138186009\\
75	61.5833567237901\\
77	61.7710031660524\\
79	61.9452382308148\\
81	62.1269607475209\\
83	62.3171432940162\\
85	62.4964675138573\\
87	62.6850206829542\\
89	62.8703930840472\\
91	63.0515750938509\\
93	63.2442942185767\\
95	63.4240989248081\\
97	63.6021306093351\\
99	63.8017979751458\\
101	63.9824048950654\\
103	64.1681397456915\\
105	64.3514830702174\\
107	64.5625229852539\\
109	64.7770208871714\\
111	64.9661196657761\\
113	65.1430141974949\\
115	65.3142109382363\\
117	65.4561350134185\\
119	65.6444457326454\\
121	65.8652755300907\\
123	66.0401272231708\\
125	66.2364108730159\\
127	66.3873286384236\\
129	66.5982780020869\\
131	66.7698133527826\\
133	66.9583773099857\\
135	67.1048980764098\\
137	67.246190638039\\
139	67.387963835739\\
141	68.0141224651522\\
143	68.1626325662156\\
145	68.328068417793\\
147	68.4973036737401\\
149	68.811383526878\\
151	68.9721769279549\\
153	69.1373795326723\\
155	69.2810429957194\\
157	69.4072088435366\\
159	69.5504067827714\\
161	69.6892117139137\\
163	69.8090102610929\\
165	69.941761565663\\
167	70.0786994499136\\
169	70.215395155551\\
171	70.3413282759254\\
173	70.4958646115495\\
175	70.6202002176113\\
177	70.7613303996721\\
179	70.909598085966\\
181	71.0357896386722\\
183	71.1721907631908\\
185	71.3046455456196\\
187	71.4412664713891\\
189	71.5564639543118\\
191	71.684336991762\\
193	71.8288366912434\\
195	71.9472975893057\\
197	72.0751244849663\\
199	72.1932215610201\\
};
\addlegendentry{$M=6$};

\addplot [color=mycolor5,dashed,very thick]
  table[row sep=crcr]{%
1	39.8372978974754\\
3	44.6790667927151\\
5	46.6114884529099\\
7	47.5843415238418\\
9	48.5244239034703\\
11	49.3190665623961\\
13	50.0527640375442\\
15	50.7031230551152\\
17	51.3896198538106\\
19	52.0346360820495\\
21	52.6048492162417\\
23	53.1947229722048\\
25	53.7498804979364\\
27	54.2661544239991\\
29	54.7660092157351\\
31	55.2707752788665\\
33	55.7696158953742\\
35	56.1550776751194\\
37	56.5077252045437\\
39	56.9549922297563\\
41	57.3402030123345\\
43	57.6791507471058\\
45	58.0598317511327\\
47	58.3823480063094\\
49	58.773920269665\\
51	59.0788855923564\\
53	59.368623995426\\
55	59.7246330840987\\
57	60.0679009672797\\
59	60.3878259812822\\
61	60.713127748456\\
63	61.0827768937835\\
65	61.385518912795\\
67	61.6942088890083\\
69	61.9796182436461\\
71	62.2722941790016\\
73	62.5363070904494\\
75	62.8318407491098\\
77	63.1133585778807\\
79	63.4042703109511\\
81	63.6627415211818\\
83	63.9410967422567\\
85	64.1735241205919\\
87	64.4132741190987\\
89	64.6407911925915\\
91	64.8568065813169\\
93	65.0873808156858\\
95	65.3067434521944\\
97	65.5486173768053\\
99	65.784017441138\\
101	65.9946824153648\\
103	66.1921779832185\\
105	66.3952072655384\\
107	66.5783382069966\\
109	66.7406549601084\\
111	66.9217792551517\\
113	67.1130303242352\\
115	67.2916935903704\\
117	67.4674799628828\\
119	67.6436153811222\\
121	67.7755681443408\\
123	68.0098409556721\\
125	68.2025091143894\\
127	68.3933855210283\\
129	68.541099872381\\
131	68.6630816716541\\
133	68.766707683138\\
135	68.9074778313831\\
137	69.0432071809138\\
139	69.1941083239253\\
141	69.3365130037179\\
143	70.2061395744964\\
145	70.3572622107385\\
147	70.7750609691424\\
149	70.9342796963155\\
151	71.1005513196241\\
153	71.3056030304462\\
155	71.4647638912199\\
157	71.6313759601958\\
159	71.7721209967628\\
161	71.9235565330571\\
163	72.0596895539347\\
165	72.2063174848718\\
167	72.3659828974171\\
169	72.5051753600956\\
171	72.6408353344136\\
173	72.7850246269202\\
175	72.9063879867863\\
177	73.0437617137462\\
179	73.1701904583427\\
181	73.3079619654506\\
183	73.4430223410776\\
185	73.5893554766401\\
187	73.7170548519952\\
189	73.8176794942711\\
191	73.9394258722181\\
193	74.0563227342652\\
195	74.1750878502631\\
197	74.2956771855218\\
199	74.4205995066226\\
};
\addlegendentry{$M=10$};

\end{axis}
\end{tikzpicture}%

%% file: figureTIKZ/fKernel_6_ops_loglog.tex
%
%
\begin{tikzpicture}

\begin{axis}[%
width=0.761\textwidth,
height=0.65\textwidth,
at={(0\textwidth,0\textwidth)},
scale only axis,
xmode=log,
xmin=65536,
xmax=1315066810.39345,
xminorticks=true,
xlabel={Number of operations},
ymin=0,
ymax=0.984414407590819,
ylabel={Spectral Norm},
axis background/.style={fill=white},
legend style={legend cell align=left,align=left,draw=black,font=\tiny, at={(1.05,0.95)}}
]
\addplot [color=blue,solid,very thick,mark=x,mark options={solid}]
  table[row sep=crcr]{%
65536	0.446254862273023\\
131072	0.314857703013986\\
262144	0.179034145524103\\
524288	0.112168581230024\\
1048576	0.0706976902492774\\
2097152	0.0409946939221879\\
4194304	0.0185047324515776\\
8388608	0.00718192850515575\\
16777216	0.00214396951346818\\
33554432	0.000415413977333634\\
67108864	4.81168168386947e-05\\
134217728	8.11069086059657e-07\\
268435456	1.0339501477219e-10\\
422376336	4.21028767405573e-11\\
};
\addlegendentry{Thresholded Matrix};

\addplot [color=black!50!green,dashed,very thick,mark=x,mark options={solid}]
  table[row sep=crcr]{%
1315066810.39345	3.33066907387547e-16\\
346340052.621125	0.0194947607925858\\
95746925.3551443	0.0391522481212855\\
28896543.6916466	0.0795875998112532\\
10089599.0106159	0.156192814577784\\
4352379.72866357	0.327973210348852\\
2419037.25527911	0.58363836967208\\
1713051.18443504	0.77384650656865\\
1459790.08258224	0.984414407590819\\
};
\addlegendentry{WC, Overlap 0\%};

\addplot [color=red,very thick,mark=o,mark options={solid}]
  table[row sep=crcr]{%
1315066810.39345	3.33066907387547e-16\\
382987701.420577	0.0121341498869502\\
115586174.766586	0.011922069509159\\
40358396.0424635	0.0199812336218719\\
17409518.9146543	0.0436463863934036\\
9676149.02111643	0.100451634701922\\
6852204.73774015	0.21548195015929\\
5839160.33032897	0.490025454196642\\
1459790.08258224	0.984414407590819\\
};
\addlegendentry{WC, Overlap 50\%};

\end{axis}
\end{tikzpicture}%

%% file: figureTIKZ/fKernel_rotation_ops_loglog.tex
%
%
\definecolor{mycolor1}{rgb}{0.00000,0.49804,0.00000}%
\begin{tikzpicture}

\begin{axis}[%
width=0.75\textwidth,
height=0.65\textwidth,
at={(0\textwidth,0\textwidth)},
scale only axis,
unbounded coords=jump,
xmode=log,
xmin=65536,
xmax=39638528373.3126,
xminorticks=true,
xlabel={Number of operations},
ymin=0,
ymax=0.998624179646799,
ylabel={Spectral Norm},
axis background/.style={fill=white},
legend style={legend cell align=left,align=left,draw=black,font=\tiny, at={(1.25,0.95)}}
]
\addplot [color=blue,solid,very thick,mark=x,mark options={solid}]
  table[row sep=crcr]{%
65536	0.368610219879878\\
131072	0.335225188646109\\
262144	0.312325665252851\\
524288	0.302622047645085\\
1048576	0.200502271750724\\
2097152	0.10694095470395\\
4194304	0.0523420967971351\\
8388608	0.0196373657539074\\
16777216	0.00550447601662102\\
33554432	0.00121719451020649\\
67108864	0.000150885505997178\\
134217728	5.67267808659899e-06\\
268435456	1.83140687337602e-09\\
493285692	4.48152626120191e-11\\
};
\addlegendentry{Thresholded Matrix};

\addplot [color=black!50!green,dashed,very thick,mark=x,mark options={solid}]
  table[row sep=crcr]{%
39638528373.3126	1.95143017617898e-09\\
10018955185.0615	0.145386502274266\\
2559882537.71534	0.311316984477992\\
668025520.888928	0.483462916923925\\
181518449.719924	0.636094080858595\\
53123382.2329546	0.816434929806987\\
17646281.2325612	0.970922029943637\\
7098155.4665525	0.99688825563715\\
3638585.73601291	nan\\
};
\addlegendentry{WC, Overlap 0\%};

\addplot [color=red,very thick,mark=o,mark options={solid}]
  table[row sep=crcr]{%
39638528373.3126	5.59708631792981e-10\\
10239530150.8613	0.104012543728068\\
2672102083.55571	0.172210122174309\\
726073798.879697	0.235965904172467\\
212493528.931818	0.424915396461683\\
70585124.9302447	0.834746166781855\\
28392621.86621	0.981537311520505\\
14554342.9440516	0.998624179646799\\
3638585.73601291	nan\\
};
\addlegendentry{WC, Overlap 50\%};

\end{axis}
\end{tikzpicture}%

%% file: figureTIKZ/fKernel_rotation_skew_ops_loglog.tex
%
%
\definecolor{mycolor1}{rgb}{0.00000,0.44700,0.74100}%
\definecolor{mycolor2}{rgb}{0.85000,0.32500,0.09800}%
\definecolor{mycolor3}{rgb}{0.92900,0.69400,0.12500}%
\begin{tikzpicture}

\begin{axis}[%
width=0.761\textwidth,
height=0.65\textwidth,
at={(0\textwidth,0\textwidth)},
scale only axis,
xmode=log,
xmin=65536,
xmax=24621071264.8222,
xminorticks=true,
xlabel={Number of operations},
ymin=0,
ymax=0.998441645218796,
ylabel={Spectral Norm},
yminorticks=true,
axis background/.style={fill=white},
legend style={legend cell align=left,align=left,draw=white!15!black,font=\tiny,at={(0.6,0.8)}}
]
\addplot [color=blue,solid,very thick,mark=x,mark options={solid}]
  table[row sep=crcr]{%
65536	0.485354918228642\\
131072	0.310185880007831\\
262144	0.263213652247636\\
524288	0.207682558460773\\
1048576	0.187287615812337\\
2097152	0.0979211631112355\\
4194304	0.0386851421632911\\
8388608	0.0134074164719266\\
16777216	0.00398196872003623\\
33554432	0.000947348077667873\\
67108864	5.1354186159612e-05\\
134217728	2.87473206422861e-06\\
268435456	9.53796200006685e-10\\
488679348	2.19047002758543e-13\\
};
\addlegendentry{Thresholded Matrix};

\addplot [color=black!50!green,dashed,very thick,mark=x,mark options={solid}]
  table[row sep=crcr]{%
24267773776.5152	0.145460162564421\\
6155267816.20555	0.213774181575657\\
1583490596.94258	0.422840550958196\\
418722010.807152	0.669139408309532\\
116619838.608815	0.901053753043959\\
35643321.7416618	0.97535922910597\\
12681121.3043208	0.9923018717283\\
};
\addlegendentry{WC, Overlap 0\%};

\addplot [color=red,very thick,mark=o,mark options={solid}]
  table[row sep=crcr]{%
24621071264.8222	0.11522519330331\\
6333962387.77032	0.136570253492223\\
1674888043.22861	0.233062579171239\\
466479354.435261	0.675084242872704\\
142573286.966647	0.929696293535247\\
50724485.2172833	0.992786345073446\\
22377226.3250812	0.998441645218796\\
};
\addlegendentry{WC, Overlap 50\%};

\end{axis}
\end{tikzpicture}%

%% file: figureTIKZ/fKernel_6_ops_Besov_loglog.tex
%
%
\definecolor{mycolor1}{rgb}{0.00000,0.44700,0.74100}%
\definecolor{mycolor2}{rgb}{0.85000,0.32500,0.09800}%
\definecolor{mycolor3}{rgb}{0.92900,0.69400,0.12500}%
\begin{tikzpicture}

\begin{axis}[%
width=0.731\textwidth,
height=0.622\textwidth,
at={(0\textwidth,0\textwidth)},
scale only axis,
xmode=log,
xmin=65536,
xmax=268435456,
xminorticks=true,
tick label style={font=\tiny},
xlabel={Number of operations},
ymin=0,
ymax=0.125000169957959,
yminorticks=true,
ylabel={Approximation error},
axis background/.style={fill=white},
legend style={legend cell align=left,align=left,draw=white!15!black,font=\tiny}
]
\addplot [color=blue,solid,very thick,mark=x,mark options={solid}]
  table[row sep=crcr]{%
65536	0.0888584003643672\\
131072	0.0682082429811505\\
262144	0.0463502868963712\\
524288	0.034327111620532\\
1048576	0.0257699450436875\\
2097152	0.0164644718196296\\
4194304	0.0084304402760991\\
8388608	0.00315327208434122\\
16777216	0.00114356764888351\\
33554432	0.000255116922137302\\
67108864	2.22254684490316e-05\\
134217728	3.96278774951832e-07\\
268435456	2.45250423862459e-10\\
};
\addlegendentry{Thresholded Matrix};

\addplot [color=black!50!green,dashed,very thick]
  table[row sep=crcr]{%
65536	0.0298837371958821\\
131072	0.0168154309058722\\
262144	0.0113256025288182\\
524288	0.00756084543088517\\
1048576	0.00457596593883503\\
2097152	0.00248362166409551\\
4194304	0.00111187358528503\\
8388608	0.000413223767852354\\
16777216	0.000125725759377502\\
33554432	2.46438359127917e-05\\
67108864	2.12031728447915e-06\\
134217728	2.55175389504706e-07\\
268435456	2.54294448579562e-07\\
};
\addlegendentry{Algorithm \ref{algo:thresh}};

\addplot [color=mycolor3,solid,very thick,mark=o,mark options={solid}]
  table[row sep=crcr]{%
65536	0.119954398040093\\
145408	0.116203165875487\\
261888	0.115406186189971\\
534976	0.0944886315654302\\
1052160	0.0944886315654302\\
2079872	0.0862935107857968\\
4111136	0.0779726473652421\\
8218752	0.0749225129225072\\
15901424	0.0470006609166888\\
29817520	0.0300103333332411\\
53230192	0.0257444345053762\\
96534512	0.010067125890768\\
179753856	0.000903743521493812\\
};
\addlegendentry{Algorithm \ref{algo:thresh2}};

\end{axis}
\end{tikzpicture}%

%% file: figureTIKZ/fKernel_rotation_ops_Besov_loglog.tex
%
%
\definecolor{mycolor1}{rgb}{0.00000,0.44700,0.74100}%
\definecolor{mycolor2}{rgb}{0.85000,0.32500,0.09800}%
\definecolor{mycolor3}{rgb}{0.92900,0.69400,0.12500}%
\begin{tikzpicture}

\begin{axis}[%
width=0.731\textwidth,
height=0.622\textwidth,
at={(0\textwidth,0\textwidth)},
scale only axis,
xmode=log,
xmin=65536,
xmax=268435456,
xminorticks=true,
tick label style={font=\tiny},
xlabel={Number of operations},
ymin=0,
ymax=0.182884143188321,
yminorticks=true,
ylabel={Approximation error},
axis background/.style={fill=white},
legend style={legend cell align=left,align=left,draw=white!15!black, font=\tiny}
]
\addplot [color=blue,solid,very thick,mark=x,mark options={solid}]
  table[row sep=crcr]{%
65536	0.182884143188321\\
131072	0.16036340732461\\
262144	0.136208780997275\\
524288	0.110875036642239\\
1048576	0.0924130504985174\\
2097152	0.0581732261325165\\
4194304	0.0289296949051342\\
8388608	0.0116205580534185\\
16777216	0.00345513246930612\\
33554432	0.00076736779628973\\
67108864	8.33954111151558e-05\\
134217728	2.47503177311827e-06\\
268435456	3.90020914870588e-09\\
};
\addlegendentry{Thresholded Matrix};

\addplot [color=black!50!green,dashed,very thick]
  table[row sep=crcr]{%
65536	0.0558089403940646\\
131072	0.0417171399848174\\
262144	0.0300888849622725\\
524288	0.0209313672437054\\
1048576	0.0133661841796302\\
2097152	0.0073900315789433\\
4194304	0.00341264182539281\\
8388608	0.00128796967529825\\
16777216	0.000381422039975081\\
33554432	7.57882903755172e-05\\
67108864	7.83672783139564e-06\\
134217728	3.37826563099241e-07\\
268435456	2.67351341874918e-07\\
};
\addlegendentry{Algorithm \ref{algo:thresh}};

\addplot [color=mycolor3,solid,very thick,mark=o,mark options={solid}]
  table[row sep=crcr]{%
67584	0.182884143188321\\
139264	0.182884143188321\\
264960	0.168880352845242\\
525824	0.161773536224027\\
1053696	0.160384717102864\\
2100736	0.150808123841209\\
4209664	0.134256469715825\\
8389632	0.117675270228437\\
16453488	0.0857091911769237\\
30607743	0.0675942962760025\\
59008348	0.0588620554958665\\
108585392	0.0346517720439939\\
202828634	0.00946671987587235\\
};
\addlegendentry{Algorithm \ref{algo:thresh2}};

\end{axis}
\end{tikzpicture}%

%% file: figureTIKZ/fKernel_rotation_skew_ops_Besov_loglog.tex
%
%
\definecolor{mycolor1}{rgb}{0.00000,0.44700,0.74100}%
\definecolor{mycolor2}{rgb}{0.85000,0.32500,0.09800}%
\definecolor{mycolor3}{rgb}{0.92900,0.69400,0.12500}%
\begin{tikzpicture}

\begin{axis}[%
width=0.731\textwidth,
height=0.622\textwidth,
at={(0\textwidth,0\textwidth)},
scale only axis,
xmode=log,
xmin=65536,
xmax=268435456,
xminorticks=true,
tick label style={font=\tiny},
xlabel={Number of operations},
ymin=0,
ymax=0.502873895371765,
yminorticks=true,
ylabel={Approximation error},
axis background/.style={fill=white},
legend style={legend cell align=left,align=left,draw=white!15!black,font=\tiny}
]
\addplot [color=blue,solid,very thick,mark=x,mark options={solid}]
  table[row sep=crcr]{%
65536	0.380533563217209\\
131072	0.32624523846054\\
262144	0.255079262744851\\
524288	0.190117484793674\\
1048576	0.13349948195938\\
2097152	0.0707742644500429\\
4194304	0.030947795633546\\
8388608	0.0119931284079998\\
16777216	0.00368622802884437\\
33554432	0.000860516567022863\\
67108864	8.52323700234465e-05\\
134217728	2.28647541381975e-06\\
268435456	3.10490114549402e-09\\
};
\addlegendentry{Thresholded Matrix};

\addplot [color=black!50!green,dashed,very thick]
  table[row sep=crcr]{%
65536	0.117047923484813\\
131072	0.0824832563183131\\
262144	0.0534604139633752\\
524288	0.0320290520866279\\
1048576	0.017507677342871\\
2097152	0.00862636550127428\\
4194304	0.00368865849941645\\
8388608	0.00129581259047169\\
16777216	0.000377688813881235\\
33554432	7.93944413406549e-05\\
67108864	7.9810538218937e-06\\
134217728	4.04311429996231e-07\\
268435456	3.56993473374943e-07\\
};
\addlegendentry{Algorithm \ref{algo:thresh}};

\addplot [color=mycolor3,solid,very thick,mark=o,mark options={solid}]
  table[row sep=crcr]{%
69376	0.502873895371765\\
131072	0.502853991712991\\
262912	0.464956589444791\\
525056	0.434713276980079\\
1059840	0.38749358532462\\
2113280	0.330604100437448\\
4197119	0.295946324092019\\
8393726	0.248649653993522\\
16599037	0.174916920804258\\
32259081	0.127408250158811\\
60307064	0.124525143006556\\
112092136	0.105906783379061\\
206563400	0.0400196423062025\\
};
\addlegendentry{Algorithm \ref{algo:thresh2}};

\end{axis}
\end{tikzpicture}%

%% file: figureTIKZ/Figure_8_Classique_4_fKernel_6.tex
%
%
\definecolor{mycolor1}{rgb}{0.00000,0.44700,0.74100}%
\definecolor{mycolor2}{rgb}{0.85000,0.32500,0.09800}%
\definecolor{mycolor3}{rgb}{0.92900,0.69400,0.12500}%
\definecolor{mycolor4}{rgb}{0.49400,0.18400,0.55600}%
\definecolor{mycolor5}{rgb}{0.46600,0.67400,0.18800}%
\definecolor{mycolor6}{rgb}{0.30100,0.74500,0.93300}%
\definecolor{mycolor7}{rgb}{0.87059,0.49020,0.00000}%
\begin{tikzpicture}

\begin{axis}[%
width=0.951\textwidth,
height=0.804\textwidth,
at={(0\textwidth,0\textwidth)},
scale only axis,
xmin=0,
xmax=100,
xlabel={Number of operations divided by $N$},
ymin=30,
ymax=100,
ylabel={pSNR},
axis background/.style={fill=white},
legend style={at={(0.03,0.97)},anchor=north west,legend cell align=left,align=left,draw=white!15!black,font=\tiny}
]
\addplot [color=mycolor1,solid, very thick]
  table[row sep=crcr]{%
1	38.3737655092748\\
5	42.3609799677523\\
9	43.8868892209351\\
13	47.3343845317157\\
17	49.4010141662496\\
21	50.4650780924243\\
25	51.236109854016\\
29	52.3199383546051\\
33	53.900781193337\\
37	55.4018804671553\\
41	55.9596249043284\\
45	56.9850783082451\\
49	57.5312708391459\\
53	57.9404391498826\\
57	58.3342739466836\\
61	59.2404298834065\\
65	59.9563041991075\\
69	61.1014058913323\\
73	61.7144758199929\\
77	61.9786673824602\\
81	62.6544538393127\\
85	62.3828734196056\\
89	62.4431464422887\\
93	62.8380060259178\\
97	63.2364924317514\\
};
\addlegendentry{Thresholded Matrix};

\addplot [color=mycolor2,solid, very thick, mark=x]
  table[row sep=crcr]{%
1	37.9425601708378\\
5	42.2344990031582\\
9	43.0350854027493\\
13	43.4964917533029\\
17	43.7761954868729\\
21	44.5189668210504\\
25	46.2574898583631\\
29	48.4563674952982\\
33	50.1746071019458\\
37	50.4354641203378\\
41	51.4220410355927\\
45	52.6414068146064\\
49	52.794209051537\\
53	54.8894444593633\\
57	54.8340611900775\\
61	55.799129309698\\
65	57.1900320033437\\
69	57.4360825754678\\
73	58.6790611686254\\
77	59.5409721681366\\
81	60.2268425727356\\
85	60.0932128000088\\
89	60.2431338440255\\
93	60.4393853007391\\
97	62.0350078565629\\
};
\addlegendentry{$\bSigma$ $\text{n}^\circ\text{1}$};

\addplot [color=mycolor3,solid, very thick, mark=o]
  table[row sep=crcr]{%
1	38.9390552600598\\
5	46.8174895104493\\
9	49.4715906868721\\
13	51.5118587432975\\
17	53.2682824862763\\
21	54.8771879120294\\
25	56.2476535610529\\
29	57.3198932565004\\
33	58.4436304984217\\
37	59.4459672109116\\
41	60.4004108738815\\
45	61.3302678623594\\
49	62.2244280104397\\
53	63.08241658929\\
57	63.9525460981033\\
61	64.696358900534\\
65	65.5020501869541\\
69	66.1439945159591\\
73	66.7585358248292\\
77	67.3785903600714\\
81	68.0197011865665\\
85	68.6056835119602\\
89	69.2225899966014\\
93	69.7452046697382\\
97	70.2902958306991\\
};
\addlegendentry{$\bSigma$ $\text{n}^\circ\text{2}$};

\addplot [color=mycolor5,dotted, very thick]
  table[row sep=crcr]{%
1	27.923508901857\\
5	35.8238136176856\\
9	40.3766765762074\\
13	44.1109548719989\\
17	46.7720881310213\\
21	49.13236202486\\
25	51.2728845802231\\
29	53.8497615792038\\
33	56.6226004165184\\
37	59.8017361208358\\
41	64.2415313249062\\
45	68.800424471195\\
49	71.7135001314023\\
53	75.1464270740385\\
57	79.1536544588186\\
61	82.3993864575468\\
65	85.6948722613662\\
69	89.1985146568623\\
73	93.6175000521997\\
77	98.4396236733902\\
81	103.300161551367\\
85	107.98640462554\\
89	114.217900950707\\
93	121.127415644683\\
97	127.599599505161\\
};
\addlegendentry{\cite{wei2014fast}};

\addplot [color=mycolor6,dashed, very thick]
  table[row sep=crcr]{%
1	31.9060872774412\\
5	42.0429394100431\\
9	42.6995650201535\\
13	43.725576136002\\
17	44.2295704201468\\
21	44.2857841355936\\
25	44.4282494565143\\
29	44.5686559258873\\
33	44.6423276430479\\
37	44.7004629368023\\
41	44.8370214602251\\
45	45.0048246821164\\
49	45.0915441087839\\
53	45.1069451671724\\
57	45.2452727743759\\
61	45.2890057757044\\
65	45.2910905217892\\
69	45.3670249401526\\
73	45.4289950326762\\
77	45.4277418806428\\
81	45.4304578168263\\
85	45.4279783258076\\
89	45.4286843943939\\
93	45.4335057466767\\
97	45.8002336570791\\
};
\addlegendentry{Algo \ref{algo:thresh2}};

\addplot [color=black,solid, very thick,mark=square]
  table[row sep=crcr]{%
17.394	32.0751947274113\\
69.575	42.7098723492563\\
66.101	49.8172119495881\\
66.412	59.2892043935764\\
74.313	68.1898037082421\\
99.297	73.5915563023645\\
168.86	75.0897348119541\\
};
\addlegendentry{WC, Overlap 50\%};

\end{axis}
\end{tikzpicture}%

%% file: figureTIKZ/Figure_8_Classique_4_fKernel_rotation.tex
%
%
\definecolor{mycolor1}{rgb}{0.00000,0.44700,0.74100}%
\definecolor{mycolor2}{rgb}{0.85000,0.32500,0.09800}%
\definecolor{mycolor3}{rgb}{0.92900,0.69400,0.12500}%
\definecolor{mycolor4}{rgb}{0.49400,0.18400,0.55600}%
\definecolor{mycolor5}{rgb}{0.46600,0.67400,0.18800}%
\definecolor{mycolor6}{rgb}{0.30100,0.74500,0.93300}%
\definecolor{mycolor7}{rgb}{0.87059,0.49020,0.00000}%
\begin{tikzpicture}

\begin{axis}[%
width=0.951\textwidth,
height=0.804\textwidth,
at={(0\textwidth,0\textwidth)},
scale only axis,
unbounded coords=jump,
xmin=0,
xmax=100,
xlabel={Number of operations divided by $N$},
ymin=25,
ymax=75,
ylabel={pSNR},
axis background/.style={fill=white},
legend style={at={(0.03,0.97)},anchor=north west,legend cell align=left,align=left,draw=white!15!black,font=\tiny}
]
\addplot [color=mycolor1,solid,very thick]
  table[row sep=crcr]{%
1	33.3412922044385\\
5	35.2856522292757\\
9	38.7878235568977\\
13	41.0097231915556\\
17	42.3940910722387\\
21	43.1723541954876\\
25	43.765253533487\\
29	44.5219112899715\\
33	44.9456526725852\\
37	45.6214460986524\\
41	46.265880050114\\
45	46.7594246162776\\
49	47.3087337179771\\
53	47.9541228359501\\
57	48.3186051754158\\
61	48.7543888090019\\
65	49.2694268406547\\
69	49.6408755923933\\
73	50.0800395111575\\
77	50.4583722147827\\
81	51.0253094331269\\
85	51.569058491467\\
89	51.6458026736166\\
93	52.1142527835026\\
97	52.4190964447564\\
};
\addlegendentry{Thresholded Matrix};

\addplot [color=mycolor2,solid, very thick, mark=x]
  table[row sep=crcr]{%
1	32.9447704088915\\
5	34.3018465293273\\
9	34.814751220733\\
13	35.1359062802291\\
17	35.3556821244817\\
21	35.8820002648708\\
25	37.0675524700065\\
29	38.2236659482106\\
33	39.666249909236\\
37	41.1522125789804\\
41	42.7448473435852\\
45	43.843812049129\\
49	44.9743313744475\\
53	46.0837725024288\\
57	46.6998844698713\\
61	47.3963232213995\\
65	47.9690327632597\\
69	48.6185174073685\\
73	48.8836942549563\\
77	49.7408827175451\\
81	50.3230819555685\\
85	50.8024249824292\\
89	50.9223348696545\\
93	51.60236818212\\
97	52.1000168380007\\
};
\addlegendentry{$\bSigma$ $\text{n}^\circ\text{1}$};

\addplot [color=mycolor3,solid, very thick, mark=o]
  table[row sep=crcr]{%
1	33.9418364716922\\
5	39.7016305103405\\
9	41.8713469328484\\
13	43.3894981778886\\
17	44.7933698757962\\
21	45.9361119376355\\
25	47.0828769389484\\
29	48.1888861094176\\
33	49.218025610404\\
37	50.1478163240477\\
41	51.0422404213205\\
45	51.8804641633854\\
49	52.6687676420034\\
53	53.4136710258217\\
57	54.1643103337481\\
61	54.8680658533602\\
65	55.5668962382774\\
69	56.1989591010147\\
73	56.832959165915\\
77	57.4291452205281\\
81	57.9969219031146\\
85	58.5477872289298\\
89	59.0523646013182\\
93	59.5865627342403\\
97	60.1194689036537\\
};
\addlegendentry{$\bSigma$ $\text{n}^\circ\text{2}$};

\addplot [color=mycolor5,dotted, very thick]
  table[row sep=crcr]{%
1	26.6085848143243\\
5	32.7180602916841\\
9	35.8782868901111\\
13	38.5302614149853\\
17	40.6088893179353\\
21	42.6938222753724\\
25	44.6465214270242\\
29	46.4986276594591\\
33	48.0943809034439\\
37	49.7113299647658\\
41	51.1820308020846\\
45	52.5690990030626\\
49	53.9450430225951\\
53	55.3185792483038\\
57	56.4412707711682\\
61	57.8834527924044\\
65	59.1593503459874\\
69	60.5172903403765\\
73	61.9660855951205\\
77	63.6510469369794\\
81	65.299010272822\\
85	66.7937628780399\\
89	68.3506040125285\\
93	69.7583708496091\\
97	71.1919979684648\\
};
\addlegendentry{\cite{wei2014fast}};

\addplot [color=mycolor6,dashed, very thick]
  table[row sep=crcr]{%
1	35.1614331204701\\
5	38.8212134761636\\
9	39.3955179105383\\
13	39.9593227573035\\
17	40.176506623502\\
21	40.2531398199156\\
25	40.3538857686471\\
29	40.4320153862288\\
33	40.4803091723882\\
37	40.5464543702832\\
41	40.5551380296644\\
45	40.4315238023422\\
49	40.4524621417548\\
53	40.6098831974173\\
57	40.7810896523393\\
61	40.7987163563746\\
65	40.8770907940169\\
69	40.9187694788969\\
73	40.7702188752685\\
77	40.7723643805551\\
81	40.9892255550582\\
85	40.996436723698\\
89	40.9685529905661\\
93	41.0117418067836\\
97	41.1554714539667\\
};
\addlegendentry{Algo \ref{algo:thresh2}};

\addplot [color=black,solid, very thick,mark=square]
  table[row sep=crcr]{%
20.547	nan\\
90.405	34.6999508518416\\
107.27	41.5546841975246\\
153.88	48.3056463803247\\
281.94	56.9045938797162\\
664.53	65.9319645021475\\
1926.7	68.2406155205989\\
};
\addlegendentry{WC, Overlap 50\%};

\end{axis}
\end{tikzpicture}%

%% file: figureTIKZ/Figure_8_Classique_4_fKernel_rotation_skew.tex
%
%
\definecolor{mycolor1}{rgb}{0.00000,0.44700,0.74100}%
\definecolor{mycolor2}{rgb}{0.85000,0.32500,0.09800}%
\definecolor{mycolor3}{rgb}{0.92900,0.69400,0.12500}%
\definecolor{mycolor4}{rgb}{0.49400,0.18400,0.55600}%
\definecolor{mycolor5}{rgb}{0.46600,0.67400,0.18800}%
\definecolor{mycolor6}{rgb}{0.30100,0.74500,0.93300}%
\definecolor{mycolor7}{rgb}{0.87059,0.49020,0.00000}%
\begin{tikzpicture}

\begin{axis}[%
width=0.951\textwidth,
height=0.804\textwidth,
at={(0\textwidth,0\textwidth)},
scale only axis,
unbounded coords=jump,
xmin=0,
xmax=100,
xlabel={Number of operations divided by $N$},
ymin=25,
ymax=75,
ylabel={pSNR},
axis background/.style={fill=white},
legend style={at={(0.03,0.97)},anchor=north west,legend cell align=left,align=left,draw=white!15!black,font=\tiny}
]
\addplot [color=mycolor1,solid,very thick]
  table[row sep=crcr]{%
1	25.2590229906184\\
5	30.6059947483309\\
9	33.2023392674006\\
13	35.1781841592803\\
17	36.4890275897252\\
21	37.8736170571906\\
25	38.7429160155218\\
29	39.2367671412775\\
33	39.9082231534916\\
37	40.5483412177317\\
41	41.1831761438369\\
45	41.8158589173642\\
49	42.6752387359369\\
53	43.5383449669479\\
57	44.2688959885032\\
61	44.8066628678349\\
65	45.3227382529918\\
69	45.7219864039925\\
73	46.0871094999709\\
77	46.57560337237\\
81	47.0077058017436\\
85	47.4607929170408\\
89	47.9112463722178\\
93	48.329153312811\\
97	48.7172437414546\\
};
\addlegendentry{Thresholded Matrix};

\addplot [color=mycolor2,solid, very thick, mark=x]
  table[row sep=crcr]{%
1	25.7705976128403\\
5	28.7872815018112\\
9	31.0791035093065\\
13	33.2082290498077\\
17	34.9890075781194\\
21	36.5610396322197\\
25	37.8318460068953\\
29	39.1374038570171\\
33	40.3540788746423\\
37	41.4808653534084\\
41	42.4248458666275\\
45	43.3246098079677\\
49	44.2382487875394\\
53	44.9926710608954\\
57	45.6611866420541\\
61	46.3590349885417\\
65	46.9496840661064\\
69	47.5178561852208\\
73	48.099109764803\\
77	48.6200145872472\\
81	49.1159472860806\\
85	49.6754112610005\\
89	50.2887305478608\\
93	50.9003211270541\\
97	51.442330520237\\
};
\addlegendentry{$\bSigma$ $\text{n}^\circ\text{1}$};

\addplot [color=mycolor3,solid, very thick, mark=o]
  table[row sep=crcr]{%
1	28.283939584095\\
5	34.9043858154944\\
9	38.2849566604354\\
13	40.5529232273136\\
17	42.2677721462462\\
21	43.7713797561322\\
25	45.0601985274128\\
29	46.2049756632877\\
33	47.3055786961118\\
37	48.3577250632771\\
41	49.3418306082446\\
45	50.2933957774352\\
49	51.2022019617528\\
53	52.056106654369\\
57	52.8495526676723\\
61	53.602776594359\\
65	54.3024086806561\\
69	54.9942739501409\\
73	55.687794603869\\
77	56.3448668546199\\
81	56.9825888564327\\
85	57.5924033928448\\
89	58.1832215192953\\
93	58.7462810323283\\
97	59.3129174686777\\
};
\addlegendentry{$\bSigma$ $\text{n}^\circ\text{2}$};

\addplot [color=mycolor5,dotted, very thick]
  table[row sep=crcr]{%
1	26.0246405056029\\
5	32.5513042891205\\
9	36.836758538066\\
13	39.9348106473283\\
17	42.5054373654954\\
21	44.6664095616409\\
25	46.7664306016171\\
29	48.6596412992172\\
33	50.3698948617981\\
37	52.1725757465132\\
41	53.866135118423\\
45	55.4729167077813\\
49	57.1902300582184\\
53	58.810995077794\\
57	60.3424945524588\\
61	61.8696528169147\\
65	63.3091786557336\\
69	64.6538722274901\\
73	66.1095470946867\\
77	67.5739221500053\\
81	69.0780090095119\\
85	70.5591986548684\\
89	72.0080998994429\\
93	73.6070926702016\\
97	74.9855591148015\\
};
\addlegendentry{\cite{wei2014fast}};

\addplot [color=mycolor6,dashed, very thick]
  table[row sep=crcr]{%
1	27.1674561892405\\
5	28.0944206490864\\
9	27.73823482208\\
13	27.7787986915469\\
17	27.8357752660031\\
21	28.0484334294239\\
25	28.1486897220396\\
29	28.157801296052\\
33	28.1575074158588\\
37	28.1871721755879\\
41	28.5110800469558\\
45	28.810981630167\\
49	28.8404649436837\\
53	28.9708871987387\\
57	29.2770671966698\\
61	29.3026945731241\\
65	29.663735195741\\
69	29.9556848301555\\
73	30.4336871117045\\
77	30.4455416239947\\
81	30.5817897943087\\
85	30.5947734818524\\
89	30.5809712347743\\
93	30.6144137718161\\
97	30.7109195360293\\
};
\addlegendentry{Algo \ref{algo:thresh2}};

\addplot [color=black,solid, very thick,mark=square]
  table[row sep=crcr]{%
22.601	nan\\
90.405	27.6999508518416\\
107.27	29.5546841975246\\
153.88	41.3056463803247\\
281.94	49.9045938797162\\
664.53	65.9319645021475\\
1926.7	68.2406155205989\\
};
\addlegendentry{WC, Overlap 50\%};

\end{axis}
\end{tikzpicture}%

%% file: figureTIKZ/pSNR_vs_ops.tex
%
%
\begin{tikzpicture}

\begin{axis}[%
width=0.4\textwidth,
height=0.3\textwidth,
at={(0\textwidth,0\textwidth)},
scale only axis,
xmin=0,
xmax=30,
xlabel={Number of Operations divided by $N$},
ymin=24.6,
ymax=26,
ylabel={pSNR},
axis background/.style={fill=white},
legend style={at={(0.97,0.03)},anchor=south east,legend cell align=left,align=left,draw=black,font=\tiny}
]
\addplot [color=blue,solid,very thick]
  table[row sep=crcr]{%
1	24.6175179347433\\
2	25.0842700502701\\
3	25.3081851009426\\
4	25.3940688001741\\
5	25.4731224473305\\
6	25.5156837916106\\
7	25.5576399947774\\
8	25.5927081987149\\
9	25.6272981256923\\
10	25.6499033289016\\
11	25.6773527818184\\
12	25.7022653370157\\
13	25.7212633960335\\
14	25.7376198560422\\
15	25.7478177834637\\
16	25.7643343966873\\
17	25.7820346038751\\
18	25.7902349043193\\
19	25.7978454257447\\
20	25.803953594932\\
21	25.8072610864368\\
22	25.8114281736139\\
23	25.8139252933415\\
24	25.8159497472874\\
25	25.8169092712485\\
26	25.8207456702191\\
27	25.8224238107898\\
28	25.8224249885689\\
29	25.8215690430367\\
30	25.8202731630227\\
};
\addlegendentry{$\text{Algo. 1 with Sigma n}^\circ\text{3}$};

\addplot [color=red,dashed,very thick]
  table[row sep=crcr]{%
1	25.8419682708757\\
30	25.8419682708757\\
};
\addlegendentry{Exact Operator};

\end{axis}
\end{tikzpicture}%

%% file: Body/40-conclusion.tex
\section{Conclusion}

\subsection{Brief summary}
In this paper, we introduced an original method to represent spatially varying blur operators in the wavelet domain. 
We showed that this new technique has a great adaptivity to the smoothness of the operator and exhibit an $\mathcal{O}(N \epsilon^{-d/M})$ complexity, where $M$ denotes the kernel regularity. 
This method is versatile since it is possible to adapt it to the kind of images that have to be treated. 
We showed that much better performance to approximate the direct operator can be obtained by leveraging the fact that natural signals exhibit some structure in the wavelet domain. 
Moreover, we proposed a original method to design sparsity patterns for class of blurring operators when only the operator regularity is known.
These theoretical results were confirmed by practical experiments on real images. 
Even though our conclusions are still preliminary since we tested only small $256\times 256$ images, the wavelet based methods seem to significantly outperform standard windowed convolutions based approaches. 
Moreover, they seem to provide satisfactory deblurring results on practical problems with a complexity no greater than $5N$ operations, where $N$ denotes the pixels number.

\subsection{Outlook}

We provided a simple complexity analysis based solely on the \textit{global} regularity of the kernel function. It is well known that wavelets are able to adapt locally to the structures of images or operators \cite{cohen2002adaptive}. 
The method should thus provide an efficient tool for piecewise regular blurs appearing in computer vision for instance. 
It could be interesting to evaluate precisely the complexity of wavelet based approximations for piecewise regular blurs.

A key problem of the wavelet based approach is the need to project the operator on a wavelet basis. 
In this paper we performed this operation using the computationally intensive Algorithm \ref{algo:btheta}. 
It could be interesting to derive fast projection methods.
Let us note that such methods already exist in the literature \cite{BCR}. 
A similar procedure was used in the specific context of spatially varying blur in \cite{wei2014fast}. 

Moreover, the proposed method can already be applied to situations where the blur mostly depends on the instrument: the wavelet representation has to be computed once for all off-line, and then all deblurring operations can be handled much faster.
This situation occurs in satellite imaging or for some fluorescence microscopes (see e.g. \cite{hajlaoui2010satellite,temerinac2012multiview,maalouf2011fluorescence}).

The design of good sparsity patterns is an open and promising research avenue. In particular, designing patterns adapted to specific inverse problems could have some impact as was illustrated in section \ref{sec:Inverse_Problems}.

Another exciting research perspective is the problem of blind deconvolution. 
Expressing the unknown operator as a sparse matrix in the wavelet domain is a good way to improve the problem identifiability.
This is however far from being sufficient since the blind deconvolution problem has far more unknowns (a full operator and an image) than data (a single image).  
Further assumptions should thus be made on the wavelet coefficients regularity, and we plan to study this problem in a forthcoming work. 

Finally let us mention that we observed some artefacts when using the wavelet based methods with high sparsity levels. 
This is probably due to their non translation and rotation invariance. 
It could be interesting to study sparse approximations in redundant wavelet bases or other time-frequency bases.
It was shown for instance in \cite{candes2003curvelets} that curvelets are near optimal to represent Fourier integral operators.
Similarly, Gabor frames are known to be very efficient to describe smoothly varying integral operators in the 1D setting \cite{hrycak2011practical}.

%% file: Body/50-acknowledgements.tex
\section*{Acknowledgements}

The authors thank Fran\c{c}ois Malgouyres and Mathieu Bouyrie for stimulating discussions in the early preparation of this work.
They thank J\'er\'emie Bigot and Guillermo Cabrera for providing the SDSS database address as an illustration of a deblurring problem with a spatially variant PSF.
They thank J\'er\^ome Fehrenbach and the IP3D team of ITAV for their comments and support.
They thank Sandrine Anthoine, Caroline Chaux, Hans Feichtinger and Clothilde M\'elot for their comments on the manuscript.
They thank Michael Hirsch for providing the pictures in Figure \ref{fig:dwarves_motion} and indicating reference \cite{hirsch2010efficient}.
They thank Lo\"ic Denis and F\'err\'eol Soulez for indicating references \cite{denis2014fast,wei2014fast}.
This work was supported by ANR SPH-IM-3D (ANR-12-BSV5-0008). 
Paul Escande is pursuing a PhD degree supported by the MODIM project funded by the PRES of Toulouse University and Midi-Pyr\'en\'ees region. 

%% file: Body/109-proof_decay.tex
\section{Proof of Lemma \ref{lem:decay}} \label{appendix:proof_decay}

We let $\Pi_M$ denote the set of polynomials of degree less or equal to $M$.

Lemma \ref{lemma:poly_approx} below is a common result in numerical analysis \cite{deny1954espaces} (see also Theorem 3.2.1 in \cite{cohen2003numerical}). 
It ensures that the approximation error of a function by a polynomial of degree $M$ is bounded by the Sobolev semi-norm $W^{M,p}$.

\begin{lemma}[Polynomial approximation] \label{lemma:poly_approx}
For $1 \leq p \leq +\infty$, $M \in \N^*$ and $\Omega \subset \R^d$ a bounded domain, the following bound holds
	\begin{equation}
		\inf_{g \in \Pi_M}\norm{f - g}_{\LL^p(\Omega)} \leq C \abs{f}_{W^{M+1,p}(\Omega)},
	\end{equation}		
	where $C$ is a constant that depends on $d,M, p$ and $\Omega$ only.

	Moreover, if $I_h \subset \Omega \subset \R^d$ is a cube of sidelength $h$, the following estimate holds
	\begin{equation}
		\inf_{g \in \Pi_M}\norm{f - g}_{\LL^p(I_h)} \leq C h^{M+1} \abs{f}_{W^{M+1,p}(I_h)},
	\end{equation}		
	where $C$ is a constant only depending on $d,M, p$ and $\Omega$.
\end{lemma}


Let $I_{\lambda}=\supp(\psi_\lambda)$.
From the wavelets definition, we get
		\[
			I_\lambda = 2^{-j}(m + [-c(M)/2,c(M)/2]^d)
		\]
	therefore $\abs{I_{\lambda}} = c(M)^d \cdot 2^{-jd}$.
We will now prove Lemma \ref{lem:decay}.

\begin{proof}[Proof of Lemma \ref{lem:decay}]
Since the mapping $(x,y)\mapsto K(x,y) \psi_{\lambda}(y) \psi_{\mu}(x)$ is bounded, it is also absolutely integrable on compact domains. 
Therefore $\dotproduct{H \psi_{\lambda}}{\psi_{\mu}}$ is well-defined for all $(\lambda, \mu)$. Recall that $\lambda = (j,m,e) \in \Lambda$ and $\mu = (k,n,e') \in \Lambda$.
Moreover Fubini's theorem can be applied and we get
\begin{align*}
	\dotproduct{H\psi_{\lambda}}{\psi_{\mu}} &= \int_{I_{\mu}} \int_{I_{\lambda}} K(x,y) \psi_{\lambda}(y) \psi_{\mu}(x) dy dx \\
	& = \int_{I_\lambda} \int_{I_{\mu}} K(x,y) \psi_{\lambda}(y) \psi_{\lambda}(x) dx dy.
\end{align*}

To prove the result, we distinguish the cases $j \leq k$ and $ j > k$. 
In this proof, we focus on the case $j\leq k$. 
The other one can be obtained by symmetry, using the facts that $\dotproduct{H \psi_{\lambda}}{\psi_{\mu}}=\dotproduct{\psi_{\lambda}}{H^*\psi_{\mu}}$ and that $H$ and $H^*$ are both blurring operators in the same class.

To exploit the regularity of $K$ and $\psi$, note that for all $g \in \Pi_{M-1}$, $\displaystyle \int_{I_\mu} g(x) \psi_{\mu}(x) dx=0$ since $\psi$ has $M$ vanishing moments. 
Therefore,
\begin{align*}
  \dotproduct{H \psi_{\lambda}}{\psi_{\mu}} &= \int_{I_{\lambda}} \inf_{g \in \Pi_{M-1}} \int_{I_{\mu}} \left( K(x,y) - g(x) \right) \psi_{\lambda}(y)  \psi_{\mu}(x) dx dy,
\end{align*}
and
\begin{align*}
  \abs{\dotproduct{H \psi_{\lambda}}{\psi_{\mu}}} & \leq \int_{I_{\lambda}} \inf_{g \in \Pi_{M-1}} \int_{I_{\mu}} \abs{ K(x,y) - g(x) }  \abs{\psi_{\lambda}(y)} \abs{ \psi_{\mu}(x)} dx dy \\
  & \leq \int_{I_{\lambda}} \inf_{g \in \Pi_{M-1}} \norm{ K(\cdot,y) - g}_{\LL^{\infty}(I_{\mu}) }  \norm{\psi_{\mu}}_{\LL^1(I_{\mu})} \abs{\psi_{\lambda} (y)} dy.
\end{align*}
By Lemma \ref{lemma:poly_approx}, $\displaystyle \inf_{g \in \Pi_{M-1}} \norm{ K(\cdot,y) - g}_{\LL^{\infty}(I_{\mu}) } \lesssim 2^{-kM}  \left| K(\cdot,y) \right|_{W^{M,\infty}(I_{\mu})}$ since $I_{\mu}$ is a cube of sidelength $c(M)\cdot 2^{-k}$. We thus obtain 
\begin{align*}
\abs{\dotproduct{H \psi_{\lambda}}{\psi_{\mu}}} & \lesssim 2^{-kM} \norm{\psi_{\mu}}_{\LL^1(I_{\mu})} \norm{\psi_{\lambda}}_{\LL^1(I_{\lambda})} \esssup_{y \in I_{j,m}} \left| K(\cdot,y) \right|_{W^{M,\infty}(I_{\mu})} \nonumber \\
&\lesssim 2^{-kM} 2^{-\frac{dj}{2}} 2^{-\frac{dk}{2}} \esssup_{y \in I_{\lambda}} \left| K(\cdot,y) \right|_{W^{M,\infty}(I_{\mu})} 
\end{align*}
since $\norm{\psi_{\lambda}}_{\LL^1} = 2^{-\frac{dj}{2}} \norm{\psi}_{\LL^1}$.

Since $H \in \mathcal{A}(M,f)$
\begin{align*}
	\esssup_{y \in I_{\lambda}} \left| K(\cdot,y) \right|_{W^{M,\infty}(I_{\mu})} &= \esssup_{y \in  I_{\lambda} } \sum_{\abs{\alpha} = M} \esssup_{x \in I_{\mu}} \abs{\p^{\alpha}_x K(x,y)} \\
	&\leq \sum_{\abs{\alpha} = M} \esssup_{(x,y) \in I_{\lambda} \times I_{\mu} } f\left(\norm{x-y}_\infty\right) \\
	&\lesssim \esssup_{(x,y) \in I_{\lambda} \times I_{\mu} } f\left(\norm{x-y}_\infty\right).
\end{align*}
Because $f$ is a non-increasing function, $f\left( \norm{x-y}_\infty\right) \leq f\left( \dist{ I_{\lambda} }{ I_{\mu}} \right)$ since $ \displaystyle \dist{I_{\lambda}}{I_{\mu}} = \inf_{(x,y) \in I_{\lambda} \times I_{\mu}} \norm{x-y}_\infty$. Therefore
\begin{align*}
	\abs{\dotproduct{H \psi_{\lambda}}{\psi_{\mu}}}& \lesssim 2^{-kM} 2^{-\frac{dj}{2}} 2^{-\frac{dk}{2}} f\left( \dist{ I_{\lambda} }{ I_{\mu}} \right) \\
	& = 2^{-(M+\frac{d}{2})\abs{j-k}} 2^{-j(M+d)} f\left( \dist{ I_{\lambda} }{ I_{\mu}} \right).
\end{align*}

The case $k < j$ gives
\begin{align*}
	\abs{\dotproduct{H \psi_{\lambda}}{\psi_{\mu}}} \lesssim 2^{-(M+\frac{d}{2})\abs{j-k}} 2^{-k(M+d)} f\left( \dist{ I_{\lambda} }{ I_{\mu}} \right),
\end{align*}
which allows to conclude that
\begin{align*}
	\abs{\dotproduct{H \psi_{\lambda}}{\psi_{\mu}}} & \lesssim 2^{-(M+\frac{d}{2})\abs{j-k}} 2^{-\min(j,k)(M+d)} f\left( \dist{ I_{\lambda} }{ I_{\mu}} \right),
\end{align*}

\end{proof}

%% file: Body/116-proof_thresh_d_dim.tex
\section{Proof of Theorem \ref{thm:proof_thresh}} \label{appendix:proof_thresh_d_dim}

Let us begin with some preliminary results. Recall that $\lambda = (j,m,e) \in \Lambda$ and $\mu = (k,n,e') \in \Lambda$. Since $f$ is compactly supported on $[0,\kappa]$ and bounded by $c_f$, we have $f_{\lambda,\mu} = f\left( \dist{I_{\lambda}}{I_{\mu}} \right) \leq c_f \indic_{\dist{I_{\lambda}}{I_{\mu}} \leq \kappa}$. By equation \eqref{eq:defdist}, 
$\dist{I_{\mu}}{I_{\lambda}} \leq \kappa$ if $\|2^{-j}m-2^{-k}n\|_\infty\leq R_{j,k}^{\kappa}$, where $R_{j,k}^{\kappa} = (2^{-j} + 2^{-k})c(M)/2 + \kappa$.

\begin{lemma}
Define 
\[
	\mathcal{G}_{j,k}^{e,e'} = \set{ (m,n) \in \mathcal{T}_j \times \mathcal{T}_k \; | \indic_{\dist{I_{\lambda}}{I_{\mu}} \leq \kappa} = 1}.
\]
Then $\left|\mathcal{G}_{j,k}^{e,e'}\right| \leq (2^{j} 2^{k+1} R_{j,k}^{\kappa})^d$.
\end{lemma}

\begin{proof}
First note that
\[
	\mathcal{G}_{j,k}^{e,e'} = \set{ (m,n) \in \mathcal{T}_j \times \mathcal{T}_k | \abs{2^{-j} m_i - 2^{-k} n_i} \leq R_{j,k}^{\kappa}, \quad \forall i \in \set{1, \ldots, d} }.
\]
Now, define $\mathcal{G}_{j,k,m}^{e,e'} = \set{ n \in \mathcal{T}_k \; | (m,n) \in \mathcal{G}_{j,k}^{e,e'}}$.
For a fixed $(j,k,m,e,e')$ the set $\mathcal{G}_{j,k,m}^{e,e'}$ is a discrete hyper-cube of sidelength bounded above by $2^{k+1} R_{j,k}^{\kappa}$.
Therefore $\left|\mathcal{G}_{j,k,m}^{e,e'}\right| \leq (2^{k+1} R_{j,k}^{\kappa})^d$ coefficients. Moreover, $\left| \mathcal{T}_j\right| = 2^{jd}$, hence the number of coefficients in $\mathcal{G}_{j,k}^{e,e'}$ is bounded above by $(2^{j} 2^{k} R_{j,k}^{\kappa})^d$.
\end{proof}

\bigskip 
\begin{proof}[Proof of i)]
We denote $J_{\max} = \log_2(N) / d$ the highest scale of decomposition.
First note that a sufficient condition for $2^{-\min(j,k)(M+d)} f_{\lambda, \mu}\leq \eta$ is that $\min(j,k) \geq J(\eta)$ with
$J(\eta) = \frac{-\log_2 (\eta/c_f)}{M+d}$. In the following, we let $\wtilde J(\eta)=\min(J(\eta),J_{\max})$ and define
\begin{equation*}
\mathcal{G} = \bigcup_{ \min(j,k) < J(\eta) } \bigcup_{e,e' \in \set{0,1}^d \backslash \{0\}} \mathcal{G}_{j,k}^{e,e'}.
\end{equation*}

The overall number of non zero coefficients $| \mathcal{G}|$ in $\bTheta_\eta$ satisfies
\begin{align*}
\# \mathcal{G} &= \sum_{j=0}^{J_{\max} - 1} \sum_{k=0}^{J_{\max} - 1} \sum_{e,e' \in \set{0,1}^d} \# \mathcal{G}^{e,e'}_{j,k} \indic_{ \min(j,k) < J(\eta) }  \\
& \lesssim (2^d-1)^2 \sum_{j=0}^{J_{\max}-1} \sum_{k=0}^{J_{\max}-1} \indic_{ \min(j,k) < J(\eta) } 2^{jd} 2^{kd} \left( \frac{c(M)}{2}(2^{-j} + 2^{-k}) + \kappa\right)^d \\
& \lesssim \sum_{j=0}^{J_{\max} - 1} \sum_{k=0}^{J_{\max} - 1} \indic_{ \min(j,k) < J(\eta) } 2^{jd} 2^{kd} \left(\frac{c(M)^d}{2^d}2^{-dj} + \frac{c(M)^d}{2^d}2^{-dk} + \kappa^d\right) \\
& \lesssim  \sum_{j=0}^{J_{\max} - 1} \sum_{k=0}^{J_{\max} - 1} \indic_{ \min(j,k) < J(\eta) } 2^{kd}  + \sum_{j=0}^{J_{\max} - 1} \sum_{k=0}^{J_{\max} - 1} \indic_{ \min(j,k) < J(\eta) } 2^{jd}  \\
& \qquad +  \sum_{j=0}^{J_{\max} - 1} \sum_{k=0}^{J_{\max} - 1} \indic_{ \min(j,k) < J(\eta) } 2^{kd} 2^{jd} \kappa^d.
\end{align*}
The first sum yields
\begin{align*}
	& \sum_{j=0}^{J_{\max} - 1} \sum_{k=0}^{J_{\max} - 1} \indic_{ \min(j,k) < J(\eta) } 2^{kd} \\
	& = \left(\sum_{j=0}^{\wtilde J(\eta)-1} \sum_{k=j}^{J_{\max} - 1} 2^{kd} +  \sum_{k=0}^{\wtilde J(\eta)-1} 2^{kd} \sum_{j=k}^{J_{\max} - 1} 1 \right)\\
	& \lesssim \wtilde J(\eta)N + 2^{d\wtilde J(\eta)} \log_2(N) \lesssim \log_2(N) N.
\end{align*}
The second sum is handled similarly and the third sum gives
\begin{align*}
	& \sum_{j=0}^{J_{\max} - 1} \sum_{k=0}^{J_{\max} - 1} \indic_{ \min(j,k) < J(\eta) } 2^{kd} 2^{kd} \kappa^d \\
	& = \kappa^d \sum_{j=0}^{\wtilde J(\eta)-1} 2^{jd} \sum_{k=j}^{J_{\max} - 1} 2^{kd} +  \sum_{k=0}^{\wtilde J(\eta)-1} 2^{kd} \sum_{j=k}^{J_{\max} - 1} 2^{jd} \\
	& \lesssim \kappa^d N 2^{d\wtilde J(\eta)}.
\end{align*}
 
Overall $\left|\mathcal{G}\right| \lesssim  \log_2(N) N + \eta^{-\frac{d}{M+d}} N$.
For $\eta\leq \log_2(N)^{-(M+d)/d}$, the dominating terms are of kind $\eta^{-\frac{d}{M+d}}$, hence $|\mathcal{G}| \lesssim \eta^{-\frac{d}{M+d}} N \kappa^d$.
\end{proof}

\begin{proof}[Proof of ii)]
Since $\bPsi$ is an orthogonal wavelet transform
\begin{equation*}
\norm{\bH - \wtilde \bH_{\eta}}_{2 \rightarrow 2} = \norm{\bTheta - \bTheta_\eta}_{2 \rightarrow 2}. 
\end{equation*}
Let $\bDelta_\eta=\bTheta - \bTheta_\eta$. We will make use of the following version of Shur inequality 
\begin{equation}\label{eq:schur}
\|\bDelta_\eta\|_{2\to 2}^2\leq \|\bDelta_\eta\|_{1\to 1} \|\bDelta_\eta \|_{\infty\to \infty}.
\end{equation}
Since the upper-bound \eqref{eq:decay} is symmetric,
\begin{equation*}	
  \norm{\bDelta_\eta}_{\infty\to \infty} = \norm{\bDelta_\eta}_{1\to 1} = \max_{\lambda \in \Lambda}\sum_{\mu \in \Lambda} \abs{\Delta_{\lambda, \mu} }
\end{equation*}
By definition of $\bTheta_\eta$ we get that 
\begin{align*}
\sum_{\mu \in \Lambda} \abs{\Delta_{\lambda, \mu}} = \sum_{k =0}^{J_{\max}-1} \sum_{e' \in \{0,1\}^d \backslash \{0\}} \sum_{ n \in \mathcal{G}^{e,e'}_{j,k,m} } \abs{\theta_{\lambda, \mu}} \indic_{ \min(j,k) > J(\eta) } \\
 \lesssim \sum_{k =0}^{J_{\max}-1} \sum_{e' \in \{0,1\}^d\backslash \{0\}} \sum_{ n \in \mathcal{G}^{e,e'}_{j,k,m} } 2^{-(M+\frac{d}{2}) \abs{j-k}} 2^{-\min(j,k)(M+d)} \indic_{ \min(j,k) > J(\eta) }. \\
\end{align*}
Then
\begin{align*}
	\sum_{\mu \in \Lambda} \abs{\Delta_{\lambda, \mu}} & \lesssim \sum_{k =0}^{J_{\max}-1} 2^{-(M+\frac{d}{2}) \abs{j-k}} 2^{-\min(j,k)(M+d)}  \indic_{ \min(j,k) > J(\eta) } \left| \mathcal{G}^{e,e'}_{j,k} \right| \\
	& \lesssim \sum_{k=0}^{j-1} (2^{k} R_{j,k}^{\kappa})^d 2^{(k-j)(M+d/2)} 2^{-k(M+d)} \indic_{ k > J(\eta) }  \\
	& \quad \quad + \sum_{k=j}^{J_{\max}-1} (2^{k} R_{j,k}^{\kappa})^d 2^{(j-k)(M+d/2)} 2^{-j(M+d)} \indic_{ j > J(\eta) }.
\end{align*}
The first sum on $k < j$ is equal to
\begin{align*}
A_1 & = 2^{-jM} 2^{-jd/2} \sum_{k=0}^{j-1} (2^{k/2} R_{j,k}^{\kappa})^d \indic_{ k > J(\eta) } \\
&= 2^{-jM} 2^{-jd/2} \indic_{ j > J(\eta) } \sum_{k=J(\eta)}^{j-1} (2^{k/2} R_{j,k}^{\kappa})^d.
\end{align*} 

The second sum on $k \geq j$ is:
\begin{align*}
	A_2 = \indic_{ j > J(\eta) } 2^{-jd/2} \sum_{k=j}^{J_{\max}-1} (R_{j,k}^{\kappa})^d 2^{-k(M-d/2)}.
\end{align*}
Now, notice that $ \displaystyle (R_{j,k}^{\kappa})^d \lesssim 2^{-jd} + 2^{-kd} + \kappa^d $. Thus
\begin{equation*}
\begin{aligned}
	A_1 & \lesssim 2^{-jM} 2^{-jd/2} \indic_{ j > J(\eta) } \sum_{k=J(\eta)}^{j-1} \left(2^{dk/2} 2^{-jd} + 2^{-dk/2} + 2^{kd/2}\kappa^d \right) \\
	& \lesssim 2^{-jM} 2^{-jd/2} \indic_{ j > J(\eta) } \left( 2^{-jd} 2^{jd/2} + 2^{-\frac{d}{2}J(\eta)} + \kappa^d 2^{jd/2} \right) \\
	& = 2^{-jM} \indic_{ j > J(\eta) } \left( 2^{-jd}  + 2^{-\frac{d}{2}(J(\eta) + j)} + \kappa^d \right) .
\end{aligned}
\end{equation*}
And
\begin{align*}
	A_2 & \lesssim \indic_{ j > J(\eta) } 2^{-j d/2}  \sum_{k=j}^{J_{\max}-1} \left(2^{-jd} + 2^{-kd} + \kappa^d\right) 2^{-k(M-d/2)} \\
	& \lesssim  \indic_{ j > J(\eta) } 2^{-j d/2} \left( 2^{-jd}2^{-j(M-d/2)} + 2^{-j(M+d/2)} + \kappa^d 2^{-j(M-d/2)} \right) \\
	& \lesssim  \indic_{ j > J(\eta) } 2^{-jM} \left( 2^{-jd}  + \kappa^d \right).
\end{align*}
Hence
\begin{align*}
	 \sum_{\mu \in \Lambda} \abs{\Delta_{\lambda,\mu} } & \lesssim \indic_{ j > J(\eta) } 2^{-jM} \left( 2^{-jd} + \kappa^d + 2^{-\frac{d}{2}(J(\eta) + j)} \right).
\end{align*}
Therefore
\begin{align*}
	\norm{\mathbf{\Delta_\eta}}_{1\to 1} & \lesssim 2^{-J(\eta) M} \left( 2^{-J(\eta) d} + \kappa^d + 2^{-dJ(\eta)} \right) \\
	& \lesssim 2^{-J(\eta) M} \left( 2^{-J(\eta) d} + \kappa^d \right) \\
	& \lesssim  \eta + \kappa^d \eta^{\frac{M}{M+d}} \\
	& \lesssim \kappa^d \eta^{\frac{M}{M+d}} \quad \text{for small } \eta.
\end{align*}

Finally, we can see that there exists a constant $C_{M}$ independent of  $N$ such that
\begin{equation*}
\|\mathbf{\Delta_\eta}\|_{1\to 1}  \leq C_{M} \kappa^d \eta^{\frac{M}{M+d}} \quad  \textrm{and} \quad \|\mathbf{\Delta_\eta}\|_{\infty\to \infty} \leq C_{M} \kappa^d \eta^{\frac{M}{M+d}}.
\end{equation*}
It suffices to use inequality \eqref{eq:schur} to conclude.
\end{proof}
\begin{proof}[Proof of iii)] This is a direct consequence of point i) and ii). \end{proof}
